\documentclass[sn-mathphys-ay-modified]{sn-jnl}

\usepackage{float}
\usepackage[utf8]{inputenc}
\usepackage[T1]{fontenc}
\usepackage{graphicx}%
\usepackage{multirow}%
\usepackage{amsmath,amssymb,amsfonts}%
\DeclareUnicodeCharacter{2203}{\exists} 
\DeclareUnicodeCharacter{211D}{\mathbb{R}} 
\usepackage{amsthm}%
\usepackage{mathrsfs}%
\usepackage[title]{appendix}%
\usepackage{xcolor}%
\usepackage{textcomp}%
\usepackage{manyfoot}%
\usepackage{booktabs}%
\usepackage{algorithm}%
\usepackage{algorithmicx}%
\usepackage{algpseudocode}%
\usepackage{listings}%
\usepackage{tikz}
\usepackage{subfigure}
\usetikzlibrary{shapes.geometric, arrows}
\usepackage{rotating}
\usepackage{comment}
\usepackage{enumitem}
\usepackage{makecell}
\usetikzlibrary{shapes.geometric,arrows}
\newcommand{\lt}{\left}
\newcommand{\rt}{\right}
\usepackage{natbib}

\theoremstyle{thmstyleone}%
\newtheorem{theorem}{Theorem}
\newtheorem{proposition}[theorem]{Proposition}%

\newtheorem{suppproposition}[theorem]{Proposition F\hspace{-0.3em}}%

\newtheorem{lemma}[theorem]{Lemma}

\theoremstyle{thmstyletwo}%

\theoremstyle{thmstylethree}%

\newcommand{\yk}[1]{{\color{blue}#1}}

\begin{document}
\thispagestyle{empty}
\newgeometry{top=1in,bottom=1in,right=1.25in,left=1.25in}

\title[Article Title]{Efficient mathematical programming formulation and algorithmic framework for optimal camera placement}

\author*[1]{\fnm{Yash} \sur{Kumar}}\email{yashkumar1803@utexas.edu}

\author[1]{\fnm{Raghu} \sur{Bollapragada}}

\author[1]{\fnm{Benjamin} D. \sur{Leibowicz}}

\affil[1]{\orgdiv{Operations Research and Industrial Engineering, Walker Department of Mechanical Engineering}, \orgname{The University of Texas at Austin}, \orgaddress{\street{204 E Dean Keeton St}, \city{Austin}, \state{TX} \postcode{78712}, \country{USA}}}

\abstract{Optimal camera placement plays a crucial role in applications such as surveillance, environmental monitoring, and infrastructure inspection. Even highly abstracted versions of this problem are NP-hard due to the high-dimensional continuous domain of camera configurations (i.e., positions and orientations) and difficulties in efficiently and accurately calculating camera coverage. In this paper, we present a novel framework for optimal camera placement that uses integer programming and adaptive sampling strategies to maximize coverage, given a limited camera budget. We develop a modified maximum $k$-coverage formulation and two adaptive sampling strategies, Explore and Exploit (E\&E) and Target Uncovered Spaces (TUS), that iteratively add new camera configurations to the candidate set in order to improve the solution. E\&E focuses on local search around camera configurations chosen in previous iterations, whereas TUS focuses specifically on covering regions that were previously uncovered. We first conduct theoretical analysis to provide bounds on the probability of finding an optimal solution and expected sampling needs, while ensuring monotonic improvements in coverage. Then, we conduct a detailed numerical analysis over different environments. Results show that E\&E achieves coverage improvements of 3.3\textendash16.0\% over all baseline random sampling approaches, while maintaining manageable computational times. Meanwhile, TUS performs well in open environments and with tight camera budgets, achieving gains of 6.9\textendash9.1\% in such conditions. Compared to the baseline, our approach achieves similar coverage using only 30\textendash70\% of the sampling budget, demonstrating its computational efficiency. Through a case study, we obtain insights into optimal camera placement decisions for a typical indoor surveillance application.}

\keywords{Camera placement, Adaptive sampling, Simulation optimization, Set covering problem}

\newpage

\maketitle

\section*{Acknowledgments}
This work was supported by the Office of the Vice President for Research, Scholarship and Creative Endeavors at The University of Texas at Austin. The authors would like to thank Qixing Huang for helpful input on the study design and problem definition.

\thispagestyle{empty}

\newpage

\section{Introduction} \label{section:1-introduction}
\setcounter{page}{1}

The problem of optimizing camera placement in three-dimensional (3D) environments is important for various applications such as surveillance \citep{kritter_optimal_2019, bodor_optimal_2007}, environmental monitoring \citep{fuentes_method_2020, chaudhary_observing_2017, ali_optimal_2021}, and infrastructure inspection \citep{gai_using_2021, zarzycki_leveraging_2024, yang_computeraided_2018, khaloo_hierarchical_2015, kim_systematic_2019, bai_assessment_2024}. The objective is to determine the optimal positions and orientations of cameras to maximize the free space visible to at least one camera in the camera network, where a camera network is defined as a set of camera configurations which is feasible to the original problem. Even highly simplified approximations of this problem are NP-hard due to the high-dimensional continuous domain of camera configurations and difficulties in efficiently and accurately calculating camera coverage.

The costs incurred for deploying and maintaining camera systems are an important consideration. High-grade security cameras, especially those with advanced features, such as high resolution, night vision, and motion detection, are expensive. Additionally, installation including wiring and labor needs, further escalates overall costs. Moreover, ongoing video monitoring, video review, and network maintenance require continuous human or automated oversight, adding to the operational expenses. Furthermore, there may be use cases where the same camera needs to be repositioned to take multiple images, with time becoming a limiting factor. Given these constraints, optimizing camera placement becomes imperative to ensure that the maximum possible coverage is achieved using a limited budget of cameras. Efficient camera placement strategies can significantly enhance feasibility for widespread use, as they reduce initial capital expenditure and recurring cost components. Thus, research efforts focus on developing efficient algorithms that identify the optimal positioning of cameras to cover the largest space with the least expenditure.

We develop an efficient framework for camera placement that possesses certain key characteristics making it effective in real-world applications. Firstly, it is generalizable to a large set of environments, ensuring its applicability across different scenarios such as urban spaces, natural landscapes, or indoor spaces. Additionally, we ensure fast solution times allowing the algorithm to be implemented and operational with minimal delay. This is especially important in scenarios where quick deployment is crucial. Moreover, our algorithm's robustness to environmental changes allows for efficient redeployment in response to minor environmental shifts, enabling it to quickly improve upon its previous solutions rather than starting anew. Also, our algorithm is capable of traversing a large search space efficiently, exploring numerous potential camera configurations to identify the optimal placements. Finally, it can calibrate to high precision, ensuring that the system can fine-tune camera positions and angles to maximize coverage and minimize blind spots. With these characteristics, our algorithm provides a comprehensive and practical approach to the problem of optimal camera placement.

Our algorithm for solving the optimal camera placement problem comprises three main steps in a simulation optimization framework. First, the pre-processing step prepares the 3D polygonal mesh of the environment by adding a voxel grid format suitable for the optimization algorithm. Second, the camera placement algorithm utilizes an adaptive sampling algorithm that iteratively refines the current camera network by identifying configurations that best improve coverage, either by adjusting existing positions or sampling new ones not covered by previous iterations. There are two different types of sampling algorithms that we develop for this purpose. Finally, we utilize an optimization framework which finds the best positions based on the set of sampled positions. The algorithm iteratively updates the sampled positions, improving coverage continuously. There are two optimization frameworks to select the best camera configurations amongst the given samples to find the best camera network: (a) a greedy algorithm that incrementally selects the best camera configurations from the sample provided and (b) an integer program (IP) that selects the provably optimal subset from the sample set.

We introduce two novel adaptive sampling strategies: the Explore and Exploit (E\&E) strategy and the Target Uncovered Spaces (TUS) strategy. Both strategies aim to maximize coverage while balancing constraints such as camera budgets and partial visibility due to obstructions. In our results, the E\&E strategy demonstrates robust performance, consistently outperforming the random sampling (RS) baseline with coverage improvements ranging from 3.3\% to 16.0\%, while maintaining manageable computational times. Its blend of local and global search allows it to adapt to various room layouts and resource budgets effectively. The TUS strategy is effective in environments with unobstructed spaces with limited resource budgets, providing coverage gains between 6.9\% and 9.1\% in these scenarios. Theoretical analysis further provides bounds on the probability of finding an optimal solution and expected sampling needs, while allowing us to ensure monotonic improvements in coverage.

Case studies of real-world indoor surveillance applications reveal practical insights, showing that conventional strategies like placing cameras at room corners or along edges can be suboptimal. Instead, positioning cameras along the faces of walls and near ceilings can yield better coverage. With limited resources, central placements along corridors prove more effective, while high-resource scenarios benefit from targeting open spaces with peripheral vision directed at corridors and doorways. These findings refine our understanding of optimal camera placement, offering valuable guidance for practical implementations.

The remainder of this paper is organized as follows. Section \ref{section:2-literature} provides a review of the most relevant literature and clarifies our primary novel contributions. In Section \ref{section:3-problem-description}, the exact problem description and mathematical programming formulation are described. In Section \ref{section:4-solution-algorithms}, we describe the solution algorithms, i.e., the adaptive sampling strategies used by our algorithm. In Section \ref{section:5-theoretical-analysis}, we develop consistency results for the algorithm and prove monotonic improvements in the optimization framework. In Section \ref{section:6-case-study}, we develop a case study and obtain unique insights into how the two solution algorithms work on a large-scale, real-world problem. Section \ref{section:7-algo-perform} presents results on algorithm performance and then Section \ref{section:8-conclusions} summarizes conclusions, including valuable directions for related future work. Some of the algorithmic pseudocodes, and comprehensive notation for the various parts of the algorithmic framework, are described in the Appendix.

\section{Literature review}\label{section:2-literature}

This literature review examines the optimal camera placement problem, which aims to minimize the number of camera positions required to monitor complex spaces under practical constraints. The review begins by discussing the foundational concepts of the art gallery problem, similar to the optimal camera placement problem, including its theoretical underpinnings. Next, it introduces unique challenges pertaining to the camera placement problem, such as limited fields of view, range restrictions, and complex 3D environments, which significantly expand the problem’s complexity beyond simplistic and traditional 2D models. The review then dives into the set covering combinatorial framework which is highly related to the structure of the optimal camera placement problem. We discuss frameworks that efficiently solve set covering problems and utilize its variants to solve similar problems. The review then delves into computational strategies, presenting both visibility calculation techniques and methods for selecting optimal camera configurations from sampled positions. Subsequently, various algorithmic approaches are evaluated, including combinatorial optimization and heuristic-based methods, each addressing the computational challenges posed by real-world applications.  Finally, our approach is presented and compared with existing models, highlighting recent advancements and adaptive strategies that enhance coverage in dynamic, high-dimensional environments, setting the stage for further innovations in this field.

The optimal camera placement problem is similar to another well-established and widely studied problem, the art gallery problem \citep{chvatal_combinatorial_1975}, wherein the objective is to find the minimum number of stationary guards, each with panoptic visibility, capable of observing the whole art gallery. In a geometric setting, the objective of the art gallery problem corresponds to finding the minimum number of position vectors such that a line segment can be drawn from at least one of them to each point within the polygon's interior, ensuring that the entire line segment remains inside the polygon. Victor Klee introduced it and a bounded solution was provided by Chvatal in 1973 \citep{chvatal_combinatorial_1975, honsberger_mathematical_1976}. The problem was recently proven to be $\exists \mathbb{R}$-complete by Abrahamsen et al., which is a superset of NP \citep{abrahamsen_art_2022}. Although the literature for the 2D version of the problem is extensive, there have been fewer papers researching theoretical properties and exact algorithms \citep{bottino_towards_2009, ghosh_approximation_2010, kranakis_brief_nodate, marzal_three-dimensional_2012, nishizeki_lower_1979, grunbaum_polytopal_1975, orourke_art_1987} for the 3D version of the problem. The optimal camera placement problem generalizes the art gallery problem by incorporating additional constraints specific to cameras, such as limited fields of view (FOV) rather than assuming panoptic visibility. The objective is to position cameras so that their combined viewing frustums cover a designated space, which could be a room, open area, or complex environment. This problem often includes restrictions on camera resolution, depth of field (DOF), and the ability to track moving objects. Variants exist that further increase complexity, such as dynamic conditions like shifting lighting sources, redundant coverage requirements, and network protection constraints. Figure \ref{fig1:camera-terms} provides a brief explanation of the viewing frustum and other technical camera-related terms that will be utilized in our discussion. Theoretical discussions on the properties of camera placement problems are limited in the literature, with one such discussion found in \cite{cheng_time-optimal_2008}. This is mostly due to the lack of simplifying, valid assumptions. Thus, it is more common to use efficient heuristics to solve the optimal camera placement problem.

The space is typically modeled using a discrete approach due to the computational and modeling challenges associated with a purely continuous approach. The problem generally has two aspects: the first is visibility calculation for a given camera configuration, and the second is selecting a subset of camera configurations from a set of sampled camera configurations. Several algorithms exist to address the first aspect \citep{angella_optimal_2007, penha_coverage_2013, moller_fast_1997, murray_coverage_2007, yaagoubi_hybvor_2015}. Improvements in computational hardware and software have allowed fast visibility calculations to be performed with some approaches utilizing specialized software from gaming and 3D graphics to accelerate this process \citep{angella_optimal_2007, penha_coverage_2013}. Another focuses purely on algorithmic improvements, such as efficient ray-tracing algorithms \citep{moller_fast_1997}. In \citep{andersen_wireless_2009}, visibility coverage is calculated by randomly sampling free space points and checking their visibility from a camera configuration in the network. Two papers have used geographic information system (GIS) software to implement novel algorithms for visibility calculations \citep{murray_coverage_2007, yaagoubi_hybvor_2015}. Our solution employs a novel flood-filling approach for visibility calculation, along with software acceleration using Python's Numba. Typically visibility calculation is broken up into two parts: ensuring that a free space is within the field of view and range of a camera, and ensuring that no obstacle obstructs the free space. Obstacle checking is time-consuming. Our method utilizes the idea that the visibility polygon is one single connected component to reduce the calculations required for obstacle checking. To the best of our knowledge, no previous research has utilized this method. Visibility checking has been highly optimized. Despite significant computational bottlenecks, research on efficient sample camera configuration selection is sparse. There are two main methods to solve this.

Selecting an optimal subset of camera configurations from sampled positions often aligns with the set-covering problem, a combinatorial and strongly NP-hard optimization problem \citep{balas_set-covering_1972, caprara_algorithms_2000}. Researchers have proposed various methods to solve the set-covering problem, including genetic algorithms \citep{solar_parallel_2002}, greedy heuristics \citep{alihodzic_exact_2020}, and Lagrangean-based approaches for sensor networks \citep{jarray_lagrangean-based_2013}. Methods differ in their emphasis on cost \citep{bautista_grasp_2007}, coverage \citep{costa_enhancing_2017}, or dynamic conditions like sensor movement \citep{chrissis_dynamic_1982}. Some specialized variants focus on correlated cost and coverage structures \citep{brusco_morphing_1999}, or penalty-based variants \citep{carrabs_solving_2024}. Recent work also considers redundancy in visual fields and its impact on optimal placement in continuous spaces \citep{costa_enhancing_2017}. For applications like environmental monitoring, the focus may shift to maximizing network lifetime \citep{castano_column_2014} and information transmission efficiency \citep{elloumi_optimization_2021, rebai_sensor_2015}. Some recent research has adapted these traditional frameworks for novel applications, like multistatic sonar networks \citep{thuillier_efficient_2024}.

Despite the set-covering literature being rich, few have utilized set-covering algorithmic frameworks to solve the optimal camera placement problem. Some of these will be discussed here. \cite{erdem_automated_2006} directly use a mixed-integer programming solver to find the optimal configurations. \cite{angella_optimal_2007} and \cite{alihodzic_exact_2020} employ a greedy heuristic to solve the set-covering part, while \cite{penha_coverage_2013} utilize an exact algorithm. \cite{rebai_exact_2016} address a bi-objective problem, incorporating both coverage maximization and a cost component in the objective. \cite{kenichi_yabuta_optimum_2008} consider a variant of the set-covering model, the maximum $k$-coverage model, and develop a heuristic that adds a new camera if the marginal reward exceeds a certain threshold.

The second set of algorithms that researchers have developed to select a good subset of camera configurations is based on derivative-free optimization (DFO) methods, where a camera network is chosen as an initial solution and iteratively improved. In \cite{penha_coverage_2013}, the visibility space is modeled as a graph with nodes representing randomly sampled open spaces and undirected arcs representing cross-visibility between connected nodes. An initial solution of $k$ cameras is chosen, and a genetic algorithm is used to improve coverage at each iteration. In \cite{aissaoui_designing_2018}, a similar model is considered, with objective design specifically focusing on the image quality of human motion capture. The genetic algorithm is implemented directly on the camera network solution of $n$ cameras. Each member of the population constitutes an entire camera network, each with $n$ position and rotation vectors. Adaptive mutations and crossovers are iteratively implemented on the entire network vector to iteratively improve the solution using the `elitism' criteria. A similar model was implemented in \cite{indu_optimal_2009}, but focusing on a priority-based objective, which requires multiple coverage of priority points, and using the tournament selection criteria to iteratively improve the solution.

In \cite{morsly_particle_2012}, a new type of particle swarm optimization (PSO) algorithm, specifically binary particle swarm optimization (BPSO), is implemented for IPs. Here, a population of particles corresponds to solutions of the optimization problem, and the particles are updated iteratively to improve the configurations of the chosen camera network. Each particle update is akin to taking a step in the function space to improve the objective. The novelty of this model is that the position variables can only take values 0 or 1 based on particle velocity, which acts as a probability threshold. In \cite{fu_surveillance_2014}, the BPSO is refined by improving the information-sharing mechanism of \cite{morsly_particle_2012}. Recently, \cite{wang_solving_2020} use a different type of PSO algorithm, based on Latin hypercube-based resampling.

Our method is based on the framework of simulation optimization, and instead of utilizing DFO for choosing the next best solution, we utilize DFO to develop the candidate set of solutions and then use it as input to a set covering variant optimization model that uses a standard optimization software. Firstly, our model improves over those in \cite{wang_solving_2020} because it considers 3D coverage rather than a 2D approximation. These ignore blind spots that typically occur while considering 2D coverage approximations \citep{zhang_optimized_2013}. A similar 3D environment framework is considered in \cite{rebai_exact_2016}, but their model is not focused on a global search and selecting efficient camera positions through sampling strategies, unlike ours. It is more focused on developing a Pareto frontier for optimizing the dual objectives of maximizing coverage and minimizing cost, while considering a small subset of restricted camera configurations. 

\cite{penha_coverage_2013} consider large 3D environment frameworks, however they only consider omnidirectional light sources to solve the art gallery problem, which reduces the dimensionality of sample camera configurations and does not consider actual restricted-field cameras. \cite{morsly_particle_2012} consider a few 3D environments along with efficient visibility calculations, but their environments are simplistic and their framework focuses on complete coverage rather than minimizing costs. Also, the camera pitch angle is fixed at -90$^\circ$ which restricts all cameras to point downwards. Furthermore, real 3D models are exceedingly complex and single uncovered blocks which cannot be covered would render the formulation infeasible, especially if there are obstacles blocking the free space below. 

Our problem framework is similar to what \citet{sun_learning_2021} consider, where they focus on complex 3D environments with few restrictions on camera placements, but our model allows for different levels of discretization whereas their model is restricted by the size of the neural network input. Furthermore, our model's visibility calculation function is exact whereas their model relies on training their neural networks on sufficiently large datasets, only to achieve an approximation of the visibility calculation. Furthermore, we develop adaptive sampling strategies under resource constraints which are a novel approach, not previously utilized to solve optimal camera placement problems. Adaptive sampling strategies have been discussed in other applications, e.g., under partial information availability in data quality control systems, as shown by \cite{liu_adaptive_2015}, \cite{nabhan_correlation-based_2021}, and \cite{zan_spatial_2023}. Our sampling strategy, focusing on generating camera configurations, is discrete in position and continuous from an angle perspective. This algorithmic structure allows us to use the inherent structure and speed that IPs provide along with adaptive sampling methods to find good candidates for the optimal camera network. There are two different adaptive sampling strategies that we consider for choosing configurations iteratively. The first is based on the local sampling of camera configurations which were found to be optimal in the previous iteration. The idea is to improve local coverage while sampling over the entire space. The second adaptive sampling strategy is to sample camera configurations that help in covering blind spots. This is done by choosing camera configurations that specifically cover subsets of the open space that were largely uncovered under the previous optimal network configuration.

\section{Problem description and formulation} \label{section:3-problem-description}

In this section, we will provide the problem description and the mathematical formulation modeling the problem. We begin by defining terminology, key variables, and constraints, followed by a detailed explanation of the objective function that characterizes the problem.

\subsection{Problem outline}

We will first define terminology. In our problem statement, the term \emph{coverage space} refers to the physical 3D environment that needs to be observed by the camera network. A \emph{camera network} refers to a set of camera configurations which is feasible, i.e., it satisfies all the constraints. A particular \emph{camera configuration} represents a location in 3D space, represented by a position vector $p \in \mathbb{R}^3$, and the direction of the camera lens, represented by a direction vector $d \in \mathbb{R}^3$. This direction vector is directed from the position $p$ to the center point of the camera focus. These two vectors can categorize the position, pitch, and yaw of the camera. We assume the roll angle of the camera to be zero. However, this is an easy extension to include, and requires an additional degree of freedom. In practical applications, a setup with a non-zero roll angle is unnecessary. Brief descriptions of the \emph{roll}, \emph{pitch}, and \emph{yaw} can be found in Figure \ref{fig1a:roll-pitch-yaw} and its caption.

\begin{figure}[H]
    \centering
    \subfigure[]{\includegraphics[width=0.4\textwidth]{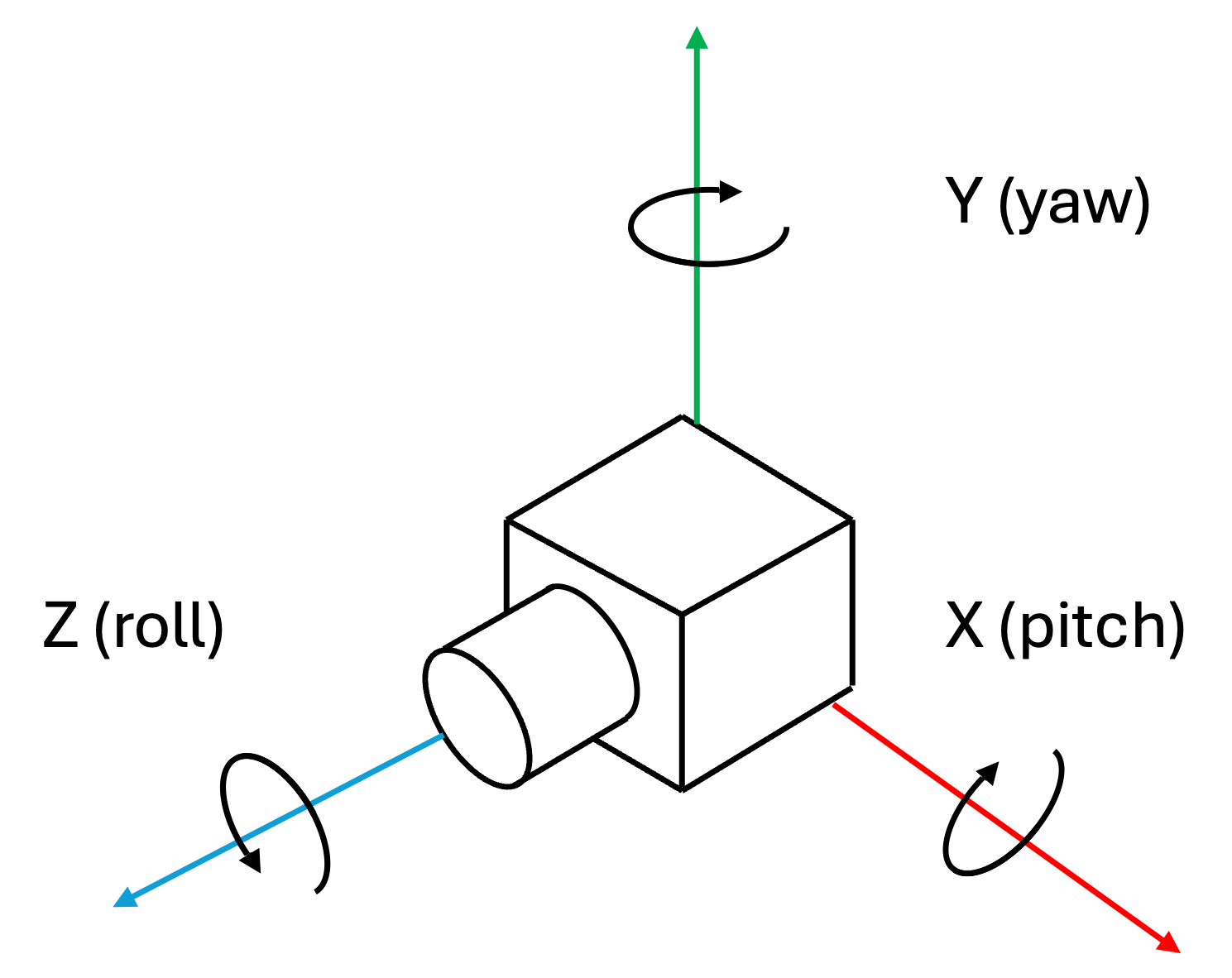} \label{fig1a:roll-pitch-yaw}}
    \hspace{1em}
    \subfigure[]{\includegraphics[width=0.4\textwidth]{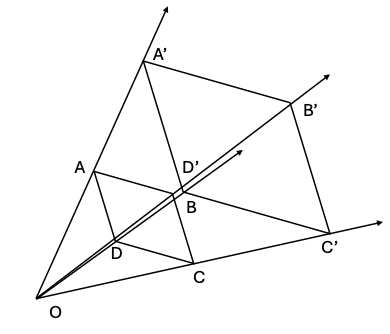} \label{fig1b:viewing-frustum}}
    \caption{Part (a) provides a description of the rotational axes pertaining to roll, pitch, and yaw. Roll refers to the rotation around the camera's forward axis. It changes the ``tilt" of the image, making the horizon appear slanted. Pitch is the rotation around the camera's horizontal axis. Yaw refers to the rotation around the camera's vertical axis. Part (b) describes the viewing frustum. Here, $O$ refers to the camera viewpoint. The planes that cut the frustum perpendicular to the viewing direction are called the near plane $(ABCD)$ and the far plane $(A'B'C'D')$. The 3D space encompassed within the frustum $(ABCDD'A'B'C')$ represents the viewing frustum. Objects closer to the camera than the near plane or beyond the far plane are not visible. The horizontal field of view is the angular measure subtended by center of lines $AD$ and $BC$ on point $O$, whereas the vertical field of view is the angular measure subtended by center of lines $AB$ and $CD$ on point $O$. The distance between the center of $(ABCD)$ and $(A'B'C'D')$ is known as the depth of field. The unit normalized vector connecting the point $O$ to the center of plane $(ABCD)$ is the view direction vector, where the tail of the direction vector is the camera position and the head is the point the camera is focused on}
    \label{fig1:camera-terms}
\end{figure}

As we are utilizing a discrete model, the space is segmented into a spatial grid made of cubes. These cubes are considered free space \emph{voxels} (3D pixels) if there is no object occupying them. If any object occupies the space, these are considered closed voxels and do not need to be observed. Maximizing coverage refers to maximizing the number of free space voxels observable by at least one camera. Within the model context, the coverage can be calculated as the total number of free space voxels that are observable in the free space by the camera network. Also, observability refers to the center of the free space voxel being visible, unobstructed by any other obstacles, within the viewing frustum of the camera. Sufficient granularity is maintained so that the non-coverage of a voxel corresponds to almost or partial coverage of the voxel by an obstacle, wall, or boundary. The standard algorithm structure is defined as follows: the scene data is imported and a voxel grid of appropriate size is constructed, then sample camera configurations are selected based on a specific sampling strategy, and the data are utilized to construct an IP.

\subsection{Mathematical program} \label{section:3.2-IP}

We formulate the optimal camera placement problem based on the maximum $k$-coverage model \citep{khuller_budgeted_1999, hochbaum_analysis_1998}, which is a variant of the popular set-covering problem from the combinatorics literature. The problem takes several sets, and a number $k$, as input. These sets typically have elements in common. Solving the problem requires taking at most $k$ of these sets such that the maximum number of elements are covered within the chosen set, i.e. the union of the selected sets has maximal size. This structure can be modified for our problem statement, as each camera configuration can be mapped to a set, and each voxel represents an element that may or may not be covered by the set.

\subsubsection{Model formulation} \label{section:3.2.1-model-formulation}

We will firstly define the notation for the mathematical program. The set of all positions $P$ is a subset of all free space voxels $V$ that need to be covered. Here, $p\in P$ refers to a camera position index (3D vector) and $d$ refers to the view direction vector described in Figure \ref{fig1b:viewing-frustum}, or just direction vector, where $d \in D$, with $D$ theoretically representing an uncountably infinite set as there are no restrictions on directions. It is normalized to a unit value. $(p,d)$ together constitute a camera configuration index. $v\in V$ represents a free space voxel index. $PD$ refers to the set of all possible camera configurations. As the sets $P$ and $D$ are independent, $PD = |P \times D|$, the Cartesian product of the two sets. Further, let $PD' \subseteq PD$ represent the set of sampled camera configurations. $V_{pd}$ represents the set of all free space voxels visible under camera configuration $(p,d)$. It is generated as the output of Algorithm \ref{alg:visibility-calc}, which conducts visibility calculations. $PD'_v$ is the subset of all camera configurations $PD'$ which can view free space voxel $v$. $PD^{\text{adj}}_p$ is the set of all camera configurations which are in the immediate neighborhood of a camera located at position $p$. As cameras have finite space, there can only be one camera per neighborhood. $x_{pd}$ is a binary variable, which is 1 iff the camera at configuration $(p,d)$ is part of the selected solution at that iteration. $y_v$ is a binary variable that is 1 iff the selected camera network covers voxel $v$.

Now, the objective of the program is to maximize the number of visible voxels, or covered voxels. Visibility of a voxel is determined by whether one of the cameras in the selected camera network has a clear line of sight to that particular voxel. This is subject to the constraint that the chosen subset of all camera configurations should be such that the sum of the cost $f_{pd}$ corresponding to the chosen camera sample at $(p,d)$ does not exceed total resource budget $\beta$. Note that the resource budget can represent the maximum number of cameras possible in the camera network. If that is the case, $f_{pd}$ reduces to 1, referring to a single camera. Another variant of the resource budget could account for varying costs of each sampled camera configuration, influenced by factors such as wiring or other extraneous expenses, with the budget representing the upper limit of these costs. In addition to the modified resource budget constraint, the model has an additional positional constraint, requiring that not more than one camera be placed at any location. The overall framework is very similar to a column-based formulation in IPs \citep{barnhart_branch-and-price_1998, wilhelm_technical_2001}. Typically, a row-based formulation for this problem is infeasible as the visibility constraints are nonlinear and continuous.

The model below is based on having mapped the entire sample of all camera configurations to the corresponding voxels visible from it, and encoding all sampled configurations as decision variables to construct an IP model. This is in the form of a modified maximum $k$-coverage problem. We can parameterize the problem by $PD'$, $V$, and $\beta$ and $\{f_{pd} |f_{pd}: (p, d) \in PD'\}$. This is equivalent to solving a reduced master problem under a column generation framework.
\[
\begin{array}{lllll}
    IP(PD', V, \beta, \{f_{pd}\}) = & \max & \sum_{v \in V} y_v  && (1) \\
    & \text{s.t.} &&& \\
    & \text{Camera cost:} & \sum_{(p,d) \in PD'} f_{pd} x_{pd} \le \beta && (2) \\
    & \text{Voxel coverage:} & y_v - \sum_{(p,d) \in PD'_v} x_{pd} \le 0 &\quad \forall v \in V & (3) \\
    & \text{One camera per locale:} & \sum_{(p, d) \in PD'^{\text{adj}}_p} x_{pd} \le 1 &\quad \forall p \in P & (4) \\
    & \text{Integrality:} & x_{pd}, y_{v} \in \{0, 1\} && (5)
\end{array}
\]

Under this practical framework, we assume that, as it is infeasible to consider all possible camera configurations explicitly, we only solve the program with a subset of variables, i.e. the set $PD'$ represents a sampled set of camera configurations, instead of the entire space of camera configurations ($PD' \subseteq PD$). Constraint (3) requires that we map every voxel $v$ to the subset of selected camera configurations $PD'_v$ which can cover it. We develop an efficient algorithm to solve the inverse problem, i.e., calculate the voxels covered by each of the selected configuration samples and then invert the function to get the mapping $PD'_v$. Even so, it is computationally prohibitive to perform these calculations for all possible camera configurations $PD$. So, we conduct adaptive sampling to obtain camera configurations and improve coverage efficiently. This is equivalent to solving a subproblem and generating new columns. The main differences are that:
\begin{enumerate}
    \item The subproblem is not an IP and is a heuristic because of the nature of the problem, which admits infinite number of columns.
    \item The new columns thus generated correspond only to constraints and do not augment the objective function directly, which typically happens in column generation. The objective is to maximize coverage and does not directly involve camera configurations, which represent the columns being generated.
    \item Generating columns and updating the master problem requires using a visibility algorithm to create the mapping between the new sampled configurations and the visible voxels.
\end{enumerate}

We will proceed to explain the adaptive sampling strategies in the next section.

\section{Solution algorithms} \label{section:4-solution-algorithms}

In this section, we will expand on how the various camera configurations are obtained and form the input for the IP. In our framework, the term camera configuration samples refers to all the sampled camera configurations, sampled based on a strategy, and then utilized in the optimization algorithm. The solution of the IP is post-processed to extract the optimal solutions from the IP and the exact coverage obtained. This process is repeated with new camera configurations being sampled and added to the original set of samples for the optimization model. As we are continuously adding new variables to the IP, we can guarantee that our solution improves monotonically. We will prove this statement in Section \ref{section:5-theoretical-analysis}. Furthermore, we use the solution from the previous iteration to warm-start the optimization algorithm in the next iteration. The total number of samples that are considered over the period of the whole algorithm, from which the best camera network is chosen, is referred to as the \emph{camera sampling budget}. The flowchart of the algorithm is presented in Figure \ref{fig2:optimization-flowchart}.

\tikzstyle{terminator} = [rectangle, draw, text centered, rounded corners, very thick, minimum height=2em]
\tikzstyle{process} = [rectangle, draw, text centered, very thick, minimum height=2em]
\tikzstyle{decision} = [diamond, 
minimum width=0.8cm,   
minimum height=0.4cm,  
text centered, 
draw=black, very thick, text width=3em]  
\tikzstyle{data}=[trapezium, draw, text centered, trapezium left angle=85, trapezium right angle=95, very thick, minimum height=1.5em]
\tikzstyle{connector} = [draw, -latex']
\tikzstyle{block} = [rectangle, draw, text width=5em, text centered, very thick, minimum height=4em]
\tikzstyle{line} = [very thick,->,>=stealth] 

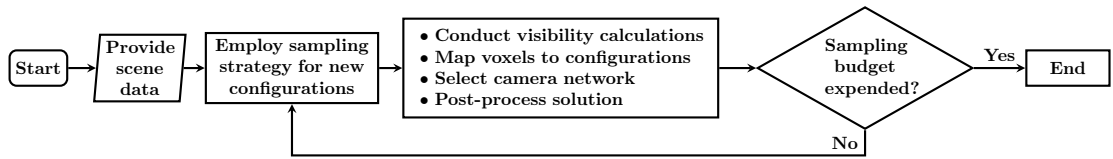
\begin{figure}[h!]
\centering
\resizebox{14.5cm}{!}{%
\begin{tikzpicture}[node distance = 2.5cm]  
\node [terminator] (start) {\textbf{Start}};
\node [data, right of=start, text width=3.8em, xshift=-0.5cm] (data) {\textbf{Provide scene data}};
\node [process, right of=data, text width=9em, xshift=0.5cm] (sample) {\textbf{Employ sampling strategy for new configurations}};
\node [process, right of=sample, text width=17em, xshift=2.8cm] (solve) { 
    \textbf{
    \begin{tabular}{l}
    $\bullet$ Conduct visibility calculations \\
    $\bullet$ Map voxels to configurations
    \\ 
    $\bullet$ Select camera network \\
    $\bullet$ Post-process solution
    \end{tabular}
    }
};
\node [decision, right of=solve, yshift=0cm, aspect=1.75, text width=4.5em, xshift=3.5cm] (solvemore) {\textbf{Sampling budget expended?}}; 
\node [process, right of=solvemore, text width=4em, xshift=1.5cm, yshift=-0cm] (end) {\textbf{End}}; 

    \draw [line] (start) -- (data);
    \draw [line] (data) -- (sample);
    \draw [line] (sample) -- (solve);
    \draw [line] (solve) -- (solvemore);
    \draw [line] (solvemore) -- node [midway,above] {\textbf{Yes}} (end);

    \draw [line] (solvemore.south) -- ++(0,-.5) node[midway, left]{\textbf{No}} -- ++ (-11.3,0) -- (sample.south);
\end{tikzpicture}}
\caption{Flowchart of the optimization process. It describes an iterative process where the environmental data are processed, and iteratively more sample camera configurations are added to improve the solution.}
\label{fig2:optimization-flowchart}
\end{figure}

\subsection{Preliminaries: visibility calculations}

The model requires mapping every camera configuration with the free space voxels that it covers. This is an implicit requirement of the IP model in Section \ref{section:3.2-IP}. We aim to achieve this by utilizing a fast visibility algorithm. The purpose of this is to mimic camera vision and calculate which voxels are unobstructed and observable from a particular camera configuration. Thus, given a configuration, subject to field of view and depth of field restrictions, it allows us to compute the free space voxels that are observable by the camera. A free space voxel is said to be observable if three conditions are satisfied:

\begin{enumerate}[label=(V\arabic*)]
    \item It is within the depth of field of the camera configuration. \label{4.1-vis-condition-1}
    \item It is within the horizontal and vertical field of view of the camera configuration. \label{4.1-vis-condition-2}
    \item It is not blocked by any obstacles. \label{4.1-vis-condition-3}
\end{enumerate}

We develop and employ a novel flood-filling approach \citep{levoy_area_1981} to determine which of the voxels, $r^v \in r^V$ (representing position vectors of free space voxels in the environment), are present in the viewing frustum of a specified camera configuration. The procedure starts with preliminary calculations by computing normalized direction vectors $\hat{d}_{wn}$ and $\hat{d}_{hn}$ which respectively define the vectors normal to the horizontal plane and vertical plane in the field of view of the camera configuration's viewing frustum relative to its position $r^c_p$. These vectors are derived from corner points $r^c_{w1}, r^c_{w2}$ for the horizontal axis and $r^c_{h1}, r^c_{h2}$ for the vertical axis.  

After these preliminary calculations, we employ flood-filling: the main idea here is that we start from all voxels which are at a minimum depth of field distance from the camera and check connected voxels to visible ones to find out which of them are present in the viewing frustum. If any voxel $r^v$ is within the viewing frustum and is a free space voxel, then the voxels around it also have a possibility of being visible. Voxels that are a unit distance away from current voxel $r^v$ are then added to the visibility stack $r^{V_{\text{stack}}}$. This way, we only stack and consider the voxels that are nearby a visible voxel. As the visibility set is a connected set, this allows us to omit large sets of voxels which are behind the camera viewing direction. Also, this further allows us to omit voxels from consideration which completely block the visibility frustum with an obstacle. Checking for obstacle obstruction is the most computationally intensive aspect of visibility calculations. The flood-filling approach allows us to bypass all voxels which are within the FOV, but are segmented completely and behind obstructions.

Then, we must assess whether each voxel in $r^{V_{\text{stack}}}$ satisfies conditions \ref{4.1-vis-condition-1} and \ref{4.1-vis-condition-2}, i.e., whether the FOV and DOF requirements are satisfied. To ensure this, we calculate the direction vector $\hat{d}^c_v$, which is the vector from the camera location to each voxel location $r^v$. We decompose this vector into two components, one parallel to $\hat{d}_{wn}$ and the other vector, $\hat{d}^c_{vw}$, parallel to the horizontal field of view plane. This process is repeated for the vertical field of view plane, providing $\hat{d}^c_{vh}$. Then, we check whether the vector $\hat{d}^c_{vw}$ and $\hat{d}^c_{vh}$ is within the non-reflex angle created by the horizontal and vertical FOV plane. This is done assuming the camera field of view forms an angle less than $180^\circ$. If both of these statements are true, then the voxel is within the field of view of the camera configuration. We repeat this procedure until the stack is empty. This way we generate a set of voxels $r^{\tilde{V}}$ which are connected, within the viewing frustum, and within the set of free space voxels.

Finally, we assess whether the subset of free space voxels, which have been checked for FOV and DOF restrictions, are not obstructed by objects in the environment \ref{4.1-vis-condition-3}. Within the set of voxels, $r^v \in r^{\tilde{V}}$, we calculate the first intersection point of the ray with vertex $r^c_p$ and direction vector $r^v - r^c_p$ with the environment mesh object which is a set of triangles, providing us with position vector $r^c_{Fv}$. This is done natively using the pre-implemented M\"{o}ller-Trumbore algorithm \citep{moller_fast_1997} in Python's Open3D package.\footnote{This algorithm is referred to as the Ray-Intersection minor function in Algorithm \ref{alg:visibility-calc} and is equivalent to the \texttt{\detokenize{o3d.t.geometry.RaycastingScene().cast_rays()}} method in Open3D.} This provides us with the first intersection point of a camera trying to view voxel $r^v$. If the length of vector $\hat{d}^c_v$ is less than the maximum depth of field and length of $r^c_{Fv} - r^c_p$, that means the voxel $r^v$ is unobstructed and near enough from the camera. The pseudocode for this algorithm is given in Algorithm \ref{alg:visibility-calc}. The notation for the model is provided in Appendix \ref{appB:vis-notation}.

In the following section, we will describe and compare the two proposed adaptive sampling strategies. The idea behind adaptive sampling of camera configurations is to use information obtained about the environment and the solution network from previous iterations to improve the selection of new camera configurations. This allows the solution to reach a better coverage and with a lower sampling budget. In both strategies that we discuss here, there are two types of camera configurations sampled: 1) informed configurations, that utilize knowledge of the environment and/or previous camera network solutions; and 2) random configurations, that sample positions and directions at random.

\subsection{Adaptive sampling strategy: Explore and Exploit (E\&E)}

This sampling strategy is designed to efficiently sample camera configurations by balancing the exploration of new potential configurations and the exploitation of existing ones, i.e. it balances the need for exploring new configurations and refining existing solutions to ensure comprehensive coverage of the search space while achieving optimal camera placements. Prior to this, techniques employing adaptive sampling frameworks that balance the trade-off between exploration and exploitation have been applied to a few domains, such as bandit problems \citep{lu_daisee_2023} and robotics-related challenges \citep{munir_analysis_2021}. We propose a similar targeted strategy applied to the camera placement domain.

The algorithm begins by determining the number of camera configurations to be sampled, denoted as \( N_{\text{tot}} \). This is then expressed as the sum of two components: \( N_{\text{explore}} \), representing the number of exploration samples, and \( N_{\text{exploit}} \), representing the number of exploitation samples. This partitioning is conducted using $f_{\text{exploit}}$, which denotes the fraction of all camera configurations in an iteration that are to be sampled during the exploit phase. Correspondingly, $(1-f_{\text{exploit}})$ represents the fraction to be sampled during the explore phase.

The exploration phase aims to diversify the search space by considering random configurations, exploring unexplored regions, or supporting other camera configurations to achieve the best coverage. In this phase, the algorithm samples \( N_{\text{explore}} \) exploratory camera configurations, i.e. set of randomly sampled position-direction pairs, \( PD_{\text{explore}} \), from the position set \( P \) and direction set of unit norm vectors \( D \). $N_{\text{pos-explore}}$ represents the number of positions sampled per iteration in the exploration phase, whereas $N_{\text{dir-pos}}$ represents the number of directions sampled per position. This allows for a many to one mapping, i.e. a single position is attached to $N_{\text{dir-pos}}$ directions to provide $N_{\text{dir-pos}}$ configurations per sampled position. Thus, the number of configurations to be sampled in the exploration phase is $N_{\text{dir-pos}} \cdot N_{\text{pos-explore}} \approx N_{\text{explore}}$.
 
This sample is generated using the Sample-Random-Configurations auxiliary function. The function definition is described in Appendix \label{app:random_config}.

For the exploitation phase, we need to calculate the number of exploitative configurations (informed configurations) that can be sampled, while maintaining the sampling budget designated. For this reason, we calculate \( N_{\text{config-sol}} \), which is derived by taking the ratio of the number of configurations set aside for exploitation, \( N_{\text{exploit}} \), to the number of camera configurations chosen in the optimal solution of the previous iteration, $N_{\text{sol}}$. Then a mapping of ordinals to directions, \( \mathcal{D}_z \), is sampled within a specified jitter range, \( \theta_{\text{jitter}} \), of the direction vector pointing to the positive $z$-axis allowing for slight variations in the original solution's camera orientation. 

The logic behind the exploitation phase is that the configurations included in the previous optimal solutions are likely close to a good solution, so more configurations are added to the sample set by choosing sample configurations that have a position, $p_\text{new}$, and direction, $d_\text{new}$, close to the previous optimal camera network's configurations. This is done by first rotating \( N_{\text{config-sol}} \) angles from \( \mathcal{D}_z \) in the same direction as rotating the unit z-axis vector, $z_0$, towards the previous optimal configuration $d$. Additionally, random voxel offsets, which are uniformly sampled from a multivariate uniform distribution ($\mathcal{U}_{\text{d}}$) with bounds $\{-v_{\text{jitter}}, v_{\text{jitter}}\}$, are applied to the solution positions, introducing variability in the camera placement while ensuring proximity to the original solution. This provide us with the set $PD_{\text{exploit}}$.

Finally, the algorithm combines the exploration and exploitation sets to form the final set of sampled configurations, \( PD_{\text{E\&E}} \), which is then returned.  The complete algorithm is provided in Algorithm \ref{alg:ee-strategy}, 
and the complete notation is given in Appendix \ref{appC:ee-notation}. The algorithm utilizes three auxiliary functions: Sample-Random-Configurations, Sample-Spherical-Cap and Rotate-Along, which are detailed in Appendix \ref{section:F-minor-functions}.

\begin{algorithm}
\caption{E\&E strategy pseudocode}
\label{alg:ee-strategy}
\begin{algorithmic}[h!]
\Procedure{Explore-and-Exploit}{$N_{\text{pos}}, N_{\text{dir-pos}}, f_{\text{exploit}}, P, PD_{\text{iter}}, \theta_{\text{jitter}}, v_{\text{jitter}}$}
    \State $N_{\text{tot}} \gets N_{\text{pos}} \cdot N_{\text{dir-pos}}$
    \State $N_{\text{explore}} \gets \lceil \textcolor{white}{\rceil} N_{\text{tot}} \cdot (1 - f_{\text{exploit}})
    \textcolor{white}{\lfloor} \rfloor$
    \Comment{For explore-phase samples}
    \State $N_{\text{pos-explore}} \gets \lceil \textcolor{white}{\rceil} N_{\text{explore}} / N_{\text{dir-pos}}
    \textcolor{white}{\lfloor} \rfloor$
    \State $PD_{\text{explore}} \gets \text{Sample-Random-Configurations}(N_{\text{pos-explore}}, P, N_{\text{dir-pos}})$
    
    \State $N_{\text{sol}} \gets |PD_{\text{iter}}|$         \Comment{For exploit-phase samples}
    \State $N_{\text{exploit}} \gets \lceil \textcolor{white}{\rceil} N_{\text{tot}} \cdot f_{\text{exploit}}
    \textcolor{white}{\lfloor} \rfloor$
    
    \State $N_{\text{config-sol}} \gets \lceil \textcolor{white}{\rceil} N_{\text{exploit}} / N_{\text{sol}} \textcolor{white}{\lfloor} \rfloor$
    \State $\mathcal{D}_z \gets \text{Sample-Spherical-Cap}(\theta_{\text{jitter}},  N_{\text{sol}} \cdot N_{\text{config-sol}})$
    \State $i \gets 0$
    \State $PD_{\text{exploit}} \gets \emptyset$
    
    \For{$(p,d) \in PD_{\text{iter}}$}
    
        \State $i \gets i+1$
        \For{num $= 1:N_{\text{config-sol}}$}
            \State $d_z \gets \mathcal{D}_z[\text{num} \cdot i]$
            \State $p_{\text{new}} \gets \emptyset$
            \While{$p_{\text{new}} = \emptyset$ or $p_{\text{new}} \notin P$ }
                \State $p_{\text{perturb}} \sim \mathcal{U}_{\text{d}}\{- v_{\text{jitter}},  v_{\text{jitter}}\}^3$
                \State $p_{\text{new}} = p + p_{\text{perturb}}$
            \EndWhile
            \State $d_{\text{new}} \gets \text{Rotate-Along}
            (z_0, d_z, d)$
            \State $PD_{\text{exploit}} \gets PD_{\text{exploit}} \cup \{(p_{\text{new}}, d_{\text{new}})\}$
        \EndFor
    \EndFor
    \State $PD_{\text{E\&E}} \gets PD_{\text{exploit}} \cup PD_{\text{explore}}$ 
    \State \Return $PD_{\text{E\&E}}$
    \EndProcedure
\end{algorithmic}
\end{algorithm}

\subsubsection{Illustrative example of E\&E strategy}

Figure \ref{fig3a:ee-illustration} illustrates a typical iteration of the E\&E strategy as visualized in an individual room as part of a larger apartment, which forms the environment. This environment will be discussed in more detail in Section \ref{section:6-case-study}. The legend for the individual cubic components is described in Figure \ref{fig3b:legend-box}.

In the first iteration, there are only randomly sampled configurations in that section, which correspond to the exploration phase. The visibility calculations are conducted and one of the camera configurations is made part of the solution camera network. In the next iteration, additional configurations are sampled which have positions and directions near to the one chosen in the previous iteration. This improves the coverage in iteration 2. This process is repeated for a total of 10 iterations in this case. In the end, randomly chosen configurations (explore phase) along with informed configurations (exploit phase) ensure a good coverage overall.

\begin{figure}[h!]
    \centering
    \subfigure[]{%
        \includegraphics[width=\textwidth]{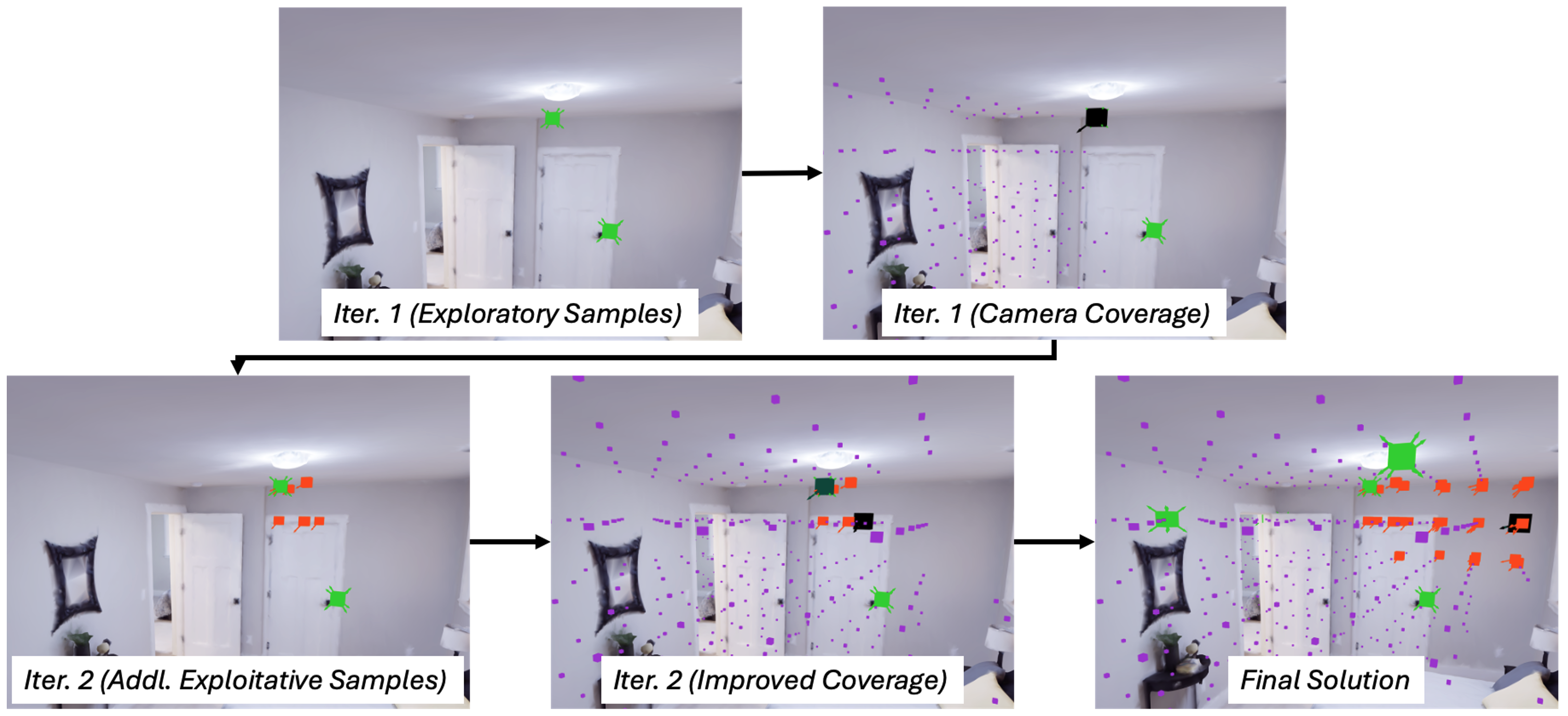}%
        \label{fig3a:ee-illustration}%
    }\\[1em]
    \subfigure[]{%
        \includegraphics[width=0.8\textwidth]{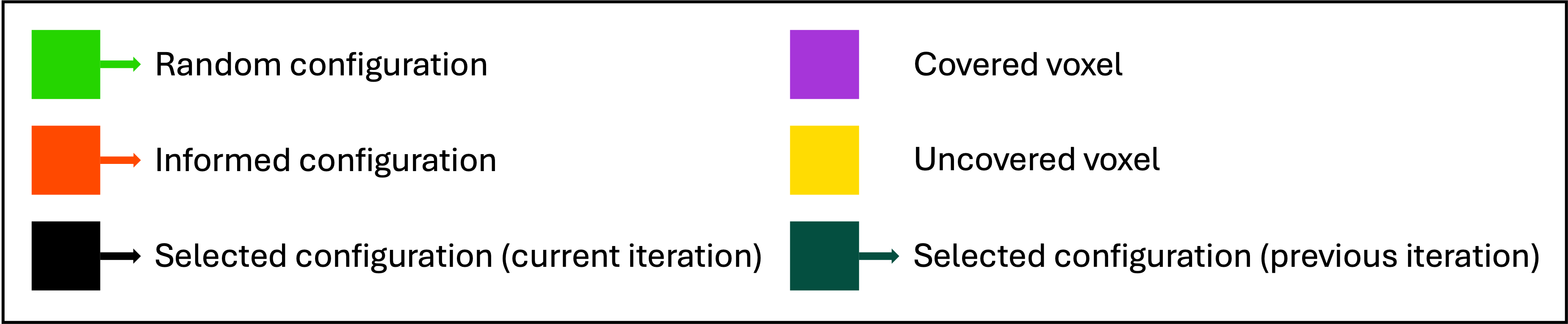}%
        \label{fig3b:legend-box}%
    }
    \caption{Part (a) (top) illustrates the iterative approach of the E\&E strategy in a portion of the environment. In the first iteration (abbreviated as iter.), two voxel positions are sampled in the section, with eight directions per position (green cubes and arrows). One of the configurations is selected as the network solution, and visible free space is shown in dark pink. Subsequent iterations include exploratory (green) and exploitative (orange) configurations. The final solution optimizes camera placements for maximum coverage. Part (b) (bottom) provides descriptions of components in different iterations of the sampling strategy. The position of the boxes corresponds to the position of the corresponding free space voxel or position vector. If the box has an arrow, the arrow direction refers to the view direction vector of the configuration. Multiple arrows may indicate that there are multiple view directions \( d \) sampled for that particular position \( p \)}
    \label{fig3:ee-illustration-legend-box}
\end{figure}

\subsection{Adaptive sampling strategy: Target Uncovered Spaces (TUS)}

This is another comprehensive and consistent strategy that ensures that the algorithm efficiently focuses on random search and more targeted search for spaces that have not been adequately covered in previous iterations, thereby optimizing camera placement for maximum coverage. There are two phases in the TUS strategy similar to the E\&E strategy: 1) random search phase, same as the exploration phase in the E\&E strategy which generates random configurations, and 2) targeted coverage phase, which generates informed configurations.

The TUS strategy starts similar to the E\&E strategy, i.e., by determining the total number of camera configurations to be sampled, denoted as \( N_{\text{tot}} \). This is further expressed as the sum of two components: \( N_{\text{random}} \), representing the number of samples in the random search phase, and \( N_{\text{targeted}} \), representing the number of samples in the targeted search phase. This partitioning is conducted using $f_{\text{unc}}$, which denotes the fraction of all camera configurations in an iteration that are to be sampled during the targeted search phase.

The random search phase works exactly like the explore phase as it aims to diversify the search space by considering random configurations to achieve the best coverage. The algorithm samples \(N_{\text{random}}\) exploratory configurations, by randomly pairing positions from position set \(P\) with directions from direction set \(D\). Each of the $N_{\text{pos-random}}$ positions is paired with \(N_{\text{dir-pos}}\) directions each, yielding \(N_{\text{dir-pos}} \cdot N_{\text{pos-random}} \approx N_{\text{random}}\) configurations. This uses the auxiliary function Sample-Random-Configurations (refer to Appendix \ref{app:random_config}.

The targeted search phase focuses on sampling a subset of camera configurations specifically designed to cover sparsely-covered areas. This requires finding spaces that have the lowest coverage. The relative focus on different uncovered spaces is calculated by superimposing a supervoxel grid over the initial voxel grid. The counts of individual voxels uncovered in the optimal solution of the previous iteration, belonging to each supervoxel, are stored in \( \mathcal{G}_{\text{unc}} \). These provide relative probability densities on where the focus of targeted coverage should be. These probabilities are normalized, such that they sum to 1, using $C_{\text{normalize}}$.

Once the probabilities have been calculated, then we sample configurations proportionally based on the calculated probabilities. For a total of \( N_{\text{targeted}} \) times, a supervoxel center is sampled with replacement from a categorical distribution with probabilities \(Pr(v^{\text{super}}_{\text{center}}) \). Thus, the probability of choosing a particular supervoxel center is proportional to the number of uncovered free space voxels that particular supervoxel contains. Then, for each sample, a position $\tilde{p}$ is sampled. There is a strict visibility hyperparameter considered here. If this hyperparameter is set to `true', then we ensure that from the chosen camera position $p_{\text{new}}$, $v^{\text{center}}$ is visible without an environmental obstruction like walls or objects. We do this by finding the furthest position away from the $v_{\text{center}}$ along the direction $-(v_{\text{center}} - \tilde{p})$ that the camera position might be located while being unobstructed. This ensures that any new camera configuration thus found will always cover the target supervoxel center without obstructions.\footnote{This computation is encoded in the auxiliary function Linear-Visibility in Algorithm \ref{alg:tus-strategy}.}  If this hyperparameter is set to `false', then this visibility condition is not checked and any position $\tilde{p} \in P$ works as the camera position $p_{\text{new}}$. Further, the camera direction vector $d_{\text{new}}$ is set to $v_{\text{center}} - p_{\text{new}}$. Thus, the new sampled camera configuration is $(p_{\text{new}}, d_{\text{new}})$. This process is repeated until the desired number of configurations is achieved. This constitutes the targeted camera configuration set. Finally, the algorithm combines the random search set and the targeted set to form the final set of sampled camera configurations, \( PD_{\text{TUS}} \), which is then returned.

The complete algorithm is given in Algorithm \ref{alg:tus-strategy}, and the notation is explained in Appendix \ref{appD:tus-notation}.

\begin{algorithm}[H]
\caption{TUS strategy pseudocode}
\label{alg:tus-strategy}
\begin{algorithmic}[1]
\Procedure{Target-Uncovered-Spaces}{$N_{\text{pos}}, N_{\text{dir-pos}}, f_{\text{unc}}, P, V^{\text{super}}_{\text{center}}, \mathcal{G}_{\text{unc}}$}

    \State $N_{\text{tot}} \gets N_{\text{pos}} \cdot N_{\text{dir-pos}}$
    \State $N_{\text{random}} \gets \lceil \textcolor{white}{\rceil} N_{\text{tot}} \cdot (1 - f_{\text{unc}})
    \textcolor{white}{\lfloor} \rfloor$
    \Comment{For random search-phase samples}
    \State $N_{\text{pos-random}} \gets \lceil \textcolor{white}{\rceil} N_{\text{random}} / N_{\text{dir-pos}}
    \textcolor{white}{\lfloor} \rfloor$
    \State $PD_{\text{random}} \gets \text{Sample-Random-Configurations}(N_{\text{pos-random}}, P, N_{\text{dir-pos}})$

\State $N_{\text{targeted}} \gets \lceil \textcolor{white}{\rceil} N_{\text{tot}} \cdot f_{\text{unc}}
    \textcolor{white}{\lfloor} \rfloor$
\State $C_{\text{normalize}} \gets 0$
\For{$v^{\text{super}}_{\text{center}} \in V^{\text{super}}_{\text{center}}$}
    \State $C_{\text{normalize}} \gets C_{\text{normalize}} + \mathcal{G}_{\text{unc}}\lt(v^{\text{super}}_{\text{center}}\rt)$
\EndFor
\For{$v^{\text{super}}_{\text{center}} \in V^{\text{super}}_{\text{center}}$}
    \Comment{Calculate probability of targeting supervoxel}
    \State $Pr(v^{\text{super}}_{\text{center}}) \gets \frac{\mathcal{G}_{\text{unc}}\lt(v^{\text{super}}_{\text{center}}\rt)}{C_{\text{normalize}}}$
\EndFor
\State $PD_{\text{unc}} \gets \emptyset$
\For{$\text{num} = 1:N_{\text{targeted}}$}
    \Comment{For targeted search-phase samples}
    \State $v_{\text{center}} \sim \text{Categorical}\lt( Pr(v) | v \in V^{\text{super}}_{\text{center}} \rt)$
    \State $\Tilde{p} \sim P$
    \If{$\text{strict\_vis\_req} = \text{True}$}
        \State $p_{\text{new}} = \text{Linear-Visibility}(\Tilde{p}, v_{\text{center}}, P)$
    \Else
        \State $p_{\text{new}} = \Tilde{p}$
    \EndIf
    \State $d_{\text{new}} \gets v_{\text{center}} - p_{\text{new}}$
    \State $PD_{\text{unc}} \gets PD_{\text{unc}} \cup \{(p_{\text{new}},d_{\text{new}})\}$
\EndFor
\State $PD_{\text{TUS}} \gets PD_{\text{unc}} \cup PD_{\text{random}}$
\State \Return $PD_{\text{TUS}}$
\EndProcedure
\end{algorithmic}
\end{algorithm}

\subsubsection{Illustrative example of TUS strategy}

Figure \ref{fig4:tus-illustration} illustrates an iteration of the TUS strategy applied to a different section, the kitchen and dining room space of the apartment that will be studied in more detail in Section \ref{section:6-case-study}. Again, the legend for the individual cubic components is described in Figure \ref{fig3b:legend-box}.

In the first block, we segment the whole duplex apartment into large cubes, represented by the black demarcating lines. These represent supervoxels, within which we estimate the count of unviewable voxels, that help guide the strategy. In the first iteration, there are randomly sampled configurations in that section, which correspond to the random search phase as seen in the second block. The best network is chosen from the first set of samples as in the third block. In the fourth block, we observe that new informed configurations are considered which point towards the floor, as the floor space has a large count of uncovered voxels. In the fifth block we observe the algorithm choose a better camera network that covers the floor space more thoroughly. This process is repeated for $10$ iterations, with the final solution showing that a single camera is able to cover a much larger space, with recursive updates. As we are visualizing a part of the entire environment, some of the resource budget is utilized to cover other spaces and are not observed in the block.

\begin{figure}
    \centering
    \includegraphics[width=\textwidth]{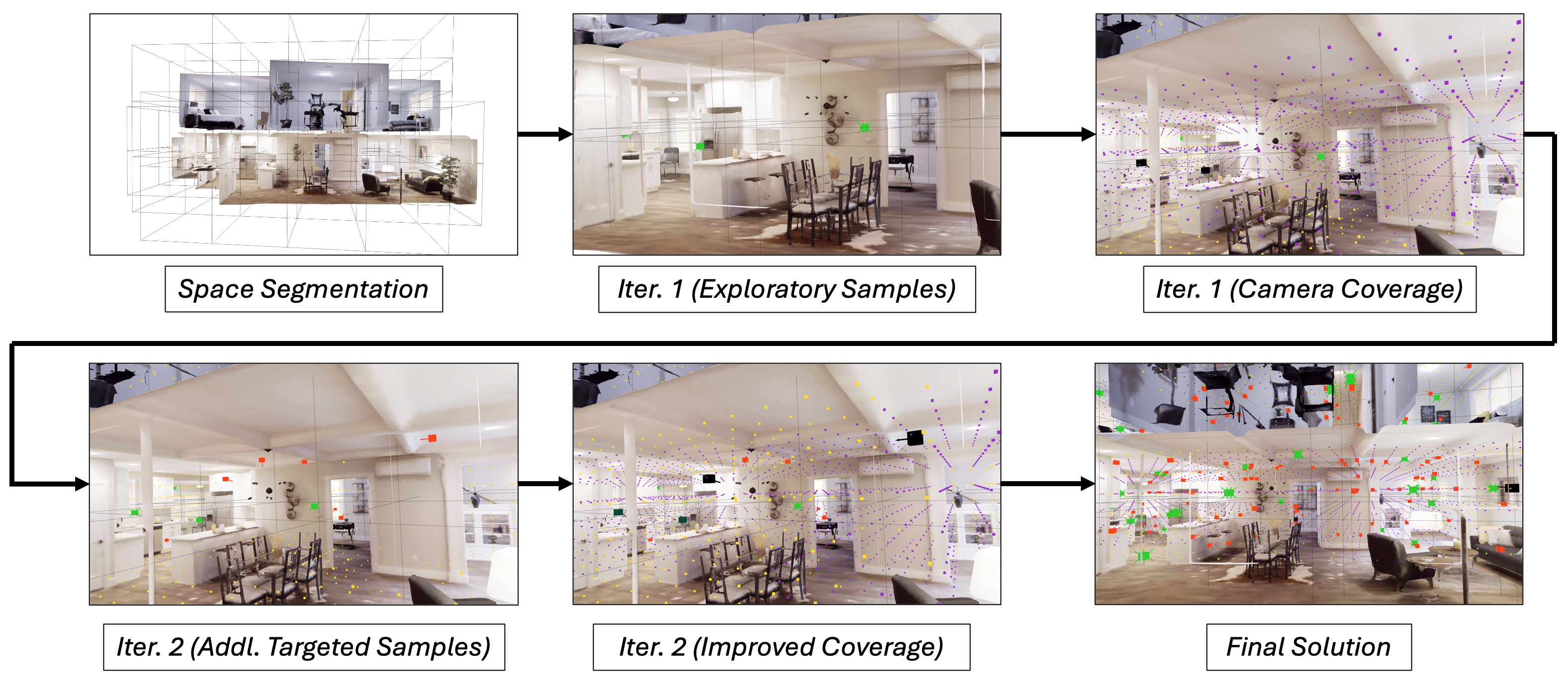}
    \caption{Iteration of the TUS strategy applied to indoor apartment, focusing on a subset of the environment—specifically, the kitchen and dining room. (detailed in Section \ref{section:6-case-study}). Large cubes (supervoxels) are segmented in the first block, followed by random sampling in the second (iteration 1). The best network is then selected in the third block, with two cameras visible in the observable portion. The fourth and fifth blocks (iteration 2) refine camera placements for better floor coverage. The sixth block illustrates the final iteration, achieving better overall coverage with a single camera placed in the observed space. The rest of the resource budget is utilized to cover other spaces which are not observed in the block. See Figure \ref{fig3b:legend-box} for the legend of cubic components}
    \label{fig4:tus-illustration}
\end{figure}

\subsection{Benchmark algorithms}

Our solution framework features three main components: sampling strategy, visibility calculations, and solving an optimization problem. We benchmark the efficiency of the components we propose by comparing them against benchmark algorithms. We conduct this benchmarking for the sampling strategy and the optimization method used to find the chosen camera network given the sample camera configurations.

\subsubsection{Sampling strategy: Random Selection (RS)} \label{section:4.4.1-random-select}

We benchmark the efficiency of the sampling strategy by considering a naive sampling strategy. The naive approach is merely choosing all sample configurations at once and at random. The approach contrasts with the adaptive sampling strategy, where we iterate over the sampling function multiple times and sample strategically. In the naive approach, the entire sampling budget is exhausted at once in one iteration. So, none of the new samples are chosen using the environmental information or based on previous iteration solutions. 

We check efficiency of our strategy by two means: time taken and maximum coverage achieved, while holding other hyperparameters constant. So, in this case, we hold the sampling budget constant. Therefore, we can set the sample configuration set as $PD' = \text{Sample-Random-Configurations}(N_{\text{pos}} \cdot N_{\text{iter}}, P, N_{\text{dir-pos}})$, where $N_{\text{iter}}$ represents the total count of iterations during which the sampling algorithm is invoked in a single run. The complete pseudocode for random sampling is provided in Appendix \ref{app:random_config}.

\subsubsection{Optimization framework: greedy heuristic} \label{section:4.4.2-greedy-heuristic}

We can compare the optimal solution of the IP with that of an alternative greedy heuristic. The core idea behind the greedy heuristic is that we greedily choose the best camera configuration, i.e., we keep choosing the camera configuration that maximizes the marginal coverage gain per unit cost after each configuration selection until we exhaust the resource budget. Simultaneously, we also remove configurations which conflict with previously chosen ones according to constraint $(4)$ of the optimization model provided in Section \ref{section:3.2-IP}. As the original formulation is a provably NP-hard problem, the greedy heuristic will provide a lower bound on the optimal solution of the IP. This algorithm is solvable in pseudopolynomial time with a complexity of $\mathcal{O}(\beta \cdot |PD'| \cdot |V|)$, as it depends on data $\beta$. Here, \( \mathcal{O}(\cdot) \) denotes an upper bound. Specifically, if \( f(n) = \mathcal{O}(g(n)) \), it means \( f(n) \) grows at most as fast as \( g(n) \), up to a constant factor. The proof of this complexity and the pseudocode for the algorithm are provided in Appendix \ref{app:greedy_heuristic}.

\section{Theoretical analysis} \label{section:5-theoretical-analysis}

In this section, we theoretically establish the consistency of our algorithmic framework. Furthermore, we demonstrate that our optimization framework exhibits monotonicity, i.e., 
the solution keeps on improving as more camera configuration samples are considered.

\subsection{Sampling strategy}

The analysis in this section provides a consistency proof, demonstrating that as the number of samples increases, the probability of identifying the optimal camera network converges to $1$.

Specifically, we analyze two key aspects: (1) the probability of obtaining the true optimal solution after sampling $n$ camera configurations, and (2) the expected number of samples required to achieve the optimal solution. Such a result is useful, as it ensures that the framework being used, while heuristic in nature, possesses provable guarantees of success. It simultaneously allows us to comment on the rate of convergence.

We begin by defining a discrete set of camera configurations, where each configuration consists of a position \( p \) and a direction \( d \). The set of feasible positions \( P \) is a discrete subset of free space voxels \( V \). In contrast, the set of feasible directions \( D \) is continuous and represented as points sampled on the unit sphere. To handle this, we discretize \( D \) using single-precision floating-point representation, where \(\epsilon\) denotes the smallest positive difference between distinct representable values in this format.

We introduce additional notation to develop the theoretical analysis. To express the growth rate of functions, we use \( \mathcal{O}(\cdot) \), which gives an upper bound. A more rigorous definition is provided in Section \ref{section:4.4.2-greedy-heuristic}. Additionally, we use \( Pr(\cdot) \) and $\mathbb{E}[\cdot]$ to denote the probability and the expectation of the expression inside the parentheses, respectively.

We now proceed to calculate the number of possible camera configurations.\\

\begin{lemma}
    Given a single-precision floating-point representation, with \( \epsilon \) as the smallest distinguishable increment, and \( |P| \) as the cardinality of the position set, the number of possible camera configurations is on the order of \( \mathcal{O} \left( \frac{|P|}{\epsilon^2} \right) \).
\end{lemma}

\begin{proof}
    The proof is provided in Appendix \ref{appE:lemma1}.
\end{proof}

Now that we have a finite number representing the cardinality for the set $(P \times D)$, we can calculate the probability of getting the optimal solution. Let us assume that we have sampled $n$ camera configurations randomly, and we intend to find the probability of finding the optimal solution. Without loss of generality, we assume that the resource budget $\beta$ represents the maximum number of cameras that can be placed in the network, implying a unit cost for each camera configuration ($f_{pd} =1$). While we consider a unit cost for simplicity, this analysis can be extended to scenarios with variable configuration costs, leading to tighter bounds.

As the sample space of possible camera configurations is discretized, the problem is now equivalent to finding the optimal subset of camera configurations with cardinality $\beta$, which provides the maximum coverage. There is at least one subset of camera configurations which achieves this maximum coverage as the coverage objective function is bounded by $|V|$. Suppose we sample $n$ configurations. In the next part, we compute the probability that this sample of size $n$ contains the optimal solution. \\

\begin{proposition}
    The probability that a randomly sampled set of $n$ configurations contains the set of optimal configurations with cardinality \(\beta\) is at least $\lt(\frac{n - (\beta - 1)}{|P \times D|}\rt)^\beta$.
\end{proposition}

\begin{proof}
    As we are sampling without replacement, we calculate the probability that the sample set of camera configurations contains the specific $\beta$ camera configurations that provide the optimal solution. Consider random variable $L$, which represents the number of optimal camera configurations chosen when $n$ camera configurations are chosen at random from a population of $|P \times D|$ configurations. $L$ can be represented by a hypergeometric distribution with parameters $N = |P \times D|$, $K = \beta$. So, the probability of finding the optimal solution is equivalent to $Pr(L = \beta)$, which is given as:
    \begin{align*}
        Pr(L = \beta) &= \frac{{|P \times D| - \beta \choose n - \beta}}{{|P \times D| \choose n}} & \\
        &= \frac{\lt(|P \times D| - \beta \rt)!}{\lt(|P \times D| \rt)!} \cdot \frac{n!}{\lt(n - \beta \rt)!}  & \\
        &= \frac{\prod_{i=0}^{\beta-1} (n-i)}{\prod_{i=0}^{\beta-1} \lt(|P \times D|-i\rt)}
    \end{align*}

    Now, $(n-\beta+1) \leq (n-i)$ for all $ i \in \{0,1, \cdots, \beta - 1\}$ and $|P \times D| \geq \lt(|P \times D| - i \rt)$ for all $i \in \{0,1, \cdots, \beta - 1\}$. Therefore,

\begin{align*}
    Pr(L = \beta) &= \frac{\prod_{i=0}^{\beta-1} (n-i)}{\prod_{i=0}^{\beta-1} \lt(|P \times D|-i\rt)} \geq \lt( \frac{n-(\beta - 1)}{|P \times D|} \rt)^\beta,
\end{align*}
which completes the proof.

\end{proof}

Typically, $n \gg \beta$ in most realistic scenarios. Thus, this would mean $Pr(L=\beta) \approx \lt( \frac{n}{|P \times D|} \rt)^\beta$.

Finally, we calculate the expected number of samples needed to achieve the optimal solution. We assume we keep sampling until all samples from the optimal camera network are included in the set of sampled configurations.\\

\begin{proposition}
    The expected number of samples required in the sample set such that it contains the set of optimal configurations is $\frac{\beta}{\beta + 1} \cdot \lt(|P \times D| + 1 \rt)$.
\end{proposition}

\begin{proof}
    Let us define a random variable $M$ which represents the number of sample configurations sampled such that its subset includes all $\beta$ optimal camera configurations. If this set contains all of them, then the optimization framework will select them. Thus, the expected number of samples needed equals $\mathbb{E}[M]$, where $M \in \{\beta, \beta+1, \cdots, |P \times D|\}$.

Now, as we stop sampling after getting the optimal network in the sample set, the last sample chosen has to be part of the optimal network. Therefore, $Pr(M=m)$ can be defined as the probability of choosing $\beta-1$ configurations from the optimal set in the first $m-1$ configurations, and then choosing the final optimal configuration in the $m^{\text{th}}$ configuration.
\begin{align*}
    Pr(M=m) &= \frac{{\beta \choose {\beta-1}} \cdot {|P \times D| - \beta \choose {m - \beta}}}{{|P \times D| \choose {m - 1}}} \cdot \frac{1}{|P \times D| - m + 1} \\
    & = \frac{\beta (|P \times D|-\beta)!(m-1)!}{|P \times D|!(m-\beta)!}
\end{align*}

Now, we can calculate the expectation as:
\begin{align*}
    \mathbb{E}[M] &= \sum_{m = \beta}^{|P \times D|} m Pr(M=m) \\
    &= \frac{\beta \cdot \lt(|P \times D| - \beta \rt)!}{|P \times D|!} \sum_{m = \beta}^{|P \times D|} \frac{m!}{(m - \beta)!} \\
    &= \frac{\beta \cdot \beta! \lt(|P \times D| - \beta \rt)!}{|P \times D|!} \sum_{m = \beta}^{|P \times D|} \binom{m}{\beta}
\end{align*}

Utilizing the result $\sum_{m=a}^{b} \binom{m}{a} = \binom{b+1}{a+1}$ \citep{andrews_special_1998}, we get:

\begin{align*}
    \mathbb{E}[M] &= \frac{\beta \cdot \beta! \lt(|P \times D| - \beta \rt)!}{|P \times D|!} \sum_{m = \beta}^{|P \times D|} \binom{m}{\beta} \\
    &= \frac{\beta \cdot \beta! \lt(|P \times D| - \beta \rt)!}{|P \times D|!} \cdot \frac{\lt(|P \times D| + 1 \rt)!}{\lt(\beta + 1 \rt)! \lt(|P \times D| - \beta \rt)!} \\
    &= \frac{\beta}{\beta + 1} \cdot \lt(|P \times D| + 1 \rt)
\end{align*}

Hence, the proposition is proven.

\end{proof}

As observed, the value of $\epsilon$ is very small, thus the probability of reaching an optimal solution within feasible sampling budgets is very low, whereas the expected number of configurations to be sampled is very high. However, it should be noted that this value of $\epsilon$ is a strict lower bound. This is based on the fact that we cannot gather any other data regarding the function's discrete smoothness, i.e., how changing the camera configuration slightly will affect the corresponding coverage. We surmise that changing a camera view slightly changes visibility slightly, however, this smoothness is hard to quantify. It is our conjecture that the true value of $\epsilon$, which is a quantification of the minimum change in input that will lead to a unit change in output, is much greater than the floating-point precision that we have categorized it as.

Further, it should be noted that random samples in the theoretical analysis correspond to exploratory samples in the E\&E framework, and to random samples in the TUS framework. Thus, these propositions ensure that our sampling strategies guarantee convergence as long as there are exploratory samples and we sample without replacement.

This raises the question of why our strategies prioritize non-random sampling, despite the fact that exploratory random sampling ensures eventual convergence to an optimal solution. As we provide empirical evidence for in subsequent sections, our frameworks consistently outperform purely random sampling strategies, particularly in terms of reaching high-quality solutions more quickly. The advantage of non-random sampling lies in its ability to utilize environmental information and iterative solution information, allowing the strategies to efficiently navigate the solution space. While random sampling guarantees eventual convergence, its practicality diminishes due to the extended time required to achieve optimal solutions. We plan to investigate these dynamics more comprehensively in future research work.

\subsection{Optimization framework}

Here, we further prove a favorable property exhibited by our optimization framework. Under our framework, we can prove that the solution after iteration $k+1$ will be at least as good as the solution after iteration $k$. This is because, even though the objective remains constant, we keep adding new variables. Let us assume that the set of sample camera configurations at iteration $k$ is $PD^k$. As the algorithm relies on adding new samples at every iteration, $|PD^{k+1}| > |PD^{k}|$. This allows us to prove the following proposition. \\

\begin{proposition}
    Suppose we obtain the camera network by following the optimization process and employing one of the proposed solution algorithms. Then, the coverage improves monotonically. Furthermore, if we choose a sufficiently large sample set such that \(PD^k \to PD\) and solve the IP to optimality, we are guaranteed to reach true optimality.
\end{proposition}

\begin{proof}
    Consider $IP(PD^k, V, \beta, \{f_{pd}\})$ at iteration $k$ as defined in Section \ref{section:3.2.1-model-formulation}. It can be expanded as:
\[
\begin{array}{llll}
    \max & \sum_{v \in V} y_v  && (1) \\
    \text{s.t.} &&& \\
    \text{Camera cost:} & \sum_{(p,d) \in PD^k} f_{pd} x_{pd} \le \beta && (2) \\
    \text{Voxel coverage:} & y_v - \sum_{(p,d) \in PD^k_v} x_{pd} \le 0 &\quad \forall v \in V & (3) \\
    \text{One camera per locale:} & \sum_{(p, d) \in PD^{k\text{adj}}_p} x_{pd} \le 1 &\quad \forall (p,d) \in PD^k & (4) \\
    \text{Integrality:} & x_{pd} \in \{0, 1\} &\quad \forall (p,d) \in PD^k & (5\text{a}) \\
    \text{Integrality:} & y_{v} \in \{0, 1\} &\quad \forall v \in V & (5\text{b})
\end{array}
\] \\

Let $x_{pd}^{k*}$, $y_v^{k*}$ represent the values obtained after solving the IP at iteration $k$. Note that it necessarily does not have to be solved to optimality. At iteration $k+1$, $PD^k \subset PD^{k+1}$, $PD_p^{\text{adj}^k} \subseteq PD_p^{\text{adj}^{k+1}} \quad \forall p \in P$ and $PD^k_v \subseteq PD^{k+1}_v$. This means all constraints are satisfied by setting the values of the IP in iteration $(k+1)$ based on the following rule: $x_{pd}^{k+1} = x_{pd}^{k*} \quad \forall (p,d) \in PD^{k}$ and $x_{pd}^{k+1} = 0 \quad \forall (p,d) \in PD^{k+1} \setminus PD^{k+1}$ and $y_v^{k+1} = y_v^{k*} \quad \forall v \in V$. Thus, the solution obtained in iteration $k$ is a lower bound on the optimal solution obtained in iteration $k+1$.

Further, as we are sampling without replacement and the set $P \times D$ is finite within computer specifications, $PD^k \xrightarrow[]{} |P \times D|$ in finite number of iterations. As that is the case, the corresponding $IP(|P \times D|, V, \beta, \{f_{pd}\})$ represents the full master problem in a column generation framework. This nature of the IP was briefly discussed in Section \ref{section:3.2.1-model-formulation}. That means if this IP is solved to optimality, we will have, within computer accuracy, the true optimal solution to the original problem.
\end{proof}

As briefly mentioned earlier, we use the solution from the previous iteration to warm-start the optimization algorithm in the next iteration. This ensures that even if we are not solving the IP optimally at every iteration, we ensure the monotonicity property of the algorithmic framework, i.e. coverage will either remain the same or improve at every iteration as the optimizer will not revert to a worse solution. Furthermore, warm-starting helps in pruning the branch-and-bound tree in traditional optimizers and helps speed up the process of finding the best solution network given the available set of configurations.

\section{Case study: assessing solution algorithms in a large and complex real-world setting} \label{section:6-case-study}

This section focuses on a case study, where the two adaptive sampling solution algorithms are implemented on a real-life 3D environment. The main theme of the case study is to perform a high-resource budget and low-resource budget coverage of the environment to gain insights into the solutions obtained. Further, it demonstrates the ability of our algorithms to find high-quality solutions to problems of real-world size and complexity, in scenarios that are better suited for them. We explain in Section \ref{section:7-algo-perform} the reasons for that. Furthermore, based on the model solutions, we are able to provide insights on how conventional wisdom for camera placement corresponds to and differs from the optimal solutions that our framework identifies.

The case study is conducted for `apartment\_0' in the REPLICA Dataset (see Figure \ref{fig5:case-study-environment}), provided by Meta \citep{straub_replica_2019}. It is a duplex apartment characterized by three bedrooms, three bathrooms, one dining room, two living rooms, one study room, one kitchen, and two staircases. 

There are several complex features of this environment that make it challenging. The apartment is segmented by multiple partitions. Each bathroom is partitioned separately from the bedrooms, and multiple walls divide the various rooms, with each bedroom facing a different direction. Additionally, several doorways and corridors obstruct complete visibility from one room to the next.

\begin{figure}[h]
    \centering
    \subfigure[]{%
        \includegraphics[width=0.7\textwidth]{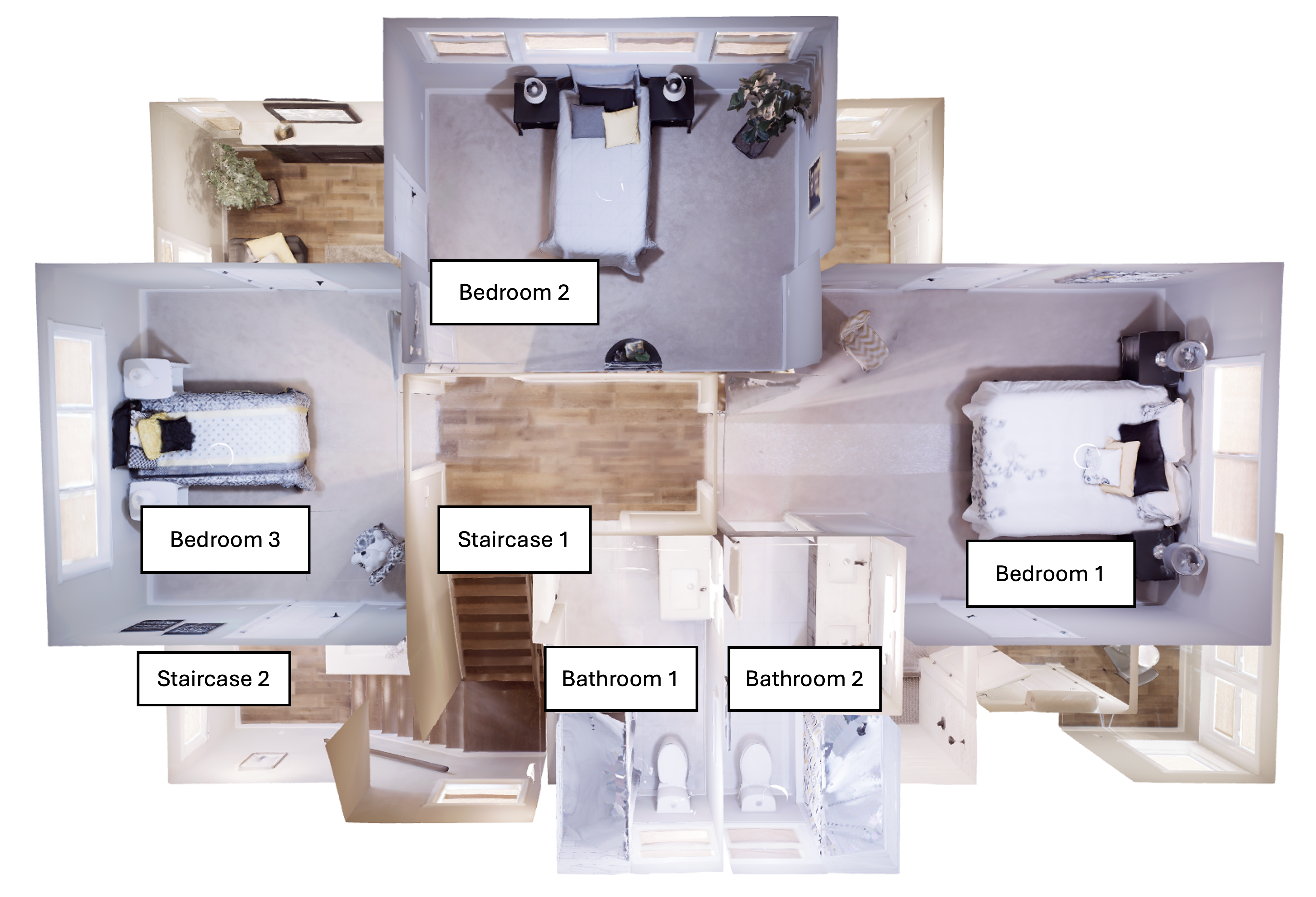}
        \label{fig5a:case-study-environment}%
    }\\[1em]
    
    \subfigure[]{%
        \includegraphics[width=0.7\textwidth]{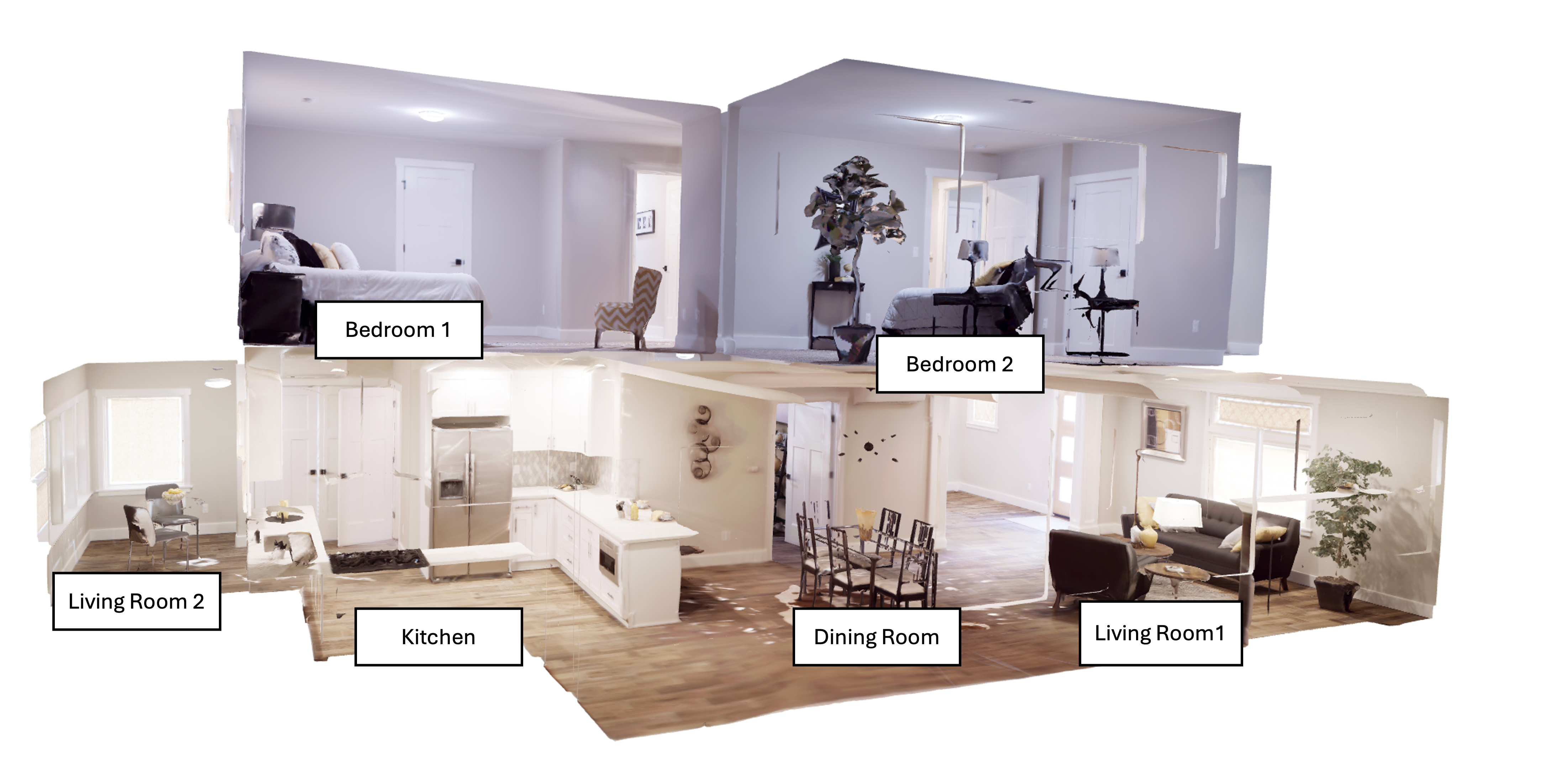}
        \label{fig5b:case-study-environment}%
    }\\[1em]
    
    \subfigure[]{%
        \includegraphics[width=0.7\textwidth]{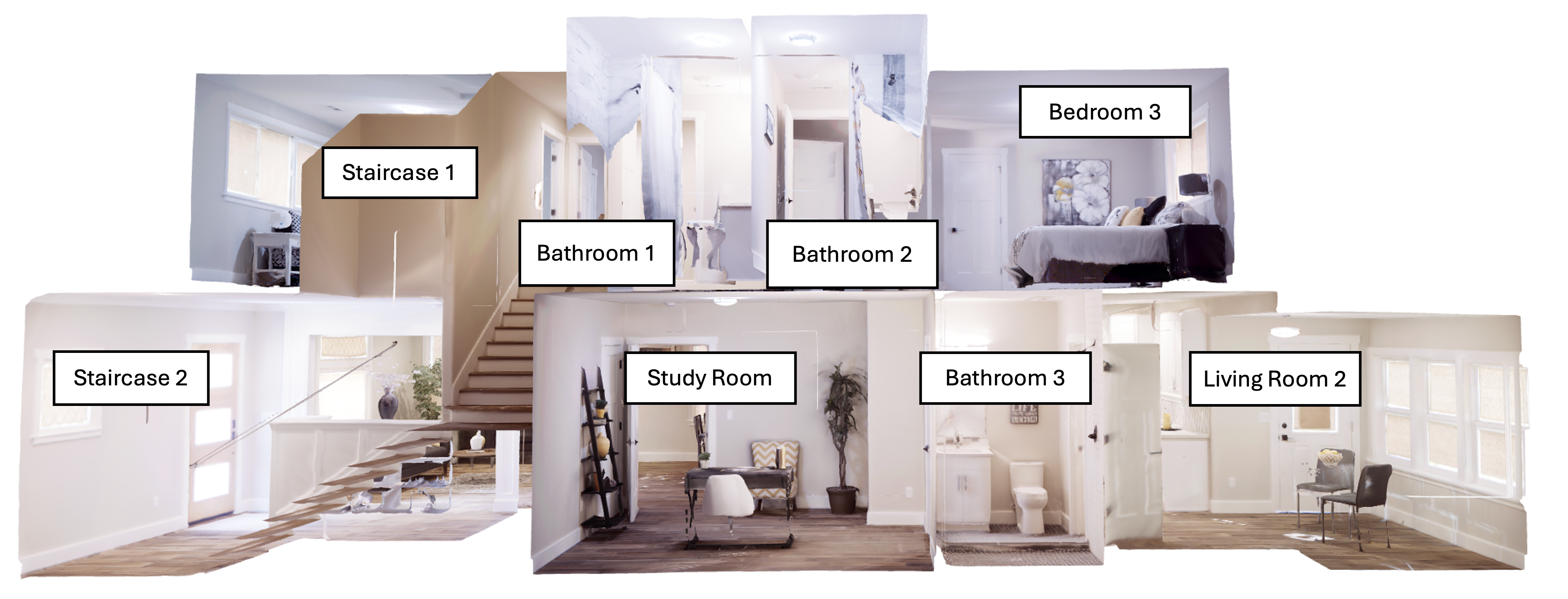}
        \label{fig5c:case-study-environment}%
    }
    \caption{Top (Part (a)), front (Part (b)), and back (Part (c)) views of the duplex apartment `apartment\_0'. The top view highlights the rooms on the top level, including bedrooms, bathrooms, and staircases (top). The front view shows the top-level bedrooms along with lower-level living spaces, including living rooms, the dining room, and the kitchen, with visibility impeded by diagonal walls (middle). The back view shows smaller lower-level spaces (study room, bathroom 3, living room 2) and their relation to top-level rooms (bottom)}
    \label{fig5:case-study-environment}
\end{figure}

\subsection{Pre-processing}

There are two types of pre-processing required. The 3D mesh object is highly granular with over $4.5 \times 10^6$ vertices and $9 \times 10^6$ triangles. All proceeding analysis has been conducted using Python.

The region is pre-processed using quadric decimation \citep{garland_surface_1997} to make the dimensions more manageable, reducing it to $5 \times 10^5$ vertices and $10^6$ triangles. This ensures maintaining the proper 3D shape and structure of the environment while easing the computational burden of visibility calculations. The primary issue lies in conducting the obstacle checking of the camera visibility. This process ensures that the number of ray-triangle intersections considered is manageable.

Furthermore, the free space grid is calculated such that every voxel is a size of 1 ft$^3$ for fine positioning of the camera. There are a total of 5,985 free space voxels. This is not equal to the volume of the apartment as a lot of the space is occupied by objects. We have implemented a standard pre-processing script that can create a voxelized grid structure for any environment.

\subsection{Scenario 1: high-resource budget performance of E\&E strategy}

Based on the algorithm performance results that we will cover in Section \ref{section:7-algo-perform}, we find that the E\&E strategy is a very efficient solution algorithm compared to the random sampling strategy, both in overall coverage and speed of achieving said coverage. The TUS strategy performance is only strong in certain contexts, such as open environments with a low budget. This is why we choose to perform the high-resource budget coverage maximization using the E\&E strategy, as TUS is not expected to perform well in this case. 

\subsubsection{Resource budget}

Given that there are three bedrooms and eight other rooms, we consider a camera resource budget of 6, or approximately 1 camera per two rooms. This ties in well with the other consideration of a high resource budget of 1 per $10^3$ free space voxels, and low resource budget of 1 per $2 \times 10^3$ free space voxels, as there are approximately 6,000 voxels of free space. It should be noted that the model complexity is not severely affected by the resource budget. It is more significantly affected by the number of free space voxels. However, the sampling algorithm is affected by the resource budget.

\subsubsection{Best configurations}

We present a visual description of the coverage achieved in Figure \ref{fig6:overall-coverage-high-resource-budget}. The uncovered voxels are clustered around the two bathrooms at the top, near the far edges of the bedrooms, and to some degree around the staircase.

\begin{figure}[h]
    \centering
    \subfigure[]{\includegraphics[width=0.9\textwidth]{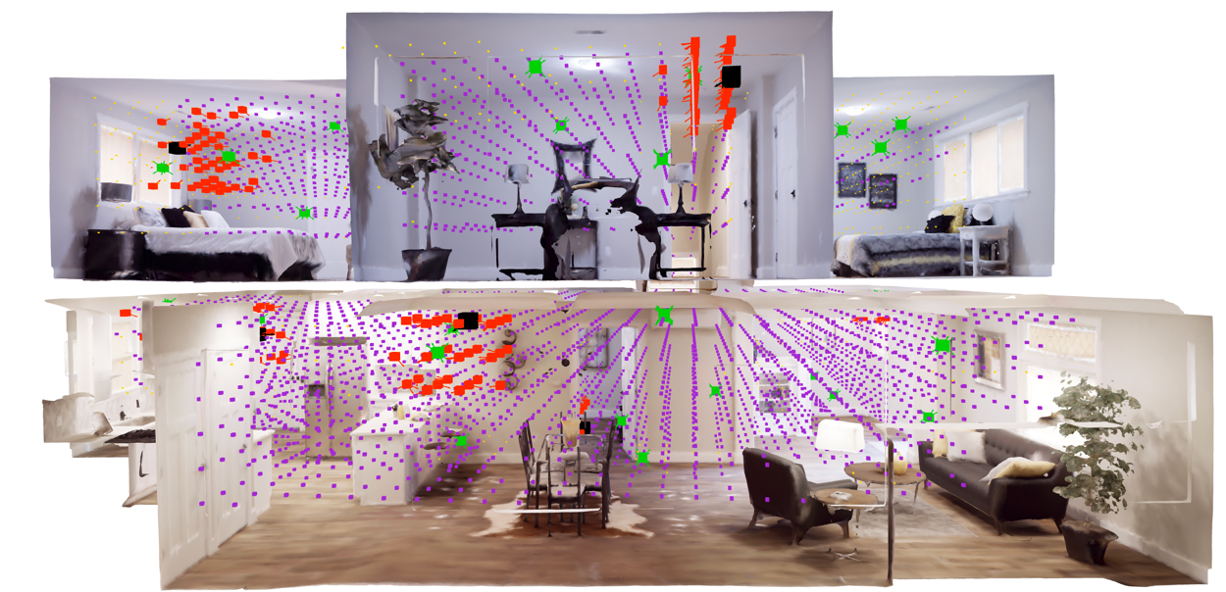} \label{fig6a}}
    
    \vspace{0.5em} 
    
    \subfigure[]{\includegraphics[width=0.9\textwidth]{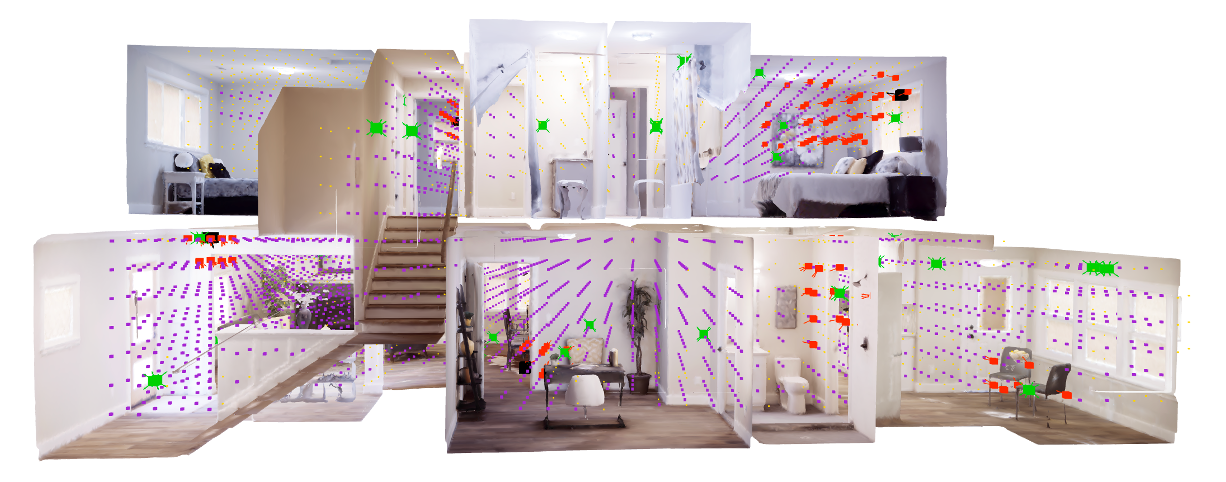} \label{fig6b}}
    
    \caption{This is a panel of two images, where the top image (Part (a)) gives a frontal view of the coverage achieved, and the bottom image (Part (b)) provides a back view (refer to Figure \ref{fig3b:legend-box} for legend). The uncovered regions are clustered around the two bathrooms at the top, near the far edges of the bedrooms, and to some degree around the staircase.}
    \label{fig6:overall-coverage-high-resource-budget}
\end{figure}

\begin{figure}[h]
    \centering
    \subfigure[]{\includegraphics[width=0.22\textwidth]{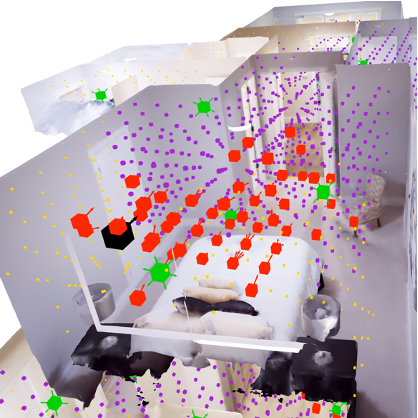} \label{fig7a}}
    \hspace{0.5em}
    \subfigure[]{\includegraphics[width=0.22\textwidth]{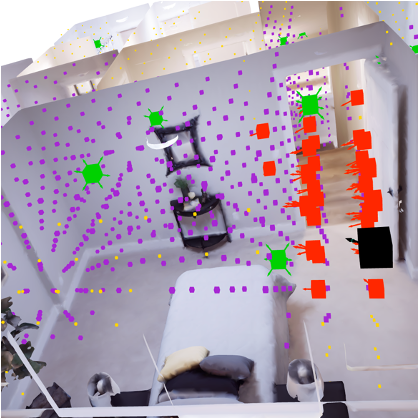} \label{fig7b}}
    \hspace{0.5em}
    \subfigure[]{\includegraphics[width=0.46\textwidth]{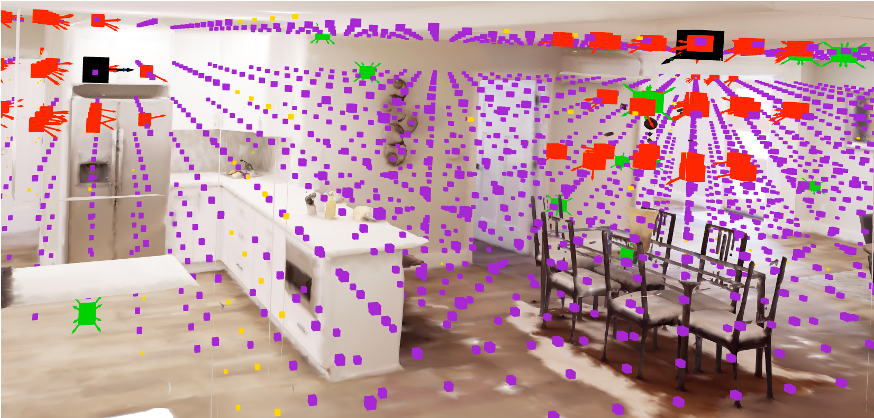} \label{fig7c}}
    
    \vspace{1em} 
    
    \subfigure[]{\includegraphics[width=0.22\textwidth]{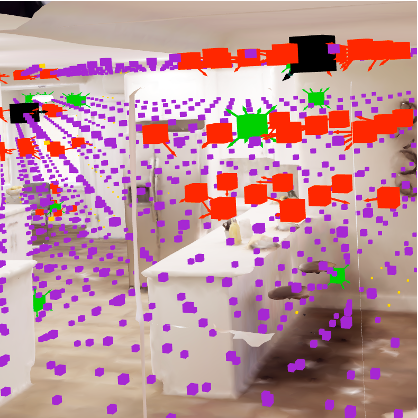} \label{fig7d}}
    \hspace{0.5em}
    \subfigure[]{\includegraphics[width=0.22\textwidth]{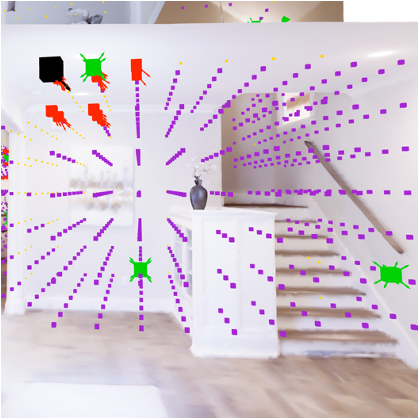} \label{fig7e}}
    \hspace{0.5em}
    \subfigure[]{\includegraphics[width=0.265\textwidth]{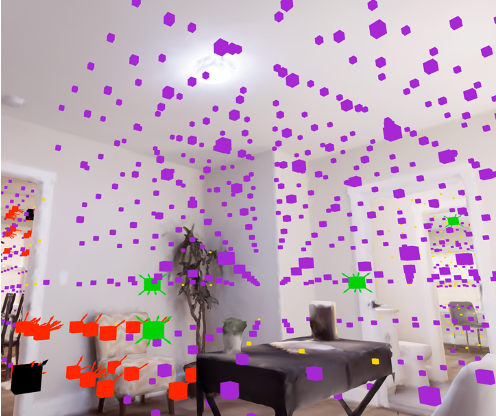} \label{fig7f}}
    \hspace{7em}
    \caption{This is a $2 \times 3$ panel of the six camera configurations as chosen by the E\&E algorithm (refer to Figure \ref{fig3b:legend-box} for legend). The individual camera configurations and their coverages are presented here. (Disambiguation: camera configuration c is the camera on the left in the third panel. Similarly, camera configuration d is the camera on the left in the fourth panel.)}
    \label{fig7:best-configurations-high-resource-budget}
\end{figure}

The optimal configurations and their coverage are described below. The configurations can be observed in Figure \ref{fig7:best-configurations-high-resource-budget}.

\textbf{Camera configuration a:} It is placed at the top in the master bedroom. However, it is purposefully not located at a corner point but rather along the middle of the edge, facing straight ahead, so that it has an unobstructed view to the bedroom in front as well. This is why the third bedroom does not have a camera located in it.

\textbf{Camera configuration b:} This camera is in the bedroom facing the staircase. It is angled such that it has a clear view of the whole room and can view as much of staircase 1 as possible.

\textbf{Camera configuration c:} There is one camera placed in the kitchen at the top, angling downwards. It is placed in such a way that it is able to view the entirety of the dining room and a large portion of the unobstructed living room 1 as well.

\textbf{Camera configuration d:} This camera is placed in the larger living room and it seems like it is facing camera configuration c, which makes it redundant. However, doing that allows it to cover one section of the living room which cannot be covered by camera configuration c, and allows it to view a significant part of the second living room.

\textbf{Camera configuration e:} This camera is placed at the top near the bottom of the second staircase, placed in such a way that allows it access to the adjoining living room through a doorway.

\textbf{Camera configuration f:} This is the only position which is mildly odd from a conventional perspective but implements the objective as it is supposed to. It is placed at the bottom of the study room, to capture the entire study room, and look out of the doorway to capture part of the living room.

\subsubsection{Comparison with conventional wisdom for camera placement} 

We placed no constraints on where the cameras could be placed within the free space, i.e., the camera positions were not forced to be at the tops of rooms or near the corners, and they could be placed on the floor as well. Our optimization framework helped us determine where to place the cameras without any preconceived notions.

The main uncovered spaces are the bathrooms, which is because they are small spaces, and this follows conventional intuition. There are small free spaces along the ends, which are uncovered. Due to minor pre-processing errors, a small fraction of the free space voxels are outside the apartment, and they do not need to be covered.

Five of the six cameras are placed along or near the top wall independently. Even though the exploratory camera configurations are sampled from the middle, the algorithm iteratively gravitates the chosen configurations towards the top positions.

One convention that the algorithm's solution to the case study breaks from compared to standard camera placements is that it is not necessary to place cameras in room corners or even along edges to maximize coverage, though they are mostly near or along the ceiling or the surface of a wall. The exploited and exploratory configurations clearly show that although some of the configurations considered were along corners, they were ultimately rejected in favor of alternatives that offered superior coverage.

Another feature of the model solution which defies conventional strategies is that two of the cameras essentially face each other. Camera configuration c and camera configuration d are facing each other partially but they cover some different regions.

An important insight that we gain is that when covering large partitioned spaces, such as the main bedrooms or the living room, the algorithm selects camera configurations that maximize coverage of the open area while directing peripheral vision toward corridors and distant spaces for additional coverage. The utility of a camera position increases significantly when it is strategically placed to offer some view into adjoining spaces, particularly doorways and corridors, enabling maximum coverage.

\subsubsection{Performance metrics}

The analysis was carried out using the best hyperparameters set for the E\&E strategy, as estimated empirically in Appendix \ref{app:hyper_tuning}. We set the sampling budget at 800 configurations, spread over 10 iterations, with an exploit fraction of 0.6. This corresponds to 4 exploratory positions per iteration, and 8 angles sampled per voxel per iteration.

The performance metrics for this run can be summarized as follows. The total processing time for visibility calculations is 365.70 seconds, and the total optimization time is 7.96 seconds. The total free space is 5,985 voxels, and the final coverage achieved is 4,427 voxels (74\% of the free space). The summary statistics for coverage of each camera are given in Table \ref{table1:camera_coverage_ee}.
\begin{table}[htb]
    \centering
    \caption{Camera configurations and corresponding coverage statistics for the E\&E algorithm. It provides the optimal solution for a resource budget of 6 cameras and details the network coverage provided, number of voxels covered, and overcoverage, i.e., voxels covered more than once.}
    
    \begin{tabular}{p{3.5cm}p{4cm}p{3cm}p{2.6cm}}
        \hline
        \raggedright \textbf{Camera Configuration} & \textbf{Camera Network Coverage Provided (\%)} & \raggedright \textbf{Number of Voxels Covered} & \raggedright \textbf{Voxels Covered More than Once} \\
        \hline
        Camera configuration a & 13.8\% & 610 & 25 \\
        
        Camera configuration b & 12.7\% & 561 & 45 \\
        
        Camera configuration c & 36.2\% & 1603 & 295 \\
        
        Camera configuration d & 22.8\% & 1008 & 291 \\
        
        Camera configuration e & 10.4\% & 459 & 24 \\
        
        Camera configuration f & 11.9\% & 526 & 0 \\
        \hline
    \end{tabular}
    
    \label{table1:camera_coverage_ee}
\end{table}

Intuitively, the optimal camera network should avoid having too much overcoverage (i.e., voxels that are covered by multiple cameras), since overcoverage suggests that cameras are functioning redundantly and might be better positioned or oriented elsewhere. We see that overcoverage is limited in the solution described in Table \ref{table1:camera_coverage_ee}, as it applies to only a small fraction of the total voxels covered by each camera configuration. Only configurations c and d have overcoverage of more than 10\% of their visible voxels. This is a result of them covering large spaces on the lower level and the fact that they partially face each other.

\subsection{Scenario 2: low-resource budget performance of TUS strategy}

We test the TUS strategy in a low-resource budget scenario using the same environment as the previous case study. We contrast this with a high-resource budget case study utilizing the E\&E strategy, which provides new insights into optimal camera placement. We find under the best selected network, the three cameras are placed to cover large open spaces like the living room and kitchen. Rigorous algorithmic results, as discussed in Section \ref{section:7-algo-perform}, show that while the TUS strategy performs well in open environments with limited resources, its effectiveness is not on par with the E\&E strategy. We discuss the detailed analysis, including performance metrics, camera configurations, and voxel coverage, in Appendix \ref{app:scenario-2}.

\section{Algorithm performance} \label{section:7-algo-perform}

In this section, we compare our algorithmic components with the benchmark algorithms to evaluate how well they perform. The adaptive sampling strategies are compared against the RS strategy, which was described in Section \ref{section:4.4.1-random-select}. The optimization framework of the modified maximum $k$-coverage IP model is compared against its greedy heuristic counterpart as described in Section \ref{section:4.4.2-greedy-heuristic}.

To benchmark these components, we have developed sanitized, custom environments. We compare the strategies' performance over multiple scenarios. First, we consider two sizes for the environments: large-sized spaces and medium-sized spaces. Second, there are two categories of each space considered, type 1 (more obstructed) and type 2 (less obstructed). Finally, we consider two different resource budgets: a high-budget scenario and a low-budget scenario for every space size and space type. A more thorough description of how we construct custom environments is provided in Appendix \ref{app:room_gen}. These environments allow us to test these models under different scenarios and hyperparameters. The best hyperparameters, based on the results obtained in Appendix \ref{app:hyper_tuning}, have then been used to test the strategies to compare final model performance.

To account for the inherent stochasticity in the algorithms, which affects model performance and solution time, we repeat each experiment over each specific environment and resource budget five times. This results in five trials with varying samples. This approach allows us to evaluate the variability in coverage and computational performance across different hyperparameter sets, thereby assessing the robustness of the algorithms.

For this performance comparison, all algorithms were implemented on a shared Ubuntu 22.04.5 LTS system with CUDA support, using Python.\footnote{The codebase will be publicly released following the manuscript's publication.} Computations were run on a machine equipped with an Intel Xeon Gold 6426Y processor, which has 64 CPU cores and two NUMA nodes. Each run was allocated 4 cores and 16 GB of RAM. The available GPU for computations was an NVIDIA L40S with driver version 550.120 and CUDA 12.4. Memory available on the system totaled 251 GB, supplemented by 15 GB of swap memory, with actual allocations varying based on usage demands.

\subsection{Comparison of adaptive sampling strategies with random search}

In the benchmark sampling strategy, none of the new samples is chosen using environmental information or based on previous iteration solutions. It is a useful benchmark to compare against because it allows us to compare how well the adaptive sampling strategies leverage environmental and coverage information provided to them. We will compare the sampling strategies based on solution time, coverage achieved, and efficiency, given the same environment, resource budget, and sampling budget.

\subsubsection{Final solution comparison}

In this section, we discuss how the final solutions vary amongst the three algorithms. Table \ref{table3:sampling_strategy_summary} presents a comparative analysis of the two adaptive sampling strategies, E\&E and TUS, across the different spaces and resource budgets. It reports key metrics, including the total number of voxels covered (which is our proxy for space coverage), pre-processing time (total time taken to conduct visibility calculations and sample configurations) and runtime (total time taken for optimization), and the range of coverage achieved, highlighting performance variations between high and low resource budgets. The table also includes improvement percentages, showcasing how E\&E and TUS fare relative to RS, providing a comprehensive view of how each strategy balances coverage efficiency and computational cost under different experimental conditions.

The E\&E sampling strategy, optimized with the best set of hyperparameters, consistently achieves higher minimum, maximum, and mean coverage over all repeat trials than the other two strategies, TUS and RS. For instance, in large spaces (type 2) with a low resource budget, the E\&E strategy achieves a 16.0\% improvement over the RS strategy. The range of improvements across all scenarios varies from 3.3\% to 16.0\%, which is significant. Furthermore, given that the number of visibility calculations (the time-consuming step) is the same in both algorithms, the pre-processing time is higher for the E\&E strategy compared to the RS strategy while providing much better coverage. This is expected because of the iterative nature of the sampling strategy.

The results for the TUS strategy indicate that this strategy performs better than the RS strategy in specific contexts: those with a low camera budget and fairly unobstructed environment. It outperforms the benchmark in the large space (type 2) and medium space (type 2) with the low resource budget by 9.2\% and 6.9\%, respectively. In other scenarios, it performs similar to the RS strategy ($\sim1\%$ in all cases except one), suggesting that the TUS strategy does not effectively leverage the environmental information and previous iteration's information in these cases. While comparing runtimes for the TUS strategy, we find that they are comparable to the runtimes obtained for the E\&E strategy, suggesting the E\&E sampling strategy is a better option overall.

In terms of pre-processing time and runtime, the RS strategy generally requires the least pre-processing time, such as in small spaces with a low resource budget where the RS strategy has a pre-processing time of 16.91s compared to 31.13s for E\&E. However, this comes at the cost of lower coverage, as mean coverage for the RS strategy in this scenario is 49.5\%, while E\&E reaches 54.6\% for an improvement of 10\%. E\&E's iterative nature results in longer pre-processing time and runtime but yields better overall coverage.

\begin{sidewaystable}
\centering
\tiny
\caption{Comparison of sampling strategies with different resource budgets and space types. The first four columns describe the 3D environment and the resource budget on which the algorithms have been tested. The fifth column lists the sampling strategies. The sixth column represents voxel coverage with different sampled sets. The rest of the columns provide mean summary statistics for the five different sample runs considered for each strategy, 3D environment, and resource budget.} 

\begin{tabular}{@{}p{1.2cm}p{1cm}p{1.3cm}p{1.5cm}p{1.0cm}p{1.9cm}p{1.3cm}p{1.1cm}p{1.12cm}p{1.1cm}p{1.2cm}p{1.55cm}@{}}

\toprule

\textbf{Space Size} & \textbf{Space Type} & \textbf{Free Space Voxels} & \textbf{Resource Budget Type} & \textbf{Sampling Strategy} & \textbf{Voxels Covered (All Trials)} & \textbf{Mean Pre-processing time (s)} & \textbf{Mean Runtime (s)} & \textbf{Min Coverage} & \textbf{Max Coverage} & \textbf{Mean Coverage} & \textbf{Mean Improvement over Benchmark} \\

\midrule
&  &  &  & RS & [4638, 4814, 4770, 4583, 4692] & 27.18 & 0.67 & 85.0\% & 89.3\% & 87.2\% & - \\
Large & 1 & 5391 & High & E\&E & [4868, 4722, 4857, 4855, 4983] & 42.39 & 3.76 & 87.6\% & 92.4\% & 90.1\% & \textbf{3.33\%} \\
 &  &  &  & TUS & [4575, 4327, 4643, 4575, 4643] & 55.21 & 3.39 & 80.0\% & 86.0\% & 84.4\% & -3.21\% \\
 \midrule
 &  &  &  & RS & [2498, 2736, 2681, 2532, 2705] & 27.14 & 0.87 & 46.3\% & 50.8\% & 48.8\% & - \\
Large & 1 & 5391 & Low & E\&E & [2689, 2888, 2807, 2704, 2899] & 44.97 & 4.80 & 49.9\% & 53.8\% & 51.9\% & \textbf{6.35\%} \\
 &  &  &  & TUS & [2540, 2482, 2513, 2755, 2616] & 54.28 & 3.56 & 46.0\% & 51.0\% & 48.0\% & -1.64\% \\
 \midrule
 &  &  &  & RS & [4695, 4898, 4840, 4666, 4777] & 27.19 & 0.73 & 86.6\% & 90.9\% & 88.6\% & - \\
Large & 2 & 5391 & High & E\&E & [5138, 5129, 5212, 5055, 5185] & 43.70 & 4.00 & 93.8\% & 96.7\% & 95.4\% & \textbf{7.67\%} \\
 &  &  &  & TUS & [4865, 4801, 4865, 4893, 4770] & 55.54 & 3.45 & 88.0\% & 91.0\% & 89.6\% & 1.13\% \\
\midrule
 &  &  &  & RS & [2597, 2912, 2976, 2758, 2877] & 26.42 & 0.93 & 48.2\% & 55.2\% & 52.4\% & - \\
Large & 2 & 5391 & Low & E\&E & [3293, 3310, 3230, 3307, 3258] & 45.03 & 7.32 & 59.9\% & 61.4\% & 60.8\% & \textbf{16.03\%} \\
 &  &  &  & TUS & [3080, 3000, 3111, 3140, 3068] & 54.07 & 3.62 & 56.0\% & 58.0\% & 57.2\% & 9.16\% \\
 \midrule
 &  &  &  & RS & [2530, 2419, 2412, 2546, 2466] & 16.68 & 0.44 & 88.4\% & 93.4\% & 90.7\% & - \\
Medium & 1 & 2727 & High & E\&E & [2549, 2582, 2582, 2641, 2572] & 28.81 & 6.46 & 93.5\% & 96.8\% & 94.8\% & \textbf{4.52\%} \\
 &  &  &  & TUS & [2467, 2289, 2516, 2449, 2527] & 28.25 & 2.40 & 84.0\% & 93.0\% & 89.8\% & -0.99\% \\
 \midrule
 &  &  &  & RS & [1331, 1378, 1358, 1401, 1282] & 16.91 & 0.46 & 47.0\% & 51.4\% & 49.5\% & - \\
Medium & 1 & 2727 & Low & E\&E & [1471, 1587, 1529, 1496, 1364] & 31.13 & 7.69 & 50.0\% & 58.2\% & 54.6\% & \textbf{10.30\%} \\
 &  &  &  & TUS & [1407, 1402, 1373, 1273, 1366] & 28.35 & 2.38 & 47.0\% & 52.0\% & 50.0\% & 1.01\% \\
 \midrule
 &  &  &  & RS & [2534, 2509, 2393, 2584, 2521] & 16.84 & 0.43 & 87.8\% & 94.8\% & 92.0\% & - \\
Medium & 2 & 2727 & High & E\&E & [2616, 2657, 2643, 2646, 2588] & 28.80 & 7.29 & 94.9\% & 97.4\% & 96.4\% & \textbf{4.78\%} \\
 &  &  &  & TUS & [2543, 2447, 2438, 2520, 2537] & 27.80 & 2.14 & 89.0\% & 93.0\% & 91.4\% & -0.65\% \\
  \midrule
 &  &  &  & RS & [1335, 1502, 1316, 1529, 1433] & 16.72 & 0.50 & 48.3\% & 56.1\% & 52.2\% & - \\
Medium & 2 & 2727 & Low & E\&E & [1365, 1668, 1615, 1676, 1584] & 31.43 & 9.21 & 50.1\% & 61.5\% & 58.0\% & \textbf{11.11\%} \\
 &  &  &  & TUS & [1622, 1517, 1468, 1412, 1582] & 28.45 & 2.58 & 52.0\% & 59.0\% & 55.8\% & 6.90\% \\
 \midrule
\end{tabular}

\label{table3:sampling_strategy_summary}
\end{sidewaystable}

\subsubsection{Efficiency in early improvement}

In several optimization scenarios, the efficiency of an algorithm is not solely determined by the final solution but also by how quickly it approaches competitive results. In this context, we examine the performance of the E\&E and TUS strategy in relation to the RS benchmark. Specifically, we focus on their ability to surpass the best solution provided by the benchmark given a limited computational budget. While it does not necessarily indicate convergence to the true optimal solution, it highlights the strategies' capacities to achieve superior solutions with reduced computational effort. This analysis provides valuable insights into the early performance of the algorithms and their practical applicability in settings with low sampling budgets.

Figure \ref{fig10:early-improvements-efficiency} displays the results for early improvements in efficiency. These results have been plotted over all four environments and for both high and low resource budget scenarios as well, benchmarked against the RS strategy. The TUS strategy has been plotted only for low budget and type 2 scenarios, because those are the only scenarios under which overall coverage with TUS is better than under the RS strategy, so coverage efficiency can be described under those conditions only. Further, it should be noted that the RS benchmark is not iterative and utilizes all of its sampling budget at once, so its results appear as horizontal (orange) lines.

\begin{figure}[H]
    \centering \includegraphics[width=0.9\textwidth]{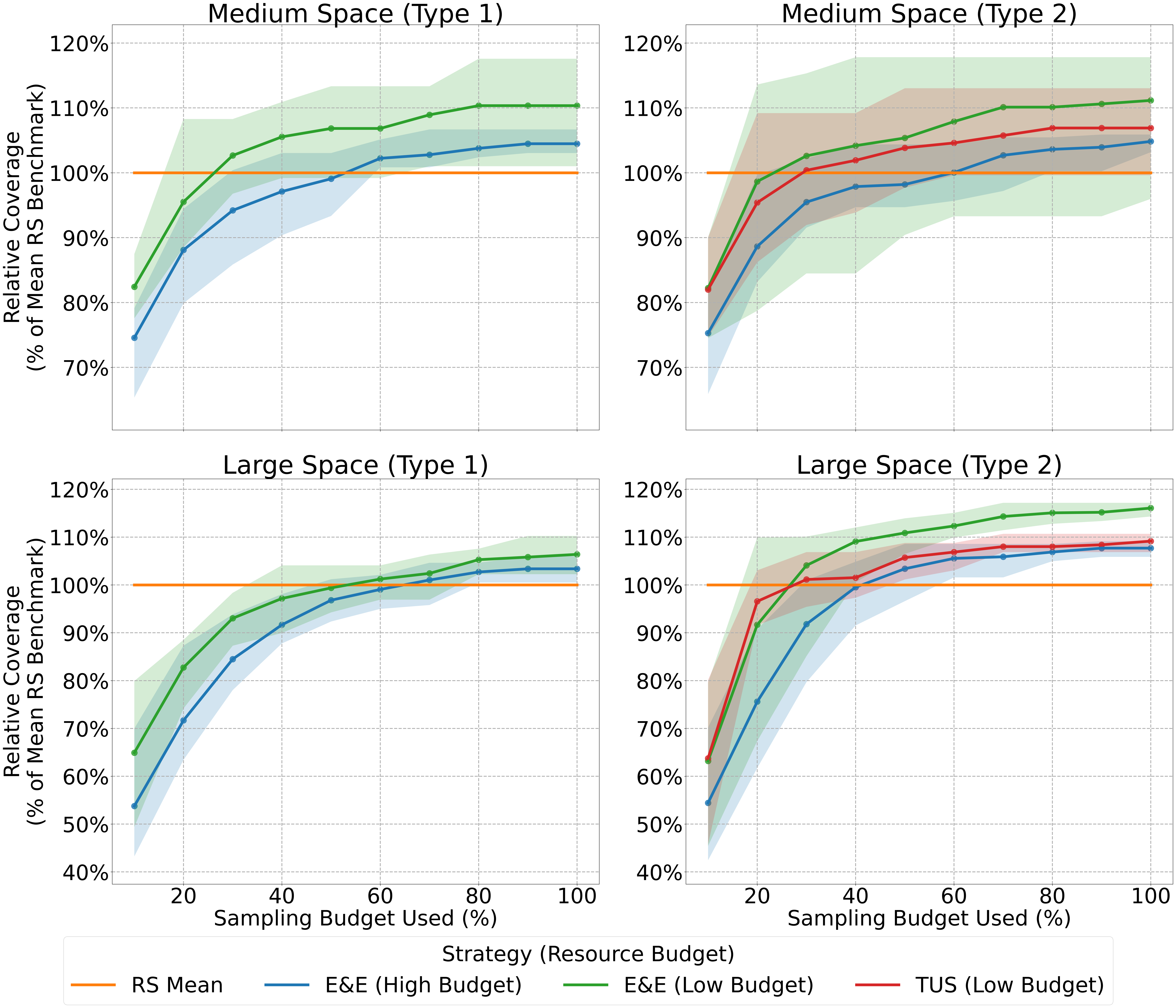}  \caption{This is a $2 \times 2$ panel of the four different environments, with the graphs displaying improvements over the RS benchmark. The mean RS benchmark has been normalized to 100\% for low and high resource budgets. Line plots of the strategies and the conditions under which they perform better than the RS benchmark have been plotted here. It should be noted that the RS benchmark is not iterative and utilizes all of its sampling budget at once. The x-axis represent the percentage of the sampling budget that has been exhausted, whereas the y-axis provides the improvement over the mean RS benchmark achieved. It should be noted that the strategies with the high resource budget are being benchmarked against the RS strategy with the high resource budget, and similarly for strategies with the low resource budget. Each scenario uses five seeds, with shaded areas indicating the ranges of results over the five trials.
    }
    \label{fig10:early-improvements-efficiency} 
\end{figure}

Two key findings emerge from the results depicted in Figure \ref{fig10:early-improvements-efficiency}. 1) The final mean coverage of the E\&E and TUS strategies is significantly better than that of the RS benchmark. 2) Under all environments and resource budgets displayed, our adaptive strategies overtake the mean RS benchmark within 30-70\% of the sampling budget utilized by the RS benchmark. In other words, by using the E\&E or TUS adaptive sampling approach, it is possible to achieve the same coverage as the RS benchmark using only about half the total number of camera configurations sampled.

In the medium space (type 1), even the minimum coverage achieved by the E\&E strategy in both scenarios including high and low resource budgets beats the RS benchmark within 50\% of the original sampling budget expended. For the medium space (type 2), there is one particular trial under a low resource budget that increases the size of the shaded area for the E\&E strategy. Otherwise, the mean coverage of the E\&E strategy is much higher. This outlier trial can be observed in Table \ref{table3:sampling_strategy_summary}.

For the large space (type 1), the E\&E strategy overtakes the mean RS benchmark after expending 70\% of the sampling budget utilized by the RS benchmark. However, the coverage improvement is lower than the one observed for the large space (type 2). In the large space (type 2), the E\&E strategy performs much better than the RS benchmark, with a 16\% improvement under low resource budget conditions overtaking the benchmark strategy in 30\% of the sampling budget, and 8\% under high resource budget conditions overtaking the benchmark strategy in 50\% of the sampling budget. The TUS strategy also provides a 9\% improvement over the RS benchmark under low budget conditions, overtaking it in 40\% of the sampling budget.

\subsection{Performance of the greedy heuristic}

The greedy heuristic method functionally performs as well as the IP model. Both methods were tested on 8 different problems using 5 different seeds, i.e. a total of 40 instances. This includes running the algorithm for two different resource budgets for 4 different indoor spaces. 

The greedy heuristic found a marginally worse solution than the IP in 14 of the 40 instances ($\sim1\%$). This is conjectured to be attributed to the problem structure, which aims to cover as many voxels as possible only once. Since the free space points are not redundantly covered by multiple camera configurations, there is minimal risk of overlapping configurations conflicting, reducing the chances of the greedy heuristic yielding suboptimal results. Our conjecture is backed by the fact that most of the solutions with differences arise in cases where the resource budget is high. When the resource budget is high, the free space points are more likely to be covered by redundant configurations and thus the greedy heuristic likely performs a little worse.

There are two major steps in the algorithm that impact the time spent optimizing the camera placement: pre-processing visibility calculations and solving the instance (solution time). Typically, the visibility calculations consume the most time and are comparable in all algorithms as it is dependent on the number of sample camera configurations, which we control for. Table \ref{table4:greedy-vs-ip} provides a comparison between the two models. For large environments, the mean pre-processing time is 27.47 seconds for the greedy heuristic and 26.98 seconds for the IP, while the mean model run time is 0.60 seconds for the heuristic and 0.80 seconds for the IP. In smaller environments, the mean pre-processing times are nearly identical, with 16.77 seconds for the heuristic and 16.79 seconds for the IP, while the mean model run times are 0.57 seconds for the heuristic and 0.46 seconds for the IP.

Given these comparable runtimes, we continue utilizing the IP solution whenever feasible, acknowledging that the greedy heuristic provides a valuable alternative when the model becomes too large for the optimization library to solve due to the presence of a high number of integer variables. For significantly larger problems where the number of integer variables exceeds the solver's capacity -- since it is an NP-hard problem -- we can effectively use the greedy heuristic as a reliable proxy for the IP solution.

Alternatively, we can enhance our algorithmic framework by warm-starting the solution with the heuristic and then refining it with IP to obtain a provably optimal solution.

\begin{sidewaystable}
    \centering
    \tiny
\caption{Comparison of different optimization models with different resource budgets and space types. The first 4 columns describe the 3D environment and the resource budget on which the algorithms have been tested. The fifth column lists the optimization strategy, Greedy or IP. The sixth column represents voxel coverage with different sampled sets. The rest of the columns provide mean summary statistics for the five different sample runs considered for each model, 3D environment, and resource budget.}
    
    \begin{tabular}{@{}p{1.2cm}p{1cm}p{1.3cm}p{1.5cm}p{1.2cm}p{1.9cm}p{1.3cm}p{1.1cm}p{1.12cm}p{1.1cm}p{1.2cm}p{1.55cm}@{}}

    \toprule

\textbf{Space Size} & \textbf{Space Type} & \textbf{Free Space Voxels} & \textbf{Resource Budget Type} & \textbf{Optimization Model Type} & \textbf{Voxels Covered (All Trials)} & \textbf{Mean Pre-processing time (s)} & \textbf{Mean Runtime (s)} & \textbf{Min Coverage} & \textbf{Max Coverage} & \textbf{Mean Coverage} & \textbf{Mean Improvement over Benchmark} \\

    \midrule

        Large & 1 & 5391 & High & Greedy & [4631, 4774, 4731, 4583, 4692] & 27.48 & 0.61 & 85.0\% & 88.6\% & 86.9\% & - \\
        & & &  & IP & [4638, 4814, 4770, 4583, 4692] & 27.18 & 0.67 & 85.0\% & 89.3\% & 87.2\% & 0.37\% \\
    \midrule    
        Large & 1 & 5391 & Low &  Greedy & [2498, 2736, 2681, 2532, 2705] & 27.88 & 0.57 & 46.3\% & 50.8\% & 48.8\% & - \\
        & & &  & IP & [2498, 2736, 2681, 2532, 2705] & 27.14 & 0.87 & 46.3\% & 50.8\% & 48.8\% & 0.00\% \\
        \midrule
        Large & 2 & 5391 & High & Greedy & [4695, 4873, 4840, 4609, 4777] & 26.92 & 0.62 & 85.5\% & 90.4\% & 88.3\% & - \\
        & & &  & IP & [4695, 4898, 4840, 4666, 4777] & 27.19 & 0.73 & 86.6\% & 90.9\% & 88.6\% & 0.34\% \\
    \midrule    
        Large & 2 & 5391 & Low & Greedy & [2597, 2912, 2976, 2746, 2877] & 27.61 & 0.61 & 48.2\% & 55.2\% & 52.3\% & - \\
        & & &  & IP & [2597, 2912, 2976, 2758, 2877] & 26.42 & 0.93 & 48.2\% & 55.2\% & 52.4\% & 0.08\% \\
        \midrule
        Medium & 1 & 2727 & High & Greedy & [2525, 2382, 2409, 2481, 2416] & 16.54 & 0.57 & 87.3\% & 92.6\% & 89.6\% & - \\
        & & &  & IP & [2530, 2419, 2412, 2546, 2466] & 16.68 & 0.44 & 88.4\% & 93.4\% & 90.7\% & 1.29\% \\
    \midrule    
        Medium & 1 & 2727 & Low & Greedy & [1331, 1378, 1358, 1401, 1282] & 16.5 & 0.55 & 47.0\% & 51.4\% & 49.5\% & - \\
        & & &  & IP & [1331, 1378, 1358, 1401, 1282] & 16.91 & 0.46 & 47.0\% & 51.4\% & 49.5\% & 0.00\% \\
        \midrule
        Medium & 2 & 2727 & High & Greedy & [2524, 2460, 2354, 2584, 2521] & 16.72 & 0.58 & 86.3\% & 94.8\% & 91.3\% & - \\
        & & &  & IP & [2534, 2509, 2393, 2584, 2521] & 16.84 & 0.43 & 87.8\% & 94.8\% & 92.0\% & 0.78\% \\
    \midrule    
        Medium & 2 & 2727 & Low & Greedy & [1335, 1502, 1316, 1529, 1433] & 17.33 & 0.57 & 48.3\% & 56.1\% & 52.2\% & - \\
        & & &  & IP & [1335, 1502, 1316, 1529, 1433] & 16.72 & 0.50 & 48.3\% & 56.1\% & 52.2\% & 0.00\% \\
        \midrule
    \end{tabular}

    \label{table4:greedy-vs-ip}
\end{sidewaystable}

\section{Conclusions} \label{section:8-conclusions}

In this study, we created a consistent, flexible, and efficient framework that works verifiably well to solve the optimal camera placement problem.

We developed and employed a novel flood-filling algorithm to efficiently determine which free-space voxels are visible within the viewing frustum of a given camera configuration. Only voxels near visible ones are added to the visibility stack, enabling us to efficiently handle visibility, which is primarily a connected set of voxels. This approach significantly reduces computations by omitting voxels outside the frustum or occluded by obstacles, bypassing costly obstruction checks for non-visible regions. The algorithm thus ensures computational efficiency while maintaining accuracy in visibility determination. We introduced two novel adaptive sampling strategies for optimizing camera placement: the Explore and Exploit (E\&E) strategy and the Target Uncovered Spaces (TUS) strategy. These strategies were designed to maximize the coverage of free space in complex indoor environments, utilizing environmental information in an iterative approach, while adhering to constraints such as limited camera budgets, partial visibility due to obstructions, and space restrictions. By combining random sampling with informed sampling based on environmental information and the previous iteration's solution, we aimed to balance the need for broad exploration of possible solutions with focused improvements. Both of these strategies are highly flexible with a large set of hyperparameters. We tested the models on custom and sanitized environments to find the best hyperparameters, which were then used to solve the case study problem and test algorithm performance against benchmarks.

The E\&E approach demonstrated robust performance across a range of scenarios, consistently outperforming the random baseline in terms of coverage while maintaining a manageable computational time. Its ability to leverage both exploratory and exploitative samples allowed it to adapt effectively to diverse room layouts and resource budgets, leading to improved coverage results. It provides an overall improvement of anywhere from 3.3\% to 16.0\% over the RS benchmark, based on the specific scenario and the resource budget. In contrast, the TUS strategy showed strengths in environments with open spaces and limited resource budgets, making it suitable for scenarios where maximizing the visibility of sparse areas is critical with coverage improvements of 6.9\% to 9.1\% under those specific conditions. Furthermore, we found that the strategies can be used to outperform the RS benchmark with a limited sampling budget, exhibiting their efficiency. Specifically, with the E\&E and TUS adaptive sampling strategies, similar coverage to that achieved by the random benchmark can be realized using only 30-70\% of the total sampling budget.

Under our framework, we also tested the modified maximum $k$-coverage model that we used for optimization against a greedy heuristic that is solvable in pseudopolynomial time. We found that the heuristic performs very well and could be utilized as a reliable proxy for the IP model. It allows us to augment the framework, such that we can utilize the greedy heuristic to warm-start the IP model or as a proxy for the IP solution, in cases where solving the IP model is computationally impractical.

We conducted a theoretical analysis to get insights into the algorithmic framework, offering bounds on the likelihood of finding an optimal solution over successive iterations and the expected number of samples required to find a solution, based on budget constraints. Furthermore, our optimization framework ensures monotonicity and constantly improving coverage with successive iterations.

We analyzed a case study based on a real-world indoor surveillance application to obtain insights into optimal camera placement and compare the model's solution with conventional camera placement wisdom. Overall, we found that placing cameras at the corners or along edges is not necessary and can lead to worse results, as compared to keeping them along the faces of the walls, and cameras typically do cover more space when placed near the ceiling than the floor. Further, we found that under low resource budgets, the means to good coverage is through central location of cameras along corridors and doorways, ignoring areas in the corner as they lead to worse overall coverage. Under high resource budgets, we find that a good way to improve coverage is for the cameras in the solution network to target large open spaces or rooms, while keeping their peripheral vision directed at corridors and doorways to maximize coverage from any camera. 

Our research points to several valuable directions for future work. 1) Termination criteria: it is still an open question on when to stop searching for a better solution. We work under the presumption of a fixed sampling budget and exhaust it to find the best solution, but gaining insights on termination criteria would be useful. 2) Our framework is applicable beyond indoor environments, and we would like to test it in diverse conditions to establish whether it continues to perform well in other contexts. 3) We conducted our analysis with all free space voxels having equal reward for coverage and each camera configuration having equal cost, but that is not necessarily the case in practical situations. Our framework is general enough that we can modify it for these different conditions, and we consider expanding into looking at how our solution strategies work under those situations.

We further aim to improve the selected camera network using an approximately continuous local optimization technique. There are some models that have implemented approaches along similar lines. In \cite{smith_aerial_2018}, once the approximate position for the camera network is calculated, the Nelder-Mead method is used to improve the configuration of the camera network further based on a 3D image reconstruction heuristical objective. In \cite{sun_learning_2021}, a neural network is utilized in a similar manner where the network is trained to calculate coverage based on the camera network configuration, and then gradient descent steps are taken backpropagating to the input to increase coverage. Under our proposed method, we intend to increase the granularity of the discrete space model and then use DFO-based methods to take steps from the discrete model's optimal solution to help improve the solution to reach near a true local optimum.

\section*{Declarations}

\subsection{Competing interests}

The authors declare that they have no competing interests, either financial or non-financial, related to the work submitted for publication.

\subsection{Funding}

This work was supported by the Office of the Vice President for Research, Scholarship and Creative Endeavors at The University of Texas at Austin.

\newpage

\bibliographystyle{sn-mathphys-ay-modified}
\bibliography{references}


\begin{thebibliography}{66}
\ifx \bisbn   \undefined \def \bisbn  #1{ISBN #1}\fi
\ifx \binits  \undefined \def \binits#1{#1}\fi
\ifx \bauthor  \undefined \def \bauthor#1{#1}\fi
\ifx \batitle  \undefined \def \batitle#1{#1}\fi
\ifx \bjtitle  \undefined \def \bjtitle#1{#1}\fi
\ifx \bvolume  \undefined \def \bvolume#1{\textbf{#1}}\fi
\ifx \byear  \undefined \def \byear#1{#1}\fi
\ifx \bissue  \undefined \def \bissue#1{#1}\fi
\ifx \bfpage  \undefined \def \bfpage#1{#1}\fi
\ifx \blpage  \undefined \def \blpage #1{#1}\fi
\ifx \burl  \undefined \def \burl#1{\textsf{#1}}\fi
\ifx \doiurl  \undefined \def \doiurl#1{\url{https://doi.org/#1}}\fi
\ifx \betal  \undefined \def \betal{\textit{et al.}}\fi
\ifx \binstitute  \undefined \def \binstitute#1{#1}\fi
\ifx \binstitutionaled  \undefined \def \binstitutionaled#1{#1}\fi
\ifx \bctitle  \undefined \def \bctitle#1{#1}\fi
\ifx \beditor  \undefined \def \beditor#1{#1}\fi
\ifx \bpublisher  \undefined \def \bpublisher#1{#1}\fi
\ifx \bbtitle  \undefined \def \bbtitle#1{#1}\fi
\ifx \bedition  \undefined \def \bedition#1{#1}\fi
\ifx \bseriesno  \undefined \def \bseriesno#1{#1}\fi
\ifx \blocation  \undefined \def \blocation#1{#1}\fi
\ifx \bsertitle  \undefined \def \bsertitle#1{#1}\fi
\ifx \bsnm \undefined \def \bsnm#1{#1}\fi
\ifx \bsuffix \undefined \def \bsuffix#1{#1}\fi
\ifx \bparticle \undefined \def \bparticle#1{#1}\fi
\ifx \barticle \undefined \def \barticle#1{#1}\fi
\bibcommenthead
\ifx \bconfdate \undefined \def \bconfdate #1{#1}\fi
\ifx \botherref \undefined \def \botherref #1{#1}\fi
\ifx \url \undefined \def \url#1{\textsf{#1}}\fi
\ifx \bchapter \undefined \def \bchapter#1{#1}\fi
\ifx \bbook \undefined \def \bbook#1{#1}\fi
\ifx \bcomment \undefined \def \bcomment#1{#1}\fi
\ifx \oauthor \undefined \def \oauthor#1{#1}\fi
\ifx \citeauthoryear \undefined \def \citeauthoryear#1{#1}\fi
\ifx \endbibitem  \undefined \def \endbibitem {}\fi
\ifx \bconflocation  \undefined \def \bconflocation#1{#1}\fi
\ifx \arxivurl  \undefined \def \arxivurl#1{\textsf{#1}}\fi
\csname PreBibitemsHook\endcsname

\bibitem[\protect\citeauthoryear{Abrahamsen et~al.}{2022}]{abrahamsen_art_2022}
\begin{barticle}
\bauthor{\bsnm{Abrahamsen}, \binits{M.}},
\bauthor{\bsnm{Adamaszek}, \binits{A.}},
\bauthor{\bsnm{Miltzow}, \binits{T.}}
(\byear{2022})
\batitle{The {Art} {Gallery} {Problem} is {R}-complete}.
\bjtitle{\emph{Journal of the ACM}}
\bvolume{69}(\bissue{1}),
\bfpage{1}--\blpage{70}
\doiurl{10.1145/3486220}
\end{barticle}
\endbibitem

\bibitem[\protect\citeauthoryear{Alihodzic et~al.}{2020}]{alihodzic_exact_2020}
\begin{bchapter}
\bauthor{\bsnm{Alihodzic}, \binits{A.}},
\bauthor{\bsnm{Delalic}, \binits{S.}},
\bauthor{\bsnm{Hasic}, \binits{D.}}
(\byear{2020})
\bctitle{An {Exact} {Two}-{Phase} {Method} {For} {Optimal} {Camera} {Placement} {In} {Art} {Gallery} {Problem}}.
\bbtitle{\emph{2020 15th {Conference} on {Computer} {Science} and {Information} {Systems} ({FedCSIS})}},
pp. \bfpage{215}--\blpage{222}.
\doiurl{10.15439/2020F79}
\end{bchapter}
\endbibitem

\bibitem[\protect\citeauthoryear{Ali and Hassanein}{2021}]{ali_optimal_2021}
\begin{bchapter}
\bauthor{\bsnm{Ali}, \binits{A.}},
\bauthor{\bsnm{Hassanein}, \binits{H.S.}}
(\byear{2021})
\bctitle{Optimal {Placement} of {Camera} {Wireless} {Sensors} in {Greenhouses}}.
\bbtitle{\emph{{ICC} 2021 - {IEEE} {International} {Conference} on {Communications}}},
pp. \bfpage{1}--\blpage{6}.
\doiurl{10.1109/ICC42927.2021.9500727} .
\bcomment{ISSN: 1938-1883}
\end{bchapter}
\endbibitem

\bibitem[\protect\citeauthoryear{Andrews}{1998}]{andrews_special_1998}
\begin{bbook}
\bauthor{\bsnm{Andrews}, \binits{L.C.}}
(\byear{1998})
\bbtitle{Special {Functions} of {Mathematics} for {Engineers}}.
\bsertitle{{SPIE} {Press} {Monograph}},
vol. \bseriesno{PM49}.
\bpublisher{SPIE Press},
\blocation{Bellingham, WA}
\end{bbook}
\endbibitem

\bibitem[\protect\citeauthoryear{Aissaoui et~al.}{2018}]{aissaoui_designing_2018}
\begin{barticle}
\bauthor{\bsnm{Aissaoui}, \binits{A.}},
\bauthor{\bsnm{Ouafi}, \binits{A.}},
\bauthor{\bsnm{Pudlo}, \binits{P.}},
\bauthor{\bsnm{Gillet}, \binits{C.}},
\bauthor{\bsnm{Baarir}, \binits{Z.-E.}},
\bauthor{\bsnm{Taleb-Ahmed}, \binits{A.}}
(\byear{2018})
\batitle{Designing a camera placement assistance system for human motion capture based on a guided genetic algorithm}.
\bjtitle{\emph{Virtual Reality}}
\bvolume{22}(\bissue{1}),
\bfpage{13}--\blpage{23}
\doiurl{10.1007/s10055-017-0310-7}
\end{barticle}
\endbibitem

\bibitem[\protect\citeauthoryear{Angella et~al.}{2007}]{angella_optimal_2007}
\begin{bchapter}
\bauthor{\bsnm{Angella}, \binits{F.}},
\bauthor{\bsnm{Reithler}, \binits{L.}},
\bauthor{\bsnm{Gallesio}, \binits{F.}}
(\byear{2007})
\bctitle{Optimal deployment of cameras for video surveillance systems}.
\bbtitle{\emph{2007 {IEEE} {Conference} on {Advanced} {Video} and {Signal} {Based} {Surveillance}}},
pp. \bfpage{388}--\blpage{392}.
\doiurl{10.1109/AVSS.2007.4425342}
\end{bchapter}
\endbibitem

\bibitem[\protect\citeauthoryear{Andersen and Tirthapura}{2009}]{andersen_wireless_2009}
\begin{bchapter}
\bauthor{\bsnm{Andersen}, \binits{T.}},
\bauthor{\bsnm{Tirthapura}, \binits{S.}}
(\byear{2009})
\bctitle{Wireless sensor deployment for {3D} coverage with constraints}.
\bbtitle{\emph{2009 {Sixth} {International} {Conference} on {Networked} {Sensing} {Systems} ({INSS})}},
pp. \bfpage{1}--\blpage{4}.
\doiurl{10.1109/INSS.2009.5409946}
\end{bchapter}
\endbibitem

\bibitem[\protect\citeauthoryear{Bodor et~al.}{2007}]{bodor_optimal_2007}
\begin{barticle}
\bauthor{\bsnm{Bodor}, \binits{R.}},
\bauthor{\bsnm{Drenner}, \binits{A.}},
\bauthor{\bsnm{Schrater}, \binits{P.}},
\bauthor{\bsnm{Papanikolopoulos}, \binits{N.}}
(\byear{2007})
\batitle{Optimal {Camera} {Placement} for {Automated} {Surveillance} {Tasks}}.
\bjtitle{\emph{Journal of Intelligent and Robotic Systems}}
\bvolume{50}(\bissue{3}),
\bfpage{257}--\blpage{295}
\doiurl{10.1007/s10846-007-9164-7}
\end{barticle}
\endbibitem

\bibitem[\protect\citeauthoryear{Bai et~al.}{2024}]{bai_assessment_2024}
\begin{barticle}
\bauthor{\bsnm{Bai}, \binits{Y.}},
\bauthor{\bsnm{Demir}, \binits{A.}},
\bauthor{\bsnm{Yilmaz}, \binits{A.}},
\bauthor{\bsnm{Sezen}, \binits{H.}}
(\byear{2024})
\batitle{Assessment and monitoring of bridges using various camera placements and structural analysis}.
\bjtitle{\emph{Journal of Civil Structural Health Monitoring}}
\bvolume{14}(\bissue{2}),
\bfpage{321}--\blpage{337}
\doiurl{10.1007/s13349-023-00720-6}
\end{barticle}
\endbibitem

\bibitem[\protect\citeauthoryear{Barnhart et~al.}{1998}]{barnhart_branch-and-price_1998}
\begin{barticle}
\bauthor{\bsnm{Barnhart}, \binits{C.}},
\bauthor{\bsnm{Johnson}, \binits{E.L.}},
\bauthor{\bsnm{Nemhauser}, \binits{G.L.}},
\bauthor{\bsnm{Savelsbergh}, \binits{M.W.P.}},
\bauthor{\bsnm{Vance}, \binits{P.H.}}
(\byear{1998})
\batitle{Branch-and-{Price}: {Column} {Generation} for {Solving} {Huge} {Integer} {Programs}}.
\bjtitle{\emph{Operations Research}}
\bvolume{46}(\bissue{3}),
\bfpage{316}--\blpage{329}
\doiurl{10.1287/opre.46.3.316}
\end{barticle}
\endbibitem

\bibitem[\protect\citeauthoryear{Brusco et~al.}{1999}]{brusco_morphing_1999}
\begin{barticle}
\bauthor{\bsnm{Brusco}, \binits{M.J.}},
\bauthor{\bsnm{Jacobs}, \binits{L.W.}},
\bauthor{\bsnm{Thompson}, \binits{G.M.}}
(\byear{1999})
\batitle{A morphing procedure to supplement a simulated annealing heuristic for cost‐ andcoverage‐correlated set‐covering problems}.
\bjtitle{\emph{Annals of Operations Research}}
\bvolume{86}(\bissue{0}),
\bfpage{611}--\blpage{627}
\doiurl{10.1023/A:1018900128545}
\end{barticle}
\endbibitem

\bibitem[\protect\citeauthoryear{Bottino}{2009}]{bottino_towards_2009}
\begin{bchapter}
\bauthor{\bsnm{Bottino}, \binits{A.}}
(\byear{2009})
\bctitle{Towards an {Iterative} {Algorithm} for the {Optimal} {Boundary} {Coverage} of a {3D} {Environment}}\beditor{\bsnm{Bayro-Corrochano}, \binits{E.}},
\beditor{\bsnm{Eklundh}, \binits{J.-O.}} (eds.)
\bbtitle{Progress in {Pattern} {Recognition}, {Image} {Analysis}, {Computer} {Vision}, And {Applications}},
pp. \bfpage{707}--\blpage{715}.
\bpublisher{Springer},
\blocation{Berlin, Heidelberg}.
\doiurl{10.1007/978-3-642-10268-4_83}
\end{bchapter}
\endbibitem

\bibitem[\protect\citeauthoryear{Balas and Padberg}{1972}]{balas_set-covering_1972}
\begin{barticle}
\bauthor{\bsnm{Balas}, \binits{E.}},
\bauthor{\bsnm{Padberg}, \binits{M.W.}}
(\byear{1972})
\batitle{On the {Set}-{Covering} {Problem}}.
\bjtitle{\emph{Operations Research}}
\bvolume{20}(\bissue{6}),
\bfpage{1152}--\blpage{1161}
\doiurl{10.1287/opre.20.6.1152}
\end{barticle}
\endbibitem

\bibitem[\protect\citeauthoryear{Bautista and Pereira}{2007}]{bautista_grasp_2007}
\begin{barticle}
\bauthor{\bsnm{Bautista}, \binits{J.}},
\bauthor{\bsnm{Pereira}, \binits{J.}}
(\byear{2007})
\batitle{A {GRASP} algorithm to solve the unicost set covering problem}.
\bjtitle{\emph{Computers \& Operations Research}}
\bvolume{34}(\bissue{10}),
\bfpage{3162}--\blpage{3173}
\doiurl{10.1016/j.cor.2005.11.026}
\end{barticle}
\endbibitem

\bibitem[\protect\citeauthoryear{Chaudhary and Chaturvedi}{2017}]{chaudhary_observing_2017}
\begin{barticle}
\bauthor{\bsnm{Chaudhary}, \binits{A.S.}},
\bauthor{\bsnm{Chaturvedi}, \binits{D.K.}}
(\byear{2017})
\batitle{Observing hotspots and power loss in solar photovoltaic array under shading effects using thermal imaging camera}.
\bjtitle{\emph{Int. J. Electr. Mach. Drives}}
\bvolume{3}(\bissue{1}),
\bfpage{15}--\blpage{23}
\end{barticle}
\endbibitem

\bibitem[\protect\citeauthoryear{Carrabs et~al.}{2024}]{carrabs_solving_2024}
\begin{barticle}
\bauthor{\bsnm{Carrabs}, \binits{F.}},
\bauthor{\bsnm{Cerulli}, \binits{R.}},
\bauthor{\bsnm{Mansini}, \binits{R.}},
\bauthor{\bsnm{Moreschini}, \binits{L.}},
\bauthor{\bsnm{Serra}, \binits{D.}}
(\byear{2024})
\batitle{Solving the {Set} {Covering} {Problem} with {Conflicts} on {Sets}: {A} new parallel {GRASP}}.
\bjtitle{\emph{Computers \& Operations Research}}
\bvolume{166},
\bfpage{106620}
\doiurl{10.1016/j.cor.2024.106620}
\end{barticle}
\endbibitem

\bibitem[\protect\citeauthoryear{Chrissis et~al.}{1982}]{chrissis_dynamic_1982}
\begin{barticle}
\bauthor{\bsnm{Chrissis}, \binits{J.W.}},
\bauthor{\bsnm{Davis}, \binits{R.P.}},
\bauthor{\bsnm{Miller}, \binits{D.M.}}
(\byear{1982})
\batitle{The dynamic set covering próblem}.
\bjtitle{\emph{Applied Mathematical Modelling}}
\bvolume{6}(\bissue{1}),
\bfpage{2}--\blpage{6}
\doiurl{10.1016/S0307-904X(82)80015-2}
\end{barticle}
\endbibitem

\bibitem[\protect\citeauthoryear{Chvátal}{1975}]{chvatal_combinatorial_1975}
\begin{barticle}
\bauthor{\bsnm{Chvátal}, \binits{V.}}
(\byear{1975})
\batitle{A combinatorial theorem in plane geometry}.
\bjtitle{\emph{Journal of Combinatorial Theory, Series B}}
\bvolume{18}(\bissue{1}),
\bfpage{39}--\blpage{41}
\doiurl{10.1016/0095-8956(75)90061-1}
\end{barticle}
\endbibitem

\bibitem[\protect\citeauthoryear{Cheng et~al.}{2008}]{cheng_time-optimal_2008}
\begin{bchapter}
\bauthor{\bsnm{Cheng}, \binits{P.}},
\bauthor{\bsnm{Keller}, \binits{J.}},
\bauthor{\bsnm{Kumar}, \binits{V.}}
(\byear{2008})
\bctitle{Time-optimal {UAV} trajectory planning for {3D} urban structure coverage}.
\bbtitle{\emph{2008 {IEEE}/{RSJ} {International} {Conference} on {Intelligent} {Robots} and {Systems}}},
pp. \bfpage{2750}--\blpage{2757}.
\doiurl{10.1109/IROS.2008.4650988} .
\bcomment{ISSN: 2153-0866}
\end{bchapter}
\endbibitem

\bibitem[\protect\citeauthoryear{Castaño et~al.}{2014}]{castano_column_2014}
\begin{barticle}
\bauthor{\bsnm{Castaño}, \binits{F.}},
\bauthor{\bsnm{Rossi}, \binits{A.}},
\bauthor{\bsnm{Sevaux}, \binits{M.}},
\bauthor{\bsnm{Velasco}, \binits{N.}}
(\byear{2014})
\batitle{A column generation approach to extend lifetime in wireless sensor networks with coverage and connectivity constraints}.
\bjtitle{\emph{Computers \& Operations Research}}
\bvolume{52},
\bfpage{220}--\blpage{230}
\doiurl{10.1016/j.cor.2013.11.001}
\end{barticle}
\endbibitem

\bibitem[\protect\citeauthoryear{Caprara et~al.}{2000}]{caprara_algorithms_2000}
\begin{barticle}
\bauthor{\bsnm{Caprara}, \binits{A.}},
\bauthor{\bsnm{Toth}, \binits{P.}},
\bauthor{\bsnm{Fischetti}, \binits{M.}}
(\byear{2000})
\batitle{Algorithms for the {Set} {Covering} {Problem}}.
\bjtitle{\emph{Annals of Operations Research}}
\bvolume{98}(\bissue{1}),
\bfpage{353}--\blpage{371}
\doiurl{10.1023/A:1019225027893}
\end{barticle}
\endbibitem

\bibitem[\protect\citeauthoryear{Costa et~al.}{2017}]{costa_enhancing_2017}
\begin{barticle}
\bauthor{\bsnm{Costa}, \binits{D.G.}},
\bauthor{\bsnm{Vasques}, \binits{F.}},
\bauthor{\bsnm{Portugal}, \binits{P.}}
(\byear{2017})
\batitle{Enhancing the availability of wireless visual sensor networks: {Selecting} redundant nodes in networks with occlusion}.
\bjtitle{\emph{Applied Mathematical Modelling}}
\bvolume{42},
\bfpage{223}--\blpage{243}
\doiurl{10.1016/j.apm.2016.10.008}
\end{barticle}
\endbibitem

\bibitem[\protect\citeauthoryear{Elloumi et~al.}{2021}]{elloumi_optimization_2021}
\begin{barticle}
\bauthor{\bsnm{Elloumi}, \binits{S.}},
\bauthor{\bsnm{Hudry}, \binits{O.}},
\bauthor{\bsnm{Marie}, \binits{E.}},
\bauthor{\bsnm{Martin}, \binits{A.}},
\bauthor{\bsnm{Plateau}, \binits{A.}},
\bauthor{\bsnm{Rovedakis}, \binits{S.}}
(\byear{2021})
\batitle{Optimization of wireless sensor networks deployment with coverage and connectivity constraints}.
\bjtitle{\emph{Annals of Operations Research}}
\bvolume{298}(\bissue{1}),
\bfpage{183}--\blpage{206}
\doiurl{10.1007/s10479-018-2943-7}
\end{barticle}
\endbibitem

\bibitem[\protect\citeauthoryear{Erdem and Sclaroff}{2006}]{erdem_automated_2006}
\begin{barticle}
\bauthor{\bsnm{Erdem}, \binits{U.M.}},
\bauthor{\bsnm{Sclaroff}, \binits{S.}}
(\byear{2006})
\batitle{Automated camera layout to satisfy task-specific and floor plan-specific coverage requirements}.
\bjtitle{\emph{Computer Vision and Image Understanding}}
\bvolume{103}(\bissue{3}),
\bfpage{156}--\blpage{169}
\doiurl{10.1016/j.cviu.2006.06.005}
\end{barticle}
\endbibitem

\bibitem[\protect\citeauthoryear{Fuentes et~al.}{2020}]{fuentes_method_2020}
\begin{barticle}
\bauthor{\bsnm{Fuentes}, \binits{J.E.}},
\bauthor{\bsnm{Moya}, \binits{F.D.}},
\bauthor{\bsnm{Montoya}, \binits{O.D.}}
(\byear{2020})
\batitle{Method for estimating solar energy potential based on photogrammetry from unmanned aerial vehicles}.
\bjtitle{\emph{Electronics}}
\bvolume{9}(\bissue{12}),
\bfpage{2144}.
\bcomment{Publisher: MDPI}
\end{barticle}
\endbibitem

\bibitem[\protect\citeauthoryear{Fu et~al.}{2014}]{fu_surveillance_2014}
\begin{barticle}
\bauthor{\bsnm{Fu}, \binits{Y.-G.}},
\bauthor{\bsnm{Zhou}, \binits{J.}},
\bauthor{\bsnm{Deng}, \binits{L.}}
(\byear{2014})
\batitle{Surveillance of a {2D} {Plane} {Area} with {3D} {Deployed} {Cameras}}.
\bjtitle{\emph{Sensors}}
\bvolume{14}(\bissue{2}),
\bfpage{1988}--\blpage{2011}
\doiurl{10.3390/s140201988}
\end{barticle}
\endbibitem

\bibitem[\protect\citeauthoryear{Garland and Heckbert}{1997}]{garland_surface_1997}
\begin{bchapter}
\bauthor{\bsnm{Garland}, \binits{M.}},
\bauthor{\bsnm{Heckbert}, \binits{P.S.}}
(\byear{1997})
\bctitle{Surface simplification using quadric error metrics}.
\bbtitle{\emph{Proceedings of the 24th Annual Conference on {Computer} Graphics and Interactive Techniques}}.
\bsertitle{{SIGGRAPH} '97},
pp. \bfpage{209}--\blpage{216}.
\bpublisher{ACM Press/Addison-Wesley Publishing Co.},
\blocation{USA}.
\doiurl{10.1145/258734.258849}
\end{bchapter}
\endbibitem

\bibitem[\protect\citeauthoryear{Ghosh}{2010}]{ghosh_approximation_2010}
\begin{bchapter}
\bauthor{\bsnm{Ghosh}, \binits{S.K.}}
(\byear{2010})
\bctitle{Approximation {Algorithms} for {Art} {Gallery} {Problems} in {Polygons} and {Terrains}}.
\bbtitle{\emph{{WALCOM}: {Algorithms} and {Computation}}}
vol. \bseriesno{5942},
pp. \bfpage{21}--\blpage{34}.
\bpublisher{Springer},
\blocation{Berlin, Heidelberg}.
\doiurl{10.1007/978-3-642-11440-3_3}
\end{bchapter}
\endbibitem

\bibitem[\protect\citeauthoryear{Grünbaum}{1975}]{grunbaum_polytopal_1975}
\begin{bchapter}
\bauthor{\bsnm{Grünbaum}, \binits{B.}}
(\byear{1975})
\bctitle{Polytopal graphs}.
\bbtitle{\emph{Studies in {Graph} {Theory}}}
vol. \bseriesno{12},
pp. \bfpage{201}--\blpage{224}.
\bpublisher{The Mathematical Association of America},
\blocation{Washington}
\end{bchapter}
\endbibitem

\bibitem[\protect\citeauthoryear{Gai et~al.}{2021}]{gai_using_2021}
\begin{barticle}
\bauthor{\bsnm{Gai}, \binits{J.}},
\bauthor{\bsnm{Xiang}, \binits{L.}},
\bauthor{\bsnm{Tang}, \binits{L.}}
(\byear{2021})
\batitle{Using a depth camera for crop row detection and mapping for under-canopy navigation of agricultural robotic vehicle}.
\bjtitle{\emph{Computers and Electronics in Agriculture}}
\bvolume{188},
\bfpage{106301}
\doiurl{10.1016/j.compag.2021.106301}
\end{barticle}
\endbibitem

\bibitem[\protect\citeauthoryear{Honsberger}{1976}]{honsberger_mathematical_1976}
\begin{bbook}
\bauthor{\bsnm{Honsberger}, \binits{R.}}
(\byear{1976})
\bbtitle{Mathematical Gems {II}}.
\bpublisher{Published and distributed by the Mathematical Association of America},
\blocation{Washington}.
\bcomment{OCLC: 1257299219}
\end{bbook}
\endbibitem

\bibitem[\protect\citeauthoryear{Hochbaum and Pathria}{1998}]{hochbaum_analysis_1998}
\begin{barticle}
\bauthor{\bsnm{Hochbaum}, \binits{D.S.}},
\bauthor{\bsnm{Pathria}, \binits{A.}}
(\byear{1998})
\batitle{Analysis of the greedy approach in problems of maximum k-coverage}.
\bjtitle{\emph{Naval Research Logistics (NRL)}}
\bvolume{45}(\bissue{6}),
\bfpage{615}--\blpage{627}
\doiurl{10.1002/(SICI)1520-6750(199809)45:6<615::AID-NAV5>3.0.CO;2-5}
\end{barticle}
\endbibitem

\bibitem[\protect\citeauthoryear{Indu et~al.}{2009}]{indu_optimal_2009}
\begin{bchapter}
\bauthor{\bsnm{Indu}, \binits{S.}},
\bauthor{\bsnm{Chaudhury}, \binits{S.}},
\bauthor{\bsnm{Mittal}, \binits{N.R.}},
\bauthor{\bsnm{Bhattacharyya}, \binits{A.}}
(\byear{2009})
\bctitle{Optimal sensor placement for surveillance of large spaces}.
\bbtitle{\emph{2009 {Third} {ACM}/{IEEE} {International} {Conference} on {Distributed} {Smart} {Cameras} ({ICDSC})}},
pp. \bfpage{1}--\blpage{8}.
\doiurl{10.1109/ICDSC.2009.5289398}
\end{bchapter}
\endbibitem

\bibitem[\protect\citeauthoryear{Jarray}{2013}]{jarray_lagrangean-based_2013}
\begin{barticle}
\bauthor{\bsnm{Jarray}, \binits{F.}}
(\byear{2013})
\batitle{A {Lagrangean}-based heuristics for the target covering problem in wireless sensor network}.
\bjtitle{\emph{Applied Mathematical Modelling}}
\bvolume{37}(\bissue{10}),
\bfpage{6780}--\blpage{6785}
\doiurl{10.1016/j.apm.2013.02.006}
\end{barticle}
\endbibitem

\bibitem[\protect\citeauthoryear{Kritter et~al.}{2019}]{kritter_optimal_2019}
\begin{barticle}
\bauthor{\bsnm{Kritter}, \binits{J.}},
\bauthor{\bsnm{Brévilliers}, \binits{M.}},
\bauthor{\bsnm{Lepagnot}, \binits{J.}},
\bauthor{\bsnm{Idoumghar}, \binits{L.}}
(\byear{2019})
\batitle{On the optimal placement of cameras for surveillance and the underlying set cover problem}.
\bjtitle{\emph{Applied Soft Computing}}
\bvolume{74},
\bfpage{133}--\blpage{153}
\doiurl{10.1016/j.asoc.2018.10.025}
\end{barticle}
\endbibitem

\bibitem[\protect\citeauthoryear{{Kenichi Yabuta} and {Hitoshi Kitazawa}}{2008}]{kenichi_yabuta_optimum_2008}
\begin{bchapter}
\bauthor{\bsnm{{Kenichi Yabuta}}},
\bauthor{\bsnm{{Hitoshi Kitazawa}}}
(\byear{2008})
\bctitle{Optimum camera placement considering camera specification for security monitoring}.
\bbtitle{\emph{2008 {IEEE} {International} {Symposium} on {Circuits} And {Systems}}},
pp. \bfpage{2114}--\blpage{2117}.
\bpublisher{IEEE},
\blocation{Seattle, WA, USA}.
\doiurl{10.1109/ISCAS.2008.4541867}
\end{bchapter}
\endbibitem

\bibitem[\protect\citeauthoryear{Kim et~al.}{2019}]{kim_systematic_2019}
\begin{barticle}
\bauthor{\bsnm{Kim}, \binits{J.}},
\bauthor{\bsnm{Ham}, \binits{Y.}},
\bauthor{\bsnm{Chung}, \binits{Y.}},
\bauthor{\bsnm{Chi}, \binits{S.}}
(\byear{2019})
\batitle{Systematic {Camera} {Placement} {Framework} for {Operation}-{Level} {Visual} {Monitoring} on {Construction} {Jobsites}}.
\bjtitle{\emph{Journal of Construction Engineering and Management}}
\bvolume{145}(\bissue{4}),
\bfpage{04019019}
\doiurl{10.1061/(ASCE)CO.1943-7862.0001636}
\end{barticle}
\endbibitem

\bibitem[\protect\citeauthoryear{Khaloo and Lattanzi}{2015}]{khaloo_hierarchical_2015}
\begin{bchapter}
\bauthor{\bsnm{Khaloo}, \binits{A.}},
\bauthor{\bsnm{Lattanzi}, \binits{D.}}
(\byear{2015})
\bctitle{A {Hierarchical} {Computer} {Vision} {Approach} to {Infrastructure} {Inspection}}.
\bbtitle{\emph{Computing in {Civil} {Engineering} 2015}},
pp. \bfpage{540}--\blpage{547}.
\bpublisher{American Society of Civil Engineers},
\blocation{Austin, Texas}.
\doiurl{10.1061/9780784479247.067}
\end{bchapter}
\endbibitem

\bibitem[\protect\citeauthoryear{Khuller et~al.}{1999}]{khuller_budgeted_1999}
\begin{barticle}
\bauthor{\bsnm{Khuller}, \binits{S.}},
\bauthor{\bsnm{Moss}, \binits{A.}},
\bauthor{\bsnm{Naor}, \binits{J.S.}}
(\byear{1999})
\batitle{The budgeted maximum coverage problem}.
\bjtitle{\emph{Information Processing Letters}}
\bvolume{70}(\bissue{1}),
\bfpage{39}--\blpage{45}
\doiurl{10.1016/S0020-0190(99)00031-9}
\end{barticle}
\endbibitem

\bibitem[\protect\citeauthoryear{Kranakis and Pocchiola}{}]{kranakis_brief_nodate}
\begin{botherref}
\oauthor{\bsnm{Kranakis}, \binits{E.}},
\oauthor{\bsnm{Pocchiola}, \binits{M.}}
A {Brief} {Survey} of {Art} {Gallery} {Problems} in {Integer} {Lattice} {Systems}
\end{botherref}
\endbibitem

\bibitem[\protect\citeauthoryear{Levoy}{1981}]{levoy_area_1981}
\begin{botherref}
\oauthor{\bsnm{Levoy}, \binits{M.}}
(1981)
Area flooding algorithms.
\emph{Two-Dimensional Computer Animation, Course Notes 9 for SIGGRAPH}
\textbf{82}
\end{botherref}
\endbibitem

\bibitem[\protect\citeauthoryear{Liu et~al.}{2015}]{liu_adaptive_2015}
\begin{barticle}
\bauthor{\bsnm{Liu}, \binits{K.}},
\bauthor{\bsnm{Mei}, \binits{Y.}},
\bauthor{\bsnm{Shi}, \binits{J.}}
(\byear{2015})
\batitle{An {Adaptive} {Sampling} {Strategy} for {Online} {High}-{Dimensional} {Process} {Monitoring}}.
\bjtitle{\emph{Technometrics}}
\bvolume{57}(\bissue{3}),
\bfpage{305}--\blpage{319}
\doiurl{10.1080/00401706.2014.947005}
\end{barticle}
\endbibitem

\bibitem[\protect\citeauthoryear{Lu et~al.}{2023}]{lu_daisee_2023}
\begin{barticle}
\bauthor{\bsnm{Lu}, \binits{X.}},
\bauthor{\bsnm{Rainforth}, \binits{T.}},
\bauthor{\bsnm{Teh}, \binits{Y.W.}}
(\byear{2023})
\batitle{Daisee: {Adaptive} importance sampling by balancing exploration and exploitation}.
\bjtitle{\emph{Scandinavian Journal of Statistics}}
\bvolume{50}(\bissue{3}),
\bfpage{1298}--\blpage{1324}
\doiurl{10.1111/sjos.12637}
\end{barticle}
\endbibitem

\bibitem[\protect\citeauthoryear{Morsly et~al.}{2012}]{morsly_particle_2012}
\begin{barticle}
\bauthor{\bsnm{Morsly}, \binits{Y.}},
\bauthor{\bsnm{Aouf}, \binits{N.}},
\bauthor{\bsnm{Djouadi}, \binits{M.S.}},
\bauthor{\bsnm{Richardson}, \binits{M.}}
(\byear{2012})
\batitle{Particle {Swarm} {Optimization} {Inspired} {Probability} {Algorithm} for {Optimal} {Camera} {Network} {Placement}}.
\bjtitle{\emph{IEEE Sensors Journal}}
\bvolume{12}(\bissue{5}),
\bfpage{1402}--\blpage{1412}
\doiurl{10.1109/JSEN.2011.2170833} .
\bcomment{Conference Name: IEEE Sensors Journal}
\end{barticle}
\endbibitem

\bibitem[\protect\citeauthoryear{Marzal}{2012}]{marzal_three-dimensional_2012}
\begin{botherref}
\oauthor{\bsnm{Marzal}, \binits{J.}}
The three-dimensional art gallery problem and its solutions.
{PhD} {Thesis},
Murdoch University
(2012).
\url{https://researchportal.murdoch.edu.au/esploro/outputs/doctoral/The-three-dimensional-art-gallery-problem-and/991005541831107891}
\end{botherref}
\endbibitem

\bibitem[\protect\citeauthoryear{Murray et~al.}{2007}]{murray_coverage_2007}
\begin{barticle}
\bauthor{\bsnm{Murray}, \binits{A.T.}},
\bauthor{\bsnm{Kim}, \binits{K.}},
\bauthor{\bsnm{Davis}, \binits{J.W.}},
\bauthor{\bsnm{Machiraju}, \binits{R.}},
\bauthor{\bsnm{Parent}, \binits{R.}}
(\byear{2007})
\batitle{Coverage optimization to support security monitoring}.
\bjtitle{\emph{Computers, Environment and Urban Systems}}
\bvolume{31}(\bissue{2}),
\bfpage{133}--\blpage{147}
\doiurl{10.1016/j.compenvurbsys.2006.06.002}
\end{barticle}
\endbibitem

\bibitem[\protect\citeauthoryear{Munir and Parasuraman}{2021}]{munir_analysis_2021}
\begin{botherref}
\oauthor{\bsnm{Munir}, \binits{A.}},
\oauthor{\bsnm{Parasuraman}, \binits{R.}}
Analysis of {Exploration} vs. {Exploitation} in {Adaptive} {Information} {Sampling}.
arXiv.
arXiv:2111.11384
(2021).
\doiurl{10.48550/arXiv.2111.11384}
\end{botherref}
\endbibitem

\bibitem[\protect\citeauthoryear{Möller and Trumbore}{1997}]{moller_fast_1997}
\begin{barticle}
\bauthor{\bsnm{Möller}, \binits{T.}},
\bauthor{\bsnm{Trumbore}, \binits{B.}}
(\byear{1997})
\batitle{Fast, {Minimum} {Storage} {Ray}-{Triangle} {Intersection}}.
\bjtitle{\emph{Journal of Graphics Tools}}
\bvolume{2}(\bissue{1}),
\bfpage{21}--\blpage{28}
\doiurl{10.1080/10867651.1997.10487468}
\end{barticle}
\endbibitem

\bibitem[\protect\citeauthoryear{Nishizeki and Baybars}{1979}]{nishizeki_lower_1979}
\begin{barticle}
\bauthor{\bsnm{Nishizeki}, \binits{T.}},
\bauthor{\bsnm{Baybars}, \binits{I.}}
(\byear{1979})
\batitle{Lower bounds on the cardinality of the maximum matchings of planar graphs}.
\bjtitle{\emph{Discrete Mathematics}}
\bvolume{28}(\bissue{3}),
\bfpage{255}--\blpage{267}
\doiurl{10.1016/0012-365X(79)90133-X}
\end{barticle}
\endbibitem

\bibitem[\protect\citeauthoryear{Nabhan et~al.}{2021}]{nabhan_correlation-based_2021}
\begin{barticle}
\bauthor{\bsnm{Nabhan}, \binits{M.}},
\bauthor{\bsnm{Mei}, \binits{Y.}},
\bauthor{\bsnm{Shi}, \binits{J.}}
(\byear{2021})
\batitle{Correlation-based dynamic sampling for online high dimensional process monitoring}.
\bjtitle{\emph{Journal of Quality Technology}}
\bvolume{53}(\bissue{3}),
\bfpage{289}--\blpage{308}
\doiurl{10.1080/00224065.2020.1726717}
\end{barticle}
\endbibitem

\bibitem[\protect\citeauthoryear{O'Rourke}{1987}]{orourke_art_1987}
\begin{bbook}
\bauthor{\bsnm{O'Rourke}, \binits{J.}}
(\byear{1987})
\bbtitle{Art Gallery Theorems and Algorithms}.
\bsertitle{The international series of monographs on computer science},
vol. \bseriesno{3}.
\bpublisher{Oxford Univ. Press},
\blocation{New York, NY}
\end{bbook}
\endbibitem

\bibitem[\protect\citeauthoryear{Penha et~al.}{2013}]{penha_coverage_2013}
\begin{bchapter}
\bauthor{\bsnm{Penha}, \binits{E.}},
\bauthor{\bsnm{Fantini}, \binits{C.}},
\bauthor{\bsnm{Chaimowicz}, \binits{L.}}
(\byear{2013})
\bctitle{Coverage in {Arbitrary} {3D} {Environments} {The} {Art} {Gallery} {Problem} in {Shooter} {Games}}.
\burl{https://api.semanticscholar.org/CorpusID:56033743}
\end{bchapter}
\endbibitem

\bibitem[\protect\citeauthoryear{Rebai et~al.}{2016}]{rebai_exact_2016}
\begin{barticle}
\bauthor{\bsnm{Rebai}, \binits{M.}},
\bauthor{\bsnm{Le~Berre}, \binits{M.}},
\bauthor{\bsnm{Hnaien}, \binits{F.}},
\bauthor{\bsnm{Snoussi}, \binits{H.}}
(\byear{2016})
\batitle{Exact {Biobjective} {Optimization} {Methods} for {Camera} {Coverage} {Problem} in {Three}-{Dimensional} {Areas}}.
\bjtitle{\emph{IEEE Sensors Journal}}
\bvolume{16}(\bissue{9}),
\bfpage{3323}--\blpage{3331}
\doiurl{10.1109/JSEN.2016.2519451}
\end{barticle}
\endbibitem

\bibitem[\protect\citeauthoryear{Rebai et~al.}{2015}]{rebai_sensor_2015}
\begin{barticle}
\bauthor{\bsnm{Rebai}, \binits{M.}},
\bauthor{\bsnm{Le~berre}, \binits{M.}},
\bauthor{\bsnm{Snoussi}, \binits{H.}},
\bauthor{\bsnm{Hnaien}, \binits{F.}},
\bauthor{\bsnm{Khoukhi}, \binits{L.}}
(\byear{2015})
\batitle{Sensor deployment optimization methods to achieve both coverage and connectivity in wireless sensor networks}.
\bjtitle{\emph{Computers \& Operations Research}}
\bvolume{59},
\bfpage{11}--\blpage{21}
\doiurl{10.1016/j.cor.2014.11.002}
\end{barticle}
\endbibitem

\bibitem[\protect\citeauthoryear{Sun et~al.}{2021}]{sun_learning_2021}
\begin{bchapter}
\bauthor{\bsnm{Sun}, \binits{Y.}},
\bauthor{\bsnm{Huang}, \binits{Q.}},
\bauthor{\bsnm{Hsiao}, \binits{D.-Y.}},
\bauthor{\bsnm{Guan}, \binits{L.}},
\bauthor{\bsnm{Hua}, \binits{G.}}
(\byear{2021})
\bctitle{Learning {View} {Selection} for {3D} {Scenes}},
pp. \bfpage{14464}--\blpage{14473}.
\burl{https://openaccess.thecvf.com/content/CVPR2021/html/Sun_Learning_View_Selection_for_3D_Scenes_CVPR_2021_paper.html}
\end{bchapter}
\endbibitem

\bibitem[\protect\citeauthoryear{Smith et~al.}{2018}]{smith_aerial_2018}
\begin{barticle}
\bauthor{\bsnm{Smith}, \binits{N.}},
\bauthor{\bsnm{Moehrle}, \binits{N.}},
\bauthor{\bsnm{Goesele}, \binits{M.}},
\bauthor{\bsnm{Heidrich}, \binits{W.}}
(\byear{2018})
\batitle{Aerial path planning for urban scene reconstruction: a continuous optimization method and benchmark}.
\bjtitle{\emph{ACM Transactions on Graphics}}
\bvolume{37}(\bissue{6}),
\bfpage{1}--\blpage{15}
\doiurl{10.1145/3272127.3275010}
\end{barticle}
\endbibitem

\bibitem[\protect\citeauthoryear{Solar et~al.}{2002}]{solar_parallel_2002}
\begin{barticle}
\bauthor{\bsnm{Solar}, \binits{M.}},
\bauthor{\bsnm{Parada}, \binits{V.}},
\bauthor{\bsnm{Urrutia}, \binits{R.}}
(\byear{2002})
\batitle{A parallel genetic algorithm to solve the set-covering problem}.
\bjtitle{\emph{Computers \& Operations Research}}
\bvolume{29}(\bissue{9}),
\bfpage{1221}--\blpage{1235}
\doiurl{10.1016/S0305-0548(01)00026-0}
\end{barticle}
\endbibitem

\bibitem[\protect\citeauthoryear{Straub et~al.}{2019}]{straub_replica_2019}
\begin{botherref}
\oauthor{\bsnm{Straub}, \binits{J.}},
\oauthor{\bsnm{Whelan}, \binits{T.}},
\oauthor{\bsnm{Ma}, \binits{L.}},
\oauthor{\bsnm{Chen}, \binits{Y.}},
\oauthor{\bsnm{Wijmans}, \binits{E.}},
\oauthor{\bsnm{Green}, \binits{S.}},
\oauthor{\bsnm{Engel}, \binits{J.J.}},
\oauthor{\bsnm{Mur-Artal}, \binits{R.}},
\oauthor{\bsnm{Ren}, \binits{C.}},
\oauthor{\bsnm{Verma}, \binits{S.}},
\oauthor{\bsnm{Clarkson}, \binits{A.}},
\oauthor{\bsnm{Yan}, \binits{M.}},
\oauthor{\bsnm{Budge}, \binits{B.}},
\oauthor{\bsnm{Yan}, \binits{Y.}},
\oauthor{\bsnm{Pan}, \binits{X.}},
\oauthor{\bsnm{Yon}, \binits{J.}},
\oauthor{\bsnm{Zou}, \binits{Y.}},
\oauthor{\bsnm{Leon}, \binits{K.}},
\oauthor{\bsnm{Carter}, \binits{N.}},
\oauthor{\bsnm{Briales}, \binits{J.}},
\oauthor{\bsnm{Gillingham}, \binits{T.}},
\oauthor{\bsnm{Mueggler}, \binits{E.}},
\oauthor{\bsnm{Pesqueira}, \binits{L.}},
\oauthor{\bsnm{Savva}, \binits{M.}},
\oauthor{\bsnm{Batra}, \binits{D.}},
\oauthor{\bsnm{Strasdat}, \binits{H.M.}},
\oauthor{\bsnm{De~Nardi}, \binits{R.}},
\oauthor{\bsnm{Goesele}, \binits{M.}},
\oauthor{\bsnm{Lovegrove}, \binits{S.}},
\oauthor{\bsnm{Newcombe}, \binits{R.}}
The {Replica} {Dataset}: {A} {Digital} {Replica} of {Indoor} {Spaces}.
arXiv.
arXiv:1906.05797 [cs, eess]
(2019).
\url{http://arxiv.org/abs/1906.05797}
\end{botherref}
\endbibitem

\bibitem[\protect\citeauthoryear{Thuillier et~al.}{2024}]{thuillier_efficient_2024}
\begin{barticle}
\bauthor{\bsnm{Thuillier}, \binits{O.}},
\bauthor{\bsnm{Le~Josse}, \binits{N.}},
\bauthor{\bsnm{Olteanu}, \binits{A.-L.}},
\bauthor{\bsnm{Sevaux}, \binits{M.}},
\bauthor{\bsnm{Tanguy}, \binits{H.}}
(\byear{2024})
\batitle{Efficient configuration of heterogeneous multistatic sonar networks: {A} mixed-integer linear programming approach}.
\bjtitle{\emph{Computers \& Operations Research}}
\bvolume{167},
\bfpage{106637}
\doiurl{10.1016/j.cor.2024.106637}
\end{barticle}
\endbibitem

\bibitem[\protect\citeauthoryear{Wilhelm}{2001}]{wilhelm_technical_2001}
\begin{barticle}
\bauthor{\bsnm{Wilhelm}, \binits{W.E.}}
(\byear{2001})
\batitle{A {Technical} {Review} of {Column} {Generation} in {Integer} {Programming}}.
\bjtitle{\emph{Optimization and Engineering}}
\bvolume{2}(\bissue{2}),
\bfpage{159}--\blpage{200}
\doiurl{10.1023/A:1013141227104} .
Accessed 2024-11-19
\end{barticle}
\endbibitem

\bibitem[\protect\citeauthoryear{Wang et~al.}{2020}]{wang_solving_2020}
\begin{barticle}
\bauthor{\bsnm{Wang}, \binits{X.}},
\bauthor{\bsnm{Zhang}, \binits{H.}},
\bauthor{\bsnm{Gu}, \binits{H.}}
(\byear{2020})
\batitle{Solving {Optimal} {Camera} {Placement} {Problems} in {IoT} {Using} {LH}-{RPSO}}.
\bjtitle{\emph{IEEE Access}}
\bvolume{8},
\bfpage{40881}--\blpage{40891}
\doiurl{10.1109/ACCESS.2019.2941069}
\end{barticle}
\endbibitem

\bibitem[\protect\citeauthoryear{Yang et~al.}{2018}]{yang_computeraided_2018}
\begin{barticle}
\bauthor{\bsnm{Yang}, \binits{X.}},
\bauthor{\bsnm{Li}, \binits{H.}},
\bauthor{\bsnm{Huang}, \binits{T.}},
\bauthor{\bsnm{Zhai}, \binits{X.}},
\bauthor{\bsnm{Wang}, \binits{F.}},
\bauthor{\bsnm{Wang}, \binits{C.}}
(\byear{2018})
\batitle{Computer‐{Aided} {Optimization} of {Surveillance} {Cameras} {Placement} on {Construction} {Sites}}.
\bjtitle{\emph{Computer-Aided Civil and Infrastructure Engineering}}
\bvolume{33}(\bissue{12}),
\bfpage{1110}--\blpage{1126}
\doiurl{10.1111/mice.12385}
\end{barticle}
\endbibitem

\bibitem[\protect\citeauthoryear{Yaagoubi et~al.}{2015}]{yaagoubi_hybvor_2015}
\begin{barticle}
\bauthor{\bsnm{Yaagoubi}, \binits{R.}},
\bauthor{\bsnm{Yarmani}, \binits{M.E.}},
\bauthor{\bsnm{Kamel}, \binits{A.}},
\bauthor{\bsnm{Khemiri}, \binits{W.}}
(\byear{2015})
\batitle{{HybVOR}: {A} {Voronoi}-{Based} {3D} {GIS} {Approach} for {Camera} {Surveillance} {Network} {Placement}}.
\bjtitle{\emph{ISPRS International Journal of Geo-Information}}
\bvolume{4}(\bissue{2}),
\bfpage{754}--\blpage{782}
\doiurl{10.3390/ijgi4020754}
\end{barticle}
\endbibitem

\bibitem[\protect\citeauthoryear{Zarzycki et~al.}{2024}]{zarzycki_leveraging_2024}
\begin{barticle}
\bauthor{\bsnm{Zarzycki}, \binits{H.}},
\bauthor{\bsnm{Ewald}, \binits{D.}},
\bauthor{\bsnm{Prokopowicz}, \binits{P.}}
(\byear{2024})
\batitle{Leveraging {Swarm} {Intelligence} for {Optimal} {Thermal} {Camera} and {Sensor} {Placement} in {Industrial} {Environments}}.
\bjtitle{\emph{Electronics}}
\bvolume{13}(\bissue{3}),
\bfpage{601}
\doiurl{10.3390/electronics13030601}
\end{barticle}
\endbibitem

\bibitem[\protect\citeauthoryear{Zan et~al.}{2023}]{zan_spatial_2023}
\begin{barticle}
\bauthor{\bsnm{Zan}, \binits{X.}},
\bauthor{\bsnm{Wang}, \binits{D.}},
\bauthor{\bsnm{Xian}, \binits{X.}}
(\byear{2023})
\batitle{Spatial {Rank}-{Based} {Augmentation} for {Nonparametric} {Online} {Monitoring} and {Adaptive} {Sampling} of {Big} {Data} {Streams}}.
\bjtitle{\emph{Technometrics}}
\bvolume{65}(\bissue{2}),
\bfpage{243}--\blpage{256}
\doiurl{10.1080/00401706.2022.2143903}
\end{barticle}
\endbibitem

\bibitem[\protect\citeauthoryear{Zhang et~al.}{2013}]{zhang_optimized_2013}
\begin{bchapter}
\bauthor{\bsnm{Zhang}, \binits{H.}},
\bauthor{\bsnm{Xia}, \binits{L.}},
\bauthor{\bsnm{Tian}, \binits{F.}},
\bauthor{\bsnm{Wang}, \binits{P.}},
\bauthor{\bsnm{Cui}, \binits{J.}},
\bauthor{\bsnm{Tang}, \binits{C.}},
\bauthor{\bsnm{Deng}, \binits{N.}},
\bauthor{\bsnm{Ma}, \binits{N.}}
(\byear{2013})
\bctitle{An optimized placement algorithm for collaborative information processing at a wireless camera network}.
\bbtitle{\emph{2013 {IEEE} {International} {Conference} on {Multimedia} and {Expo} ({ICME})}},
pp. \bfpage{1}--\blpage{6}.
\doiurl{10.1109/ICME.2013.6607594} .
\bcomment{ISSN: 1945-788X}
\end{bchapter}
\endbibitem

\end{thebibliography}

\newpage

\begin{appendices}

\section{Mathematical program notation} \label{appA:ip-notation}

\subsection{Index}

\begin{tabular}{ll}
$p$ & Camera position index (3D vector) \\
$d$ & Direction index (3D vector) \\
$(p,d)$ & Camera configuration index \\
       & (comprises of two vectors: position vector \\
       & and direction vector. The tail of the direction \\ & vector is the camera position and the head is \\ 
       & the point the camera is focused on. \\
$v$ & Free space voxel index (comprises of one position vector)
\end{tabular}

\subsection{Sets}

\begin{tabular}{ll}
$V$ & Set of all free space voxels \\
$P$ & Set of all possible camera positions \\
$PD$ & Set of all possible camera configurations \\
     & (This is a theoretical construct as this is a continuous space) \\
$PD'$ & Set of sampled camera configurations \\
& (This is an actual input in the mathematical program) \\
$V_{pd}$ & Set of all free space voxels visible under camera configuration $(p,d)$ \\
$PD'_v$ & Subset of all camera configurations which can view free space voxel $v$ \\
$PD^{\text{adj}}_p$ & Set of all camera configurations which are in the immediate neighborhood \\
& of camera located at position $p$. As cameras have finite space, there can only \\
& be one camera per neighborhood.\\
\end{tabular}

\subsection{Parameters}

\begin{tabular}{ll}
$f_{pd}$ & Fixed cost of deploying camera at $(p,d)$ (equals 1 if resource budget refers to count of cameras) \\
$\beta$ & Maximum permissible cost ($\beta \in Z$ if it is an upper bound on the cardinality \\ & 
of the camera network) \\
\end{tabular}

\subsection{Decision variables}

\begin{tabular}{ll}
$x_{pd}$ & Binary variable, which is 1 iff the camera at configuration $(p,d)$ is part of the selected \\ 
& solution at that iteration \\
$y_v$ & Binary variable, 1 iff the selected camera network covers voxel $v$ \\
\end{tabular}

\section{Visibility function algorithm} \label{appB:vis-notation}

\subsection{Complete notation}

\subsubsection{Inputs}

\begin{tabular}{ll}
$F_m$ & Faces of the 3D environment mesh object \\
$r^V$ & Set of position vectors of free space voxels\\
$r^c_p$ & Eye of the camera, i.e., position vector of the camera position $p$ \\
$r^c_{w1}$ & Position vector of the horizontal field of view extremity \#1. \\
& Significant only when used in collaboration with $r^c_p$ \\
$r^c_{w2}$ & Position vector of the horizontal field of view extremity \#2. \\
$r^c_{h1}$ & Position vector of the vertical field of view extremity \#1. \\
$r^c_{h2}$ & Position vector of the vertical field of view extremity \#2. \\
$r^v$ & Position vector of voxel index $v$.
\end{tabular}

\subsubsection{Intermediates}

\begin{tabular}{ll}
$r^{V_{\text{visited}}}$ & Set of position vectors of voxels visited during the flood fill procedure \\
$r^{V_{\text{stack}}}$ & Set of position vectors of voxels in the stack which still need to be checked for being in FOV \\
$r^{\tilde{V}}$ & Set of position vectors of connected voxels which are within the FOV \\
$\hat{d}_{wn}$ & Vector normal to the horizontal field of view plane \\
$\hat{d}_{hn}$ & Vector normal to the vertical field of view plane \\
$\hat{d}^c_v$ & Direction vector from camera eye to voxel \\
$\hat{d}^c_{vw}$ & Component of direction vector from camera eye to voxel which is coplanar \\
& to the horizontal field of view plane \\
$\hat{d}^c_{vh}$ & Component of direction vector from camera eye to voxel which is coplanar \\
& to the vertical field of view plane \\
$b_w$ & Logical operator which is 1 if voxel in horizontal field of view \\
$b_h$ & Logical operator which is 1 if voxel in vertical field of view \\
$r^c_{Fv}$ & Position vector of the ray, with terminus $r^c_p$ and direction $\hat{d}^c_v$ \\
$b_s$ & Logical operator which is 1 if voxel can be viewed from position vector $r^c_p$ \\
$e_i$ & Unit vector with $i^{th}$ index set to 1.
\end{tabular}

\textbf{Note:} The visibility function Calculate-Camera-View outputs the set $V_{pd}$: the set of free space voxels visible from camera position $p$ and direction $d$. The position $p$ is an explicit argument expressed by $r^c_p$. The direction $d$ is a derived argument by various means. One of them is: $d = \frac{(r^c_{w1} + r^c_{w2})/2 - r^c_p}{||(r^c_{w1} + r^c_{w2})/2 - r^c_p||_2}$.

\subsubsection{Pseudocode}

\begin{algorithm}[H]
\caption{Visibility calculation pseudocode}
\label{alg:visibility-calc}
\begin{algorithmic}
\Procedure{Calculate-Camera-View}{$F_{m}, r^{V}, r^{c}_{e}, r^{c}_{w1}, r^{c}_{w2}, r^{c}_{h1}, r^{c}_{h2}$}
    \State $\hat{d}_{wn} \gets \frac{(r^{c}_{w1} - r^{c}_{e}) \times (r^{c}_{w2} - r^c_p)}{||(r^c_{w1} - r^c_p) \times (r^c_{w2} - r^c_p)||_{2}}$
    \State $\hat{d}_{hn} \gets \frac{(r^c_{h1} - r^c_p) \times (r^c_{h2} - r^c_p)}{||(r^c_{h1} - r^c_p) \times (r^c_{h2} - r^c_p)||_{2}}$
    \State $r^{V_{\text{visited}}} \gets \{r^c_p\}$
    \State $r^{\tilde{V}} \gets \emptyset$
    \State $r^{V_{\text{stack}}} = [ r^c_p + e_1, r^c_p - e_1, r^c_p + e_2, r^c_p - e_2, r^c_p + e_3, r^c_p - e_3 ]$
    \While{$r^{V_{\text{stack}}} \neq \emptyset$}
        \Comment{Iteratively checking whether voxels in FOV}
        \State $r^v \gets \text{pop}(r^{V_{\text{stack}}})$
        \If{$r^v \notin r^{V_{\text{visited}}} \text{ and } r^v \in r^V$}
            \State $r^{V_{\text{visited}}} \gets r^{V_{\text{visited}}} \cup \{r^v\}$
            \State $\hat{d}^c_v \gets r^v - r^c_p$
            \State $\hat{d}^c_{vw} \gets \hat{d}^c_v - \lt(\hat{d}^c_v \cdot \hat{d}_{wn}\rt) \hat{d}_{wn}$
            \State $\hat{d}^c_{vh} \gets \hat{d}^c_v - \lt(\hat{d}^c_v \cdot \hat{d}_{hn}\rt) \hat{d}_{hn}$
            \If{$\hat{d}^c_{vw}$ is in between non-reflex angle formed by $r^c_{w1}, r^c_p, r^c_{w2}$}
                \State $b_w \gets 1$
            \Else
                \State $b_w \gets 0$
            \EndIf
            \If{$\hat{d}^c_{vh}$ is in between non-reflex angle formed by $r^c_{h1}, r^c_p, r^c_{h2}$}
                \State $b_h \gets 1$
            \Else
                \State $b_h \gets 0$
            \EndIf
            \If{$b_w \cdot b_h = 1$}
                \State $r^{\tilde{V}} \gets r^{\tilde{V}} \cup \{r^v\}$
                \State $r^{V_{\text{stack}}} \gets r^{V_{\text{stack}}} \cup [ r^v + e_1, r^v - e_1, r^v + e_2, r^v - e_2, r^v + e_3, r^v - e_3 ]$
            \EndIf 
        \EndIf
    \EndWhile
    \State $V_{pd} \gets \emptyset$
    \For{$r^v \in r^{\tilde{V}}$}
        \Comment{For loop checking condition for being no obstructions}
        \State $r^c_{Fv} \gets$ Ray-Intersection($r^c_p, r^v, F_m$)
        \If{$\min(\|r^c_{Fv} - r^c_p\|_2, \text{DOF}) \geq \|\hat{d}^c_v\|$}
            \State $b_s \gets 1$
        \Else
            \State $b_s \gets 0$
        \EndIf
        \If{$b_s = 1$}
            \State $V_{pd} \gets V_{pd} \cup \{v\}$
            \Comment{$V_{pd}$ gets voxel index $v$ corresponding to position vector $r^v$}
        \EndIf
    \EndFor
    \State \Return $V_{pd}$
\EndProcedure
\end{algorithmic}
\end{algorithm}

\section{Sampling strategy E\&E notation} \label{appC:ee-notation}

\subsection{Inputs}

\begin{tabular}{ll}
$N_{\text{pos}}$ & Total number of positions to be sampled \\
$N_{\text{dir-pos}}$ & Total number of directions to be sampled per sample position \\
$f_{\text{exploit}}$ & Exploitation fraction \\
$PD_{\text{iter}}$ & Set of solution configurations \\
$P$ & Set of possible camera positions \\
$\theta_{\text{jitter}}$ & Angle perturbation allowed \\
$v_{\text{jitter}}$ & Voxel displacement tolerance \\

\end{tabular}

\subsection{Intermediates and output}

\begin{tabular}{ll}
$N_{\text{tot}}$  & Number of total configurations to be sampled \\
$\lceil \cdot \rfloor$ & Integer rounding function \\
$N_{\text{explore}}$ & Number of total configurations to be randomly sampled for the exploration phase \\
$N_{\text{pos-explore}}$ & Number of positions to be randomly sampled for the exploration phase \\
$PD_{\text{explore}}$ & Set of camera configurations in exploration set, sampled randomly \\
$N_{\text{sol}}$  & Number of configurations present in previous optimal solution \\
$N_{\text{exploit}}$ & Number of total configurations to be randomly sampled for the exploration phase \\
$N_{\text{config-sol}}$ & Total number of camera configurations that are sampled per \\ & configuration in the previous iteration's selected camera network  \\
$z_0$ & The $(0,0,1)'$ vector \\
$\mathcal{D}_z$ & Mapping of ordinals to random 3D direction vectors sampled that are \\
& a maximum angle of $\theta_{\text{jitter}}$ degrees away from the unit normal vector \\
& along z-axis. [i.e., the $z_0$ vector]\\
$(p,d)$ & Individual camera configuration composed of direction vector $d$ and voxel position $v$ \\
$d_z$ & Individual direction from $\mathcal{D}_z$ \\
$(p_{\text{new}}, d_{\text{new}})$ & The new configuration which will be part of the exploitation set \\
$v_{\text{perturb}}$ & Random voxel positions sampled that are a maximum Manhattan distance of \\ & $v_{\text{jitter}}$ away from $v$. \\
$PD_{\text{exploit}}$ & Set of camera configurations in exploitation set, sampled near about to the \\
& optimal solutions from the previous iteration  \\
$PD_{\text{E\&E}}$ & Final set of camera configurations updated in the new set \\
\end{tabular}

\section{Sampling strategy TUS notation} \label{appD:tus-notation}

\subsection{Inputs}

\begin{tabular}{ll}
$N_{\text{pos}}$ & Total number of positions to be sampled \\
$N_{\text{dir-pos}}$ & Total number of directions to be sampled per sample position \\
$f_{\text{unc}}$ & Uncovered search fraction \\
$P$ & Set of possible camera positions \\
$V^{\text{super}}_{\text{center}}$ & Set of positions of the supervoxels  \\
$\mathcal{G}_{\text{unc}}$ & Mapping from supervoxel index to the number of uncovered, smaller \\
& voxels within the supervoxel\\
\end{tabular}

\subsection{Intermediates and output}

\begin{tabular}{ll}
    $N_{\text{tot}}$  & Number of total configurations to be sampled \\
    $N_{\text{random}}$ & Number of total configurations to be sampled for the random search phase \\
    $N_{\text{pos-random}}$ & Number of positions to be sampled for the random search phase \\
    $PD_{\text{random}}$ & Set of camera configurations in random search set \\
    $N_{\text{targeted}}$ & Number of configurations considered in the uncovered sample set \\
    $C_{\text{normalize}}$ & Total number of uncovered voxels \\
    $v^{\text{super}}_{\text{center}}$ & Position of particular supervoxel \\
    $Pr(v)$ & Probability of supervoxel $v$ being chosen as the center on which the camera \\ 
    & will be focused on \\
    $PD_{\text{unc}}$ & Set of camera configurations focusing on targeting uncovered spaces \\
    $PD_{\text{TUS}}$ & Final set of camera configurations updated in the new set \\
\end{tabular}

\section{Theoretical analysis: Lemma 1} \label{appE:lemma1}

\textit{Given a single-precision floating-point representation, with \( \epsilon \) as the smallest distinguishable increment, and \( |P| \) as the cardinality of the position set, the number of possible camera configurations is on the order of \( \mathcal{O} \left( \frac{|P|}{\epsilon^2} \right) \).}

\begin{proof}
    Firstly, we calculate the number of unique directions possible within floating-point precision. Assume the set of direction vectors is \( D \), with its cardinality denoted as \( |D| \).

    There are a total of \( \left( \frac{1}{\epsilon} \right) \) floating-point values representable between \( 0 \) and \( 1 \) by definition. Thus, there are a total of \( \left( \frac{2}{\epsilon} \right) \) distinct values in the interval \( [-1, 1] \). 

    Within the volume in the interval \( [-1, 1]^3 \), there are a total of \( \left( \frac{2}{\epsilon} \right)^3 \) unique floating-point values possible in \( \mathbb{R}^3 \). This represents the cardinality of the set of all position vectors within \( [-1, 1]^3 \), which, equivalently, represents the set of all direction vectors within this interval when originating from the origin. These values are uniformly distributed across the volume.

    Now, we aim to calculate the cardinality of the subset of vectors with unit norm. This subset corresponds to the number of direction vectors that approximately lie on the surface of the unit sphere.
    \begin{align*}
        |D| &\approx \frac{\text{Volume of hollow unit      sphere}}{\text{Volume of }[-1, 1]^3} \cdot \lt( \text{Number of vectors from origin within } [-1, 1]^3 \rt) \\
            &= \frac{\text{Volume of hollow unit sphere}}{\text{Volume of }[-1, 1]^3} \cdot \lt(\frac{2}{\epsilon}\rt)^3 \\
            &= \frac{4 \pi / 3 \lt(1^3 - (1 - \epsilon)^3 \rt)}{8} \lt(\frac{8}{\epsilon^3}\rt) \\
            & \leq \frac{4 \pi}{\epsilon^2} \\
            & = \mathcal{O}\lt(\frac{1}{\epsilon^2} \rt)\\
    \end{align*}
    Now, the set of direction vectors $D$ is independent of the position vector being chosen. Therefore, the set of camera configurations can be represented as the Cartesian product of the position set $P$ and direction set $D$:
    \begin{align*}
        |P \times D| &= |P| \cdot |D| 
        = \mathcal{O}\lt(\frac{|P|}{\epsilon^2}\rt),
    \end{align*}
    which completes the proof.
\end{proof}

\renewcommand{\thefigure}{F\arabic{figure}}

\renewcommand{\thealgorithm}{F\arabic{algorithm}}
\renewcommand{\thetable}{F\arabic{table}}
\setcounter{algorithm}{0}

\renewcommand{\theproposition}{S\arabic{proposition}}

\renewcommand{\thelemma}{S\arabic{lemma}}

\section{Supplementary Materials} \label{app:supplementary}

\subsection{Greedy heuristic algorithm} \label{app:greedy_heuristic}

\begin{algorithm}[H]
\caption{Selection procedure for $PD^*$}
\begin{algorithmic}
    \Procedure{Greedy-Heuristic}{$V, PD'_v, \beta, PD'$}
        \State $V_{pd} \gets \emptyset$
        \For{$v \in V$}
        \Comment{$1^{st}$ for loop}
            \State $V_{pd} \gets V_{pd} \cup \{v\} \quad \forall (p, d) \in PD'_v$
            \Comment{nested sub-loop}
        \EndFor
        
        \State $F \gets 0$
        \State $PD^* \gets \emptyset$
        
        \While{$F \leq \beta$}
            \Comment{$2^{nd}$ for loop}
            \State $(\tilde{p}, \tilde{d}) \gets \arg \max_{(p,d) \in PD'} \frac{|V_{pd}|}{f_{pd}}$
            \Comment{(2a)}
            \State $PD' \gets PD' \setminus PD^{\text{adj}}_{\tilde{p}}$
            \Comment{(2b)}
            \For{$(p, d) \in PD'$}
                \Comment{(2c)}
                \State $V_{pd} \gets V_{pd} \setminus V_{\tilde{p} \tilde{d}}$
            \EndFor
            
            \State $PD^* \gets PD^* \cup \{(\tilde{p}, \tilde{d})\}$
            \State $F \gets F + f_{\tilde{p}\tilde{d}}$
        \EndWhile
        
        \State \Return $PD^*$
    \EndProcedure
\end{algorithmic}
\end{algorithm}

\setcounter{theorem}{0}
\begin{suppproposition}
    The greedy heuristic algorithm has a complexity of $\mathcal{O}(\beta \cdot |PD'| \cdot |V|)$.
\end{suppproposition}

\begin{proof}
    The $1^{\text{st}}$ \texttt{for} loop along with the nested sub-loop has a time complexity of $\mathcal{O}(|V| \cdot |PD'|)$, as $|PD'| \geq |PD'_v|$.

    We can assume, without loss of generality, that the minimum cost of a camera configuration is 1. This provides an outer loop complexity of $\mathcal{O}(\beta)$ for the 2nd \texttt{for} loop. Loop (2b) and (2a) have a complexity of $\mathcal{O}(|PD'|)$ each, whereas loop $(2c)$ has a nested sub-loop requiring iterations over voxels in addition to the configurations, giving it an inner loop complexity of $\mathcal{O}(|PD'|\cdot|V|)$.

    Thus, the overall complexity is $\mathcal{O}(\beta\cdot|PD'|\cdot|V|) + \mathcal{O}(|PD'| \cdot |V|)$ = $\mathcal{O}(\beta\cdot|PD'|\cdot|V|)$.
\end{proof}

\subsection{Minor functions} \label{section:F-minor-functions}

This subsection defines the pseudocode for the three minor functions present in the adaptive sampling algorithms.

\subsubsection{Vector rotation (E\&E strategy)}

\begin{algorithm}[H]
\caption{Vector rotation pseudocode}
\begin{algorithmic}[1]
\Procedure{Rotate-Along}{$\mathbf{v}_\text{from}, \mathbf{v}_\text{to}, \mathbf{v}$}
    \State Normalize $\mathbf{v}_\text{from} \gets \frac{\mathbf{v}_\text{from}}{\|\mathbf{v}_\text{from}\|}$
    \State Normalize $\mathbf{v}_\text{to} \gets \frac{\mathbf{v}_\text{to}}{\|\mathbf{v}_\text{to}\|}$
    
    \State Compute rotation axis $\mathbf{n} \gets \mathbf{v}_\text{from} \times \mathbf{v}_\text{to}$
    \State Compute rotation angle $\theta \gets \arccos\left(\mathbf{v}_\text{from} \cdot \mathbf{v}_\text{to}\right)$
    
    \State Set $c \gets \cos(\theta)$, $s \gets \sin(\theta)$, and $t \gets 1 - c$
    
    \State Let $\mathbf{n} = (n_x, n_y, n_z)$ be the components of $\mathbf{n}$
    
    \State Construct rotation matrix:
    \[
    R \gets \begin{pmatrix}
        t n_x^2 + c & t n_x n_y - s n_z & t n_x n_z + s n_y \\
        t n_x n_y + s n_z & t n_y^2 + c & t n_y n_z - s n_x \\
        t n_x n_z - s n_y & t n_y n_z + s n_x & t n_z^2 + c
    \end{pmatrix}
    \]
    
    \State Apply the rotation: $\mathbf{v}_{\text{rot}} \gets R \cdot \mathbf{v}^T$
    
    \State \textbf{return} $\mathbf{v}_{\text{rot}}^T$
\EndProcedure
\end{algorithmic}
\end{algorithm}

\subsubsection{Spherical cap sampling (E\&E strategy)}

\textbf{Note:} $\mathcal{U}_{\text{cont}}\lt(a, b\rt)$ refers to the continuous uniform distribution parameterized by lower bound $a$ and upper bound $b$.

\begin{algorithm}[H]
\caption{Spherical cap sampling pseudocode}
\begin{algorithmic}[1]
\Procedure{Sample-Spherical-Cap}{$\alpha, N$}
    \State $z_{\min} \gets \cos(\alpha)$
    \State $d_z \gets \emptyset$

    \For{$i \in \{1, 2, \cdots, N\}$}
        \State $z \sim \mathcal{U}_{\text{cont}} \lt(0, 1 - z_{\min} \rt) + z_{\min}$
        \State $r \gets \sqrt{1 - z^2}$
        \State $\theta \sim \mathcal{U}_{\text{cont}}\lt(0, 2 \pi \rt)$
        \State $x \gets r \cos \lt(\theta \rt)$
        \State $y \gets r \sin \lt(\theta \rt)$
        \State $d_z[i] \gets (x, y, z)$
    \EndFor
    
    \State \textbf{return} $d_z$
\EndProcedure
\end{algorithmic}
\end{algorithm}

\subsubsection{Random configuration sampling} \label{app:random_config}

\textbf{Notation:} $MVN(\mathbf{0}, \mathbf{I}_3)$ refers to the probability density function of a Trivariate standard normal distribution.

\begin{algorithm}[H]
\caption{Random configuration sampling pseudocode}
\label{alg:random_sample}
\begin{algorithmic}[1]
\Procedure{Sample-Random-Configurations}{$N_{\text{pos}}, P, N_{\text{dir-pos}}$}
    \State $PD_{\text{random}} \gets \emptyset$
    
    \For{$n_{\text{pos}} \in \{1, 2, \cdots, N_{\text{pos}}\}$}
        \State $p \sim P$
        \For{$n_{\text{dir}} \in \{1, 2, \cdots, N_{\text{dir-pos}}\}$}
            \State $d \sim MVN(\mathbf{0}, \mathbf{I}_3)$
            \State $d \gets d/||d||_2$
            \State Resample and renormalize if $d \in \{ (0, 1, 0)^T, (0, -1, 0)^T \}$
            \State $PD_{\text{random}} \gets PD_{\text{random}} \cup \{(p, d)\}$
        \EndFor
    \EndFor    
    \State \textbf{return} $PD_{\text{random}}$
\EndProcedure
\end{algorithmic}
\end{algorithm}

\subsection{Scenario 2: low-resource budget performance of TUS strategy} \label{app:scenario-2}

We utilize the same environment as in the 
high-resource case study to conduct a test application of the TUS strategy. Based on the results that we will present in Section 7 of the manuscript, we find that the TUS strategy performs well in contexts with open environments and low budgets. Although there are several obstructions in the case study environment, there are some open spaces as well, such as the living room and dining room space. So, we use the TUS strategy to maximize coverage of the indoor space with a lower resource budget of three cameras.

We contrast the high-resource budget case study with this one and discuss new insights that we obtain on optimal camera placement. The lower resource budget of three cameras translates to roughly one camera per 2,000 free space voxels.

\subsubsection{Best configurations}

We present a visual description of the coverage achieved in Figure \ref{fig8:overall-coverage-low-resource-budget}. The uncovered voxels are clustered around the two bathrooms at the top, near the far edges of the bedrooms, and to some extent around the staircase. The top level is mostly uncovered, and the algorithm primarily focuses on getting access to corridors, as it allows coverage of multiple spaces. In this case, the dining and living room are well covered. All major corridors, except for the staircases, are well covered. However, within the rooms, the focus is on accessing the central spaces.

\begin{figure}[H]
    \centering
    \subfigure[]{\includegraphics[width=0.9\textwidth]{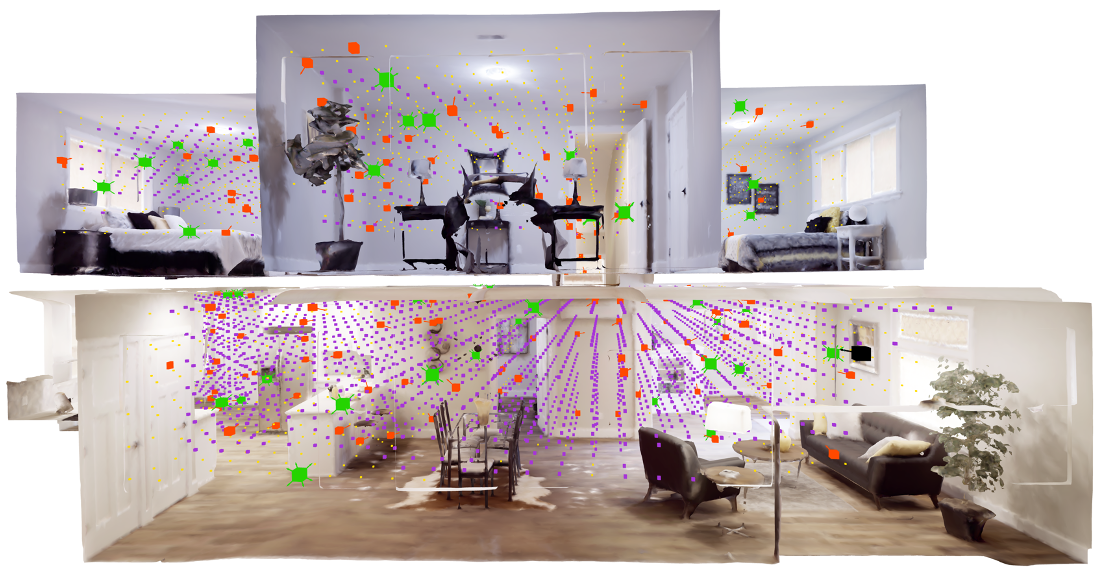} \label{fig8a}}
    
    \vspace{0.5em} 
    
    \subfigure[]{\includegraphics[width=0.9\textwidth]{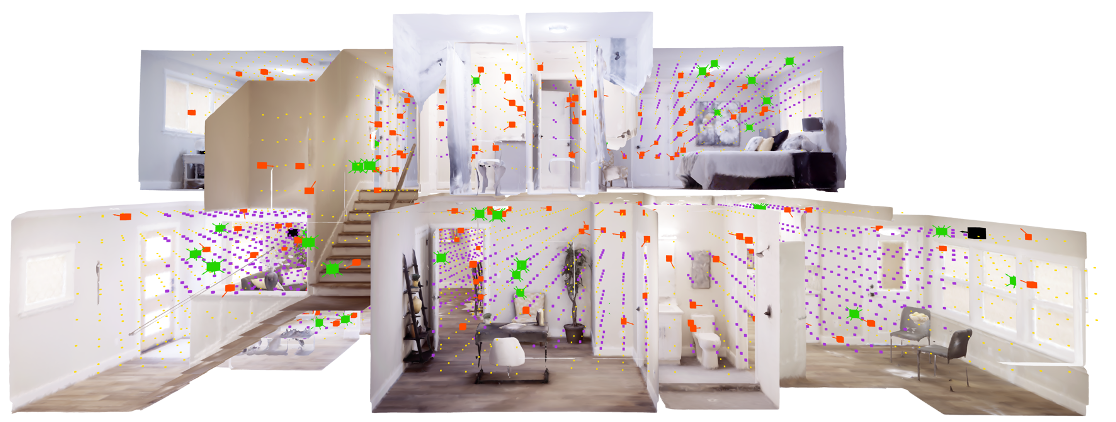} \label{fig8b}}
    
    \caption{This is a panel of two images, where the top image (Part (a)) gives a frontal view of the coverage achieved, and the bottom image (Part (b)) provides a back view (refer to Figure 3(b) for legend).}
    \label{fig8:overall-coverage-low-resource-budget}
\end{figure}

\begin{figure}[H]
    \centering
    \subfigure[]{\includegraphics[width=0.3\textwidth]{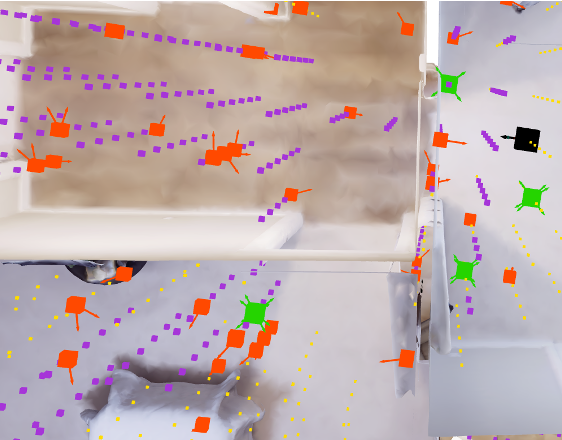} \label{fig9a}}
    \hspace{0.5em}
    \subfigure[]{\includegraphics[width=0.3\textwidth]{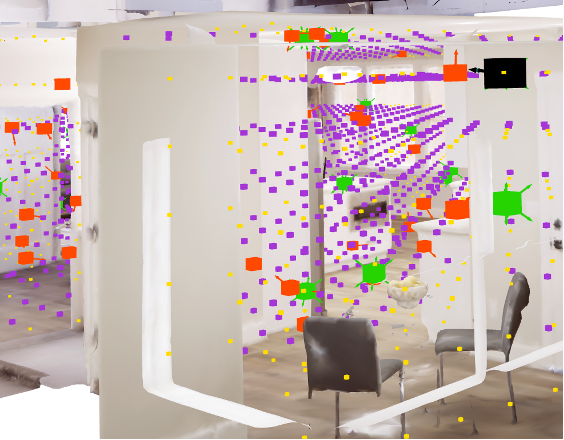} \label{fig9b}}
    \hspace{0.5em}
    \subfigure[]{\includegraphics[width=0.3\textwidth]{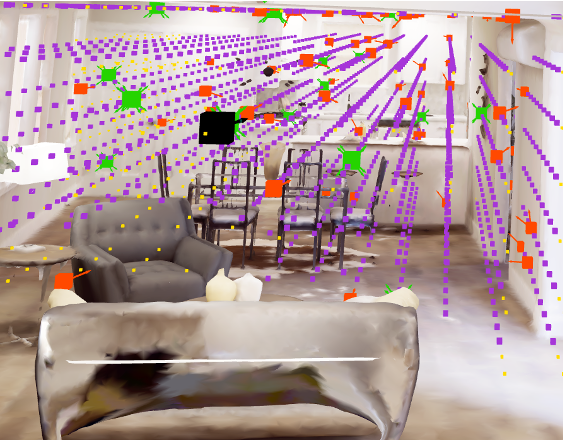} \label{fig9c}}
    
    \caption{This is a $3 \times 1$ panel of the three camera positions in the top position as chosen by the TUS strategy (refer to Figure 3(b) for legend).}
    \label{fig9:best-configurations-low-resource-budget}
\end{figure}

The optimal positions, given the sampling configuration set, and their coverage are described below. The camera configurations can be seen in Figure \ref{fig9:best-configurations-low-resource-budget}.

\textbf{Camera configuration a:}  This camera is placed in the living room at the top, angling downwards. It is placed in such a way that it is able to view large portions of living room 1 and a large part of the kitchen as well.

\textbf{Camera configuration b:} It is placed at the top in living room 2. It is located along the middle of the edge, angled in a way so that it has an unobstructed view to the corridors in the back and the front, which helps it capture parts of the study room at the back, and the kitchen and living room in the front.

\textbf{Camera configuration c:} This camera placement is not an obvious one from a conventional perspective. It is placed in bedroom 3 at the bottom, to get an unobstructed view of bedroom 1 in front along with some coverage of bedroom 2 as well. This leaves the staircases uncovered.

\subsubsection{Comparison with conventional wisdom for camera placement}

The main uncovered spaces in this case are the bathrooms, the staircases, and spaces along the edges of each room. The major spaces covered are the living room and kitchen, which occupy large spaces and allow substantial coverage even with the lower camera budget. 

This algorithm shows that under a smaller resource budget, it is important to prioritize covering large open spaces and to place the cameras that do that in positions where they can also cover corridors. In this case, two of the three cameras are placed along or near the top wall independently. Even though some of the exploratory camera configurations are sampled from the middle, the algorithm iteratively gravitates towards the top positions.

\subsubsection{Performance metrics}

The analysis was carried out using the best hyperparameters set for the TUS strategy, as estimated empirically in Appendix \ref{app:hyper_tuning}. We set the sampling budget at 800 configurations, spread over 10 iterations, with an uncovered search fraction of 0.4. This corresponds to 6 exploratory positions per iteration, and 8 angles sampled per voxel per iteration.

The performance metrics for this run are summarized as follows. The total processing time for visibility calculations is 343.61 seconds, and the total optimization time is 8.61 seconds. The total free space is 5,985 voxels, and the final coverage achieved is 3,151 voxels (53\% of the total free space). The summary statistics for coverage of each camera are given in Table \ref{table2:camera_coverage_su}.

\begin{table}[ht]
    \centering
    \caption{Camera configurations and corresponding voxel coverage statistics for the TUS algorithm. It provides the optimal solution for a resource budget of 3 cameras and details the network coverage provided, number of voxels covered, and overcoverage, i.e., voxels covered more than once.}
    \begin{tabular}{p{3.5cm}p{4cm}p{3cm}p{2.6cm}}
        \hline
        \raggedright \textbf{Camera Configuration} & \textbf{Camera Network Coverage Provided (\%)} & \raggedright \textbf{Number of Voxels Covered} & \raggedright \textbf{Voxels Covered More than Once} \\
        \hline
        Camera configuration a & 19.3\% & 609 & 0 \\
        
        Camera configuration b & 34.9\% & 1101 & 472 \\
        
        Camera configuration c & 60.7\% & 1913 & 472 \\
        \hline
    \end{tabular}

    \label{table2:camera_coverage_su}
\end{table}

In this case, the overcoverage fraction is high for camera configuration b. However, the number of unique voxels it observes is 629, which is still higher than the number of voxels that camera configuration a covers. With few cameras it is important to prioritize monitoring the largest open spaces, even if there is some redundancy in camera coverage.

\subsection{Hyperparameter tuning} \label{app:hyper_tuning}

In this section, we run the two adaptive sampling model strategies under sanitized, custom environments and develop useful insights about them. These insights allow us to evaluate which algorithm performs better and under what conditions. We analyze the efficacy of the algorithms and tune the hyperparameters of each algorithm to use it for our final case study, where we 
implement the model on much larger, real-life scenarios.

For this purpose, we construct a room generator, which based on our inputs, creates custom rooms with walls acting as obstructions.

\subsubsection{Room scenario generator} \label{app:room_gen} 

The room scenario generator develops a cuboid room and takes the following inputs to design a custom room:

\begin{tabular}{ll}
\texttt{length} & Room length (length of room along x-axis) \\
\texttt{height} & Room height (height of room along y-axis) \\
\texttt{breadth} & Room breadth (breadth of room along z-axis) \\
\texttt{num\_walls} & Number of walls that obstruct the room, installed at equidistant \\
                    & intervals, by default \\
\texttt{y\_wall\_edge\_ratio} & Ratio of height of wall and the height of the room \\
\texttt{z\_wall\_edge\_ratio} & Ratio of breadth of wall and the breadth of the room \\
\texttt{wall\_width} & Thickness of each wall \\
\texttt{random\_range} & The maximum amount of perturbation in the placement of each \\
                    & wall can take along the $X$-axis (from 0 to 1)  \\
\texttt{seed} & Pseudorandom number seed for replicating results\\
\texttt{wall\_orient} & Refers to the wall orientation and can take one of two options: \\
              & $\lt\{ \texttt{`alternate'}, \texttt{`same-side'} \rt\}$ \\
              & \texttt{`alternate'}: Wall partitions are at alternate sides of the room \\
              & \texttt{`same-side'}: Wall partitions are at the same side of the room \\
\end{tabular}
\newline

We solve the optimal camera placement problem for two different types of rooms: medium-sized rooms and large-sized rooms. We utilize two medium sized-rooms (each of dimension $40 \times 10 \times 10$, one with \texttt{`alternate'} wall partitions and one with \texttt{`same-side'} wall partitions) and two-large sized rooms (each of dimension $80 \times 10 \times 10$, similarly with \texttt{`alternate'} wall partitions and one with \texttt{`same-side'} wall partitions). We choose a resource budget of 4 cameras for the two medium-sized rooms and a resource budget of 8 cameras for the large-sized rooms.

Furthermore, as there is an inherent stochasticity in the choice of initial camera positions which dictates the model performance and solution time, we conduct experiments over two sets: 1) 5 runs with the same initial starting points, i.e., 5 sampling strategy runs, where the initial purely random phase iteration provides the same exact values, but after the first iteration, the next set of informed and random configurations provide different samples due to inherent stochasticity; and 2) 5 runs with the different initial starting points, i.e., 5 sampling strategy runs, where the initial random phase iteration itself generates different samples, and thus, the subsequent samples are bound to be different. This allows us to check robustness results to observe specific sets of hyperparameters and provide an upper and lower bound over the different runs.

For all runs, we consider a horizontal FOV of $90^\circ$, and an aspect ratio of 4:3, corresponding to a vertical FOV of $73^\circ$. As we are considering relatively smaller sized-indoor environments and not large-scale outdoor environments, we are not restricting the depth of field.

\subsubsection{Sensitivity analysis: E\&E strategy}

In this subsection, we evaluate the performance of the E\&E strategy for optimal camera placement. We systematically test this strategy across various configurations and perform a sensitivity analysis on the hyperparameters to determine the best settings for the model. This strategy balances the exploration of completely new camera configurations with the exploitation of the current solution and improving the solution locally.

\textbf{Exploitation fraction:} The exploitation fraction is one of the most important hyperparameters because it decides the critical balancing act that the model must perform in sampling camera configurations which are exploratory, i.e., randomly chosen from the pseudorandom number generator, or chosen approximately near the configurations from the incumbent camera network. A value of 1 for the exploitation fraction corresponds to all camera configurations being sampled near the incumbent camera network, whereas a value of 0 corresponds to all camera configurations being sampled randomly. To determine the optimal value for this hyperparameter, we conducted a series of experiments varying the voxel perturbation allowance from $\left\{0, 0.2, 0.4, 0.6, 0.8, 1\right\}$ and analyzed the algorithm's performance across different environments.

\begin{figure}[h]
    \centering
    \subfigure[]{\includegraphics[width=0.48\textwidth]{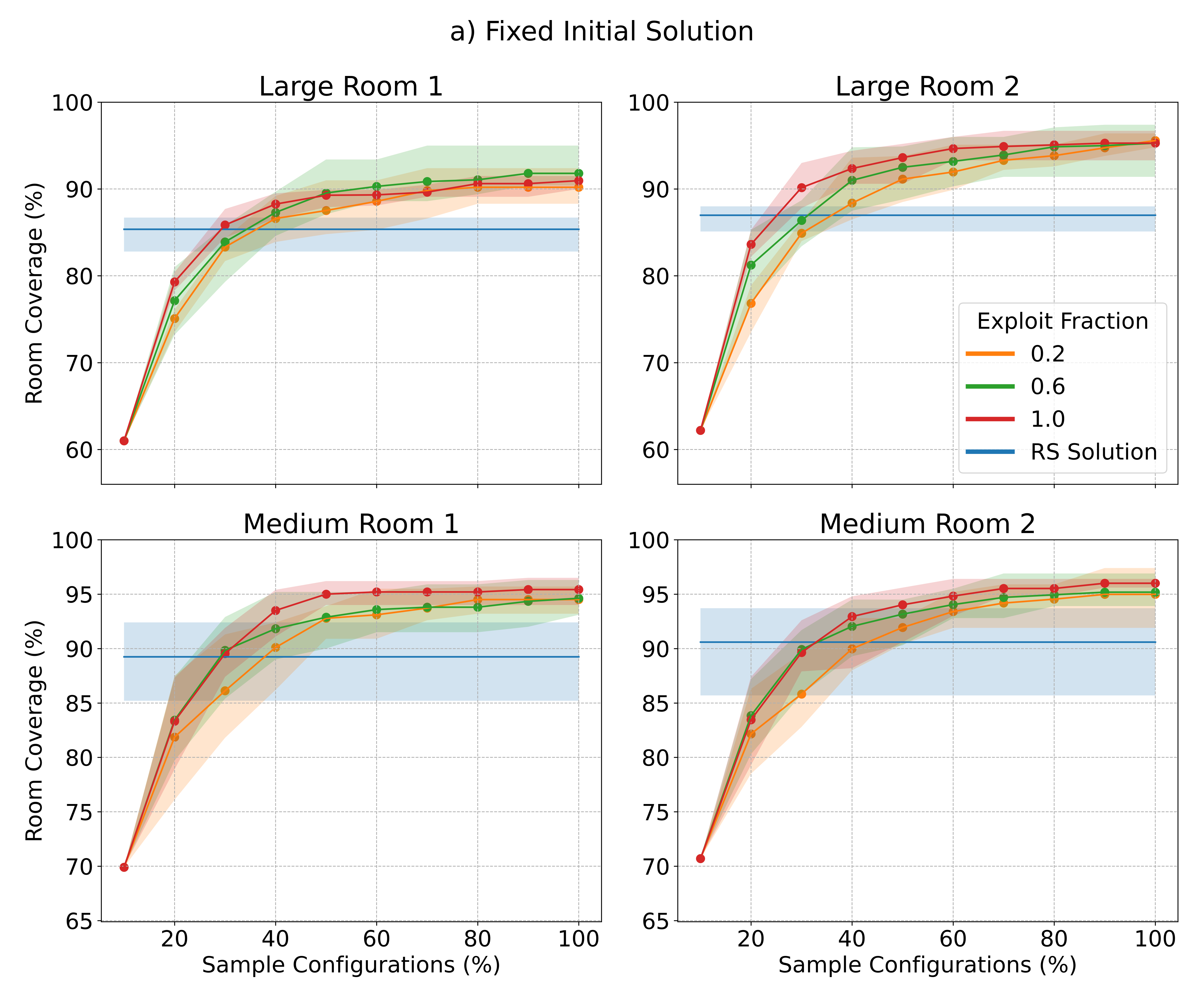} \label{figG1a}}
    \subfigure[]{\includegraphics[width=0.48\textwidth]{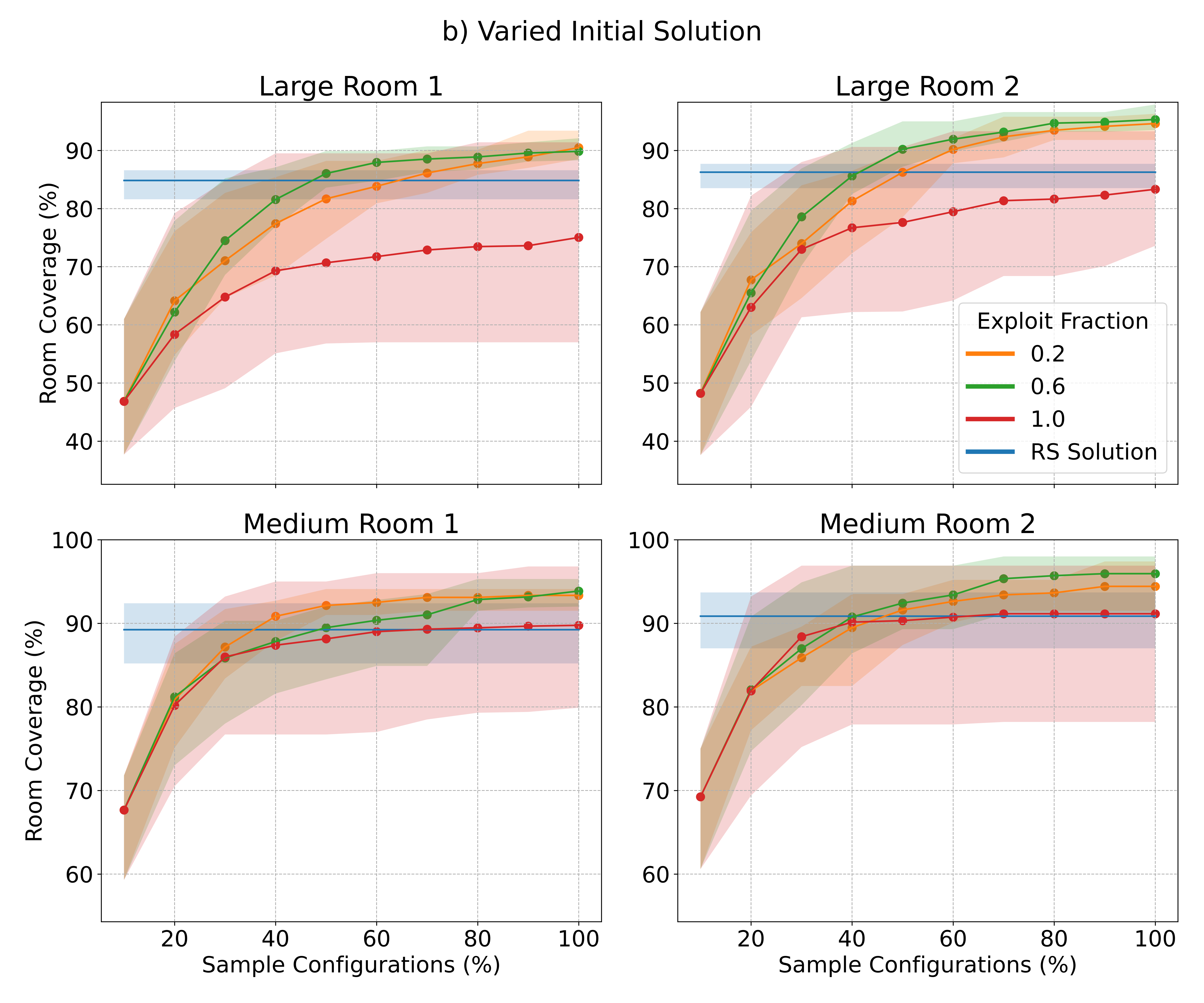} \label{figG1b}}
    \caption{Two $2 \times 2$ grids show the performance of an algorithm using the E\&E strategy with various exploit fractions (green, orange, red) versus the RS benchmark strategy (blue). Part (a) uses the same initial camera configurations, while part (b) starts with different initial solutions. The x-axis represent the percentage of the sampling budget that has been exhausted, whereas the y-axis provides the percentage of free space voxels covered. Each scenario uses five seeds, with shaded areas indicating the range of results. The E\&E strategy outperforms the RS benchmark, especially in larger rooms, where better solutions are found after sampling 30-40\% of configurations. Lower exploit fractions yield more robust results when starting from random initial solutions}
    \label{figG1:exploit-fraction-sensitivity}
\end{figure}

Figure \ref{figG1:exploit-fraction-sensitivity} displays two grids of $2 \times 2$ side-by-side, illustrating how, over iterations, the algorithm employing the E\&E strategy with various exploit fractions (represented by the colors green, orange, and red) surpasses the RS benchmark, depicted in blue. Due to inherent variability in each setting, we used five different seeds for each run, creating an envelope of the same color to represent the minimum and maximum values, highlighted around each solid line. Figure \yk{F}\ref{figG1a} represents a setting where the initial iteration provided the same optimal camera configurations as the starting point, while Figure \yk{F}\ref{figG1b} shows results obtained with different solutions in the initial iteration.

With a fixed starting point, it is evident that the E\&E strategy significantly outperforms RS benchmark. In Figure \yk{F}\ref{figG1a}, having any exploit fraction greater than 0 allows the algorithm to discover much better solutions on average. In larger rooms, as can be seen in \yk{F}\ref{figG1a}, the algorithm achieves a better solution than the RS benchmark mean after sampling 30-40\% of the configurations. In smaller rooms, while the RS benchmark is more competitive, the algorithm still finds better solutions after sampling 50-60\% of the configurations.

Furthermore, \yk{F}\ref{figG1a} suggests that the most effective exploit fraction to use is 1.0, indicating a pure exploitation strategy that samples only local solutions nearby. A close second is an exploit fraction of 0.6. However, when starting from a random initial solution, as shown in \yk{F}\ref{figG1b}, lower exploit fractions yield more robust and higher coverage, with an exploit fraction of 1.0 performing poorly due to high variance in the solutions. For this reason, we choose to use an exploit fraction of 0.6.

\textbf{Voxel perturbation allowance:} The voxel perturbation allowance is a critical hyperparameter in our E\&E strategy, directly impacting the flexibility and precision of camera placements in the exploit phase of the strategy. This parameter defines the range within which a camera's position can be adjusted, allowing for fine-tuning to improve visibility and coverage. To determine a value for this hyperparameter that generally performs well, we conducted a series of experiments varying the voxel perturbation allowance from $\left\{0, 1, 2, 3, 4\right\}$ and analyzed the algorithm's performance across different environments.

\begin{figure}[h!]
    \centering
    \includegraphics[width=0.8\textwidth]{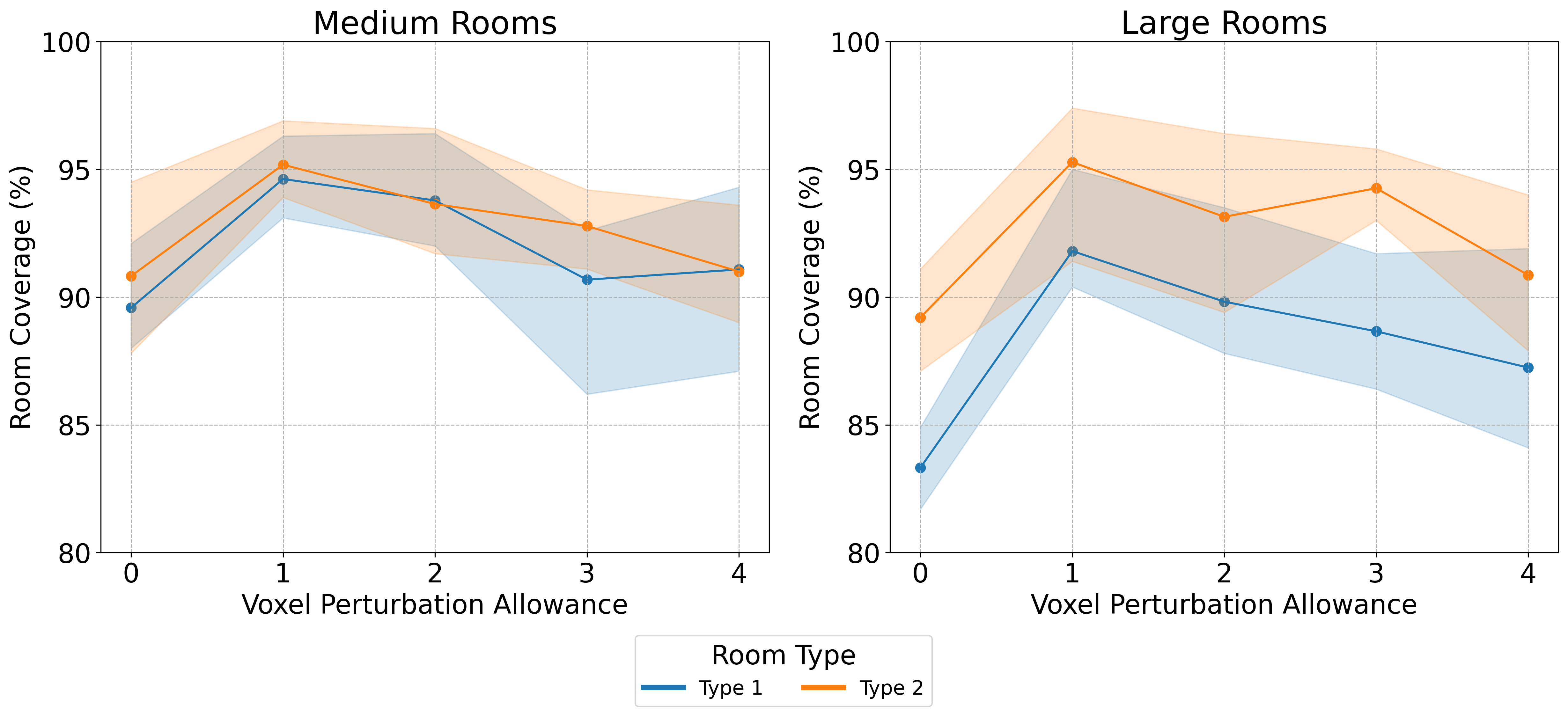}
    \caption{Sensitivity analysis of voxel perturbation allowance (VPA) on room coverage. The left graph shows mean coverage for medium-sized rooms, and the right for large-sized rooms. A VPA of 1 voxel yields the highest mean coverage and the lowest range across multiple seeds and runs, making it a good choice for the algorithm}
    \label{figG2:ee-vpa-sensitivity}
\end{figure}

Figure \ref{figG2:ee-vpa-sensitivity} shows the sensitivity analysis of the voxel perturbation allowance (VPA) on room coverage, using line plots where the x-axis represents the different VPA values tested, and the y-axis indicates the corresponding room coverage achieved. The left graph shows the mean coverage obtained by the E\&E strategy for medium-sized rooms across \texttt{alternate} (Type 1) and \texttt{same-side} (Type 2) room types, while the right graph presents the same for large-sized rooms.

The envelope around each plotted line of the same color represents the range of coverage outcomes generated by the algorithm when run with various camera configurations while controlling for the initial starting iteration. Firstly, the coverage trend demonstrates that both room types benefit from some degree of angle perturbation. This improvement is consistent across both room sizes, indicating that a moderate adjustment in camera orientation enhances the algorithm's ability to exploit camera configurations based on previous iterates' optimal configurations effectively. The analysis reveals that a VPA of 1 voxel yields the highest mean coverage for both medium and large rooms across the full algorithm run. Additionally, a VPA of 1 voxel also exhibits the lowest range over multiple seeds and runs, indicating more consistent performance. This VPA value, equivalent to approximately 10\% of the room's breadth, will be adopted as the hyperparameter for the final version of our algorithm.

\textbf{Angle perturbation allowance:} The angle perturbation allowance defines the range within which a camera's orientation can be adjusted, allowing for fine-tuning to improve visibility and coverage. A series of experiments were conducted, varying the voxel perturbation allowance from $\lt\{0^\circ, 15^\circ, 30^\circ, 45^\circ\rt\}$ and the algorithm's performance was analyzed across different environments.

The Figure \ref{figG3:ee-apa-sensitivity} shows the sensitivity analysis of the angle perturbation allowance (APA) on room coverage, using line plots where the x-axis represents the different APA values tested, and the y-axis indicates the corresponding room coverage achieved. The left graph shows the mean coverage obtained by the E\&E strategy for medium-sized rooms across two different room types, while the right graph presents the results for large-sized rooms. The envelope around each plotted line of the same color represents the range of coverage outcomes generated by the algorithm when run with various camera configurations while controlling for the initial starting iteration.

\begin{figure}[h!]
    \centering
    \includegraphics[width=0.8\textwidth]{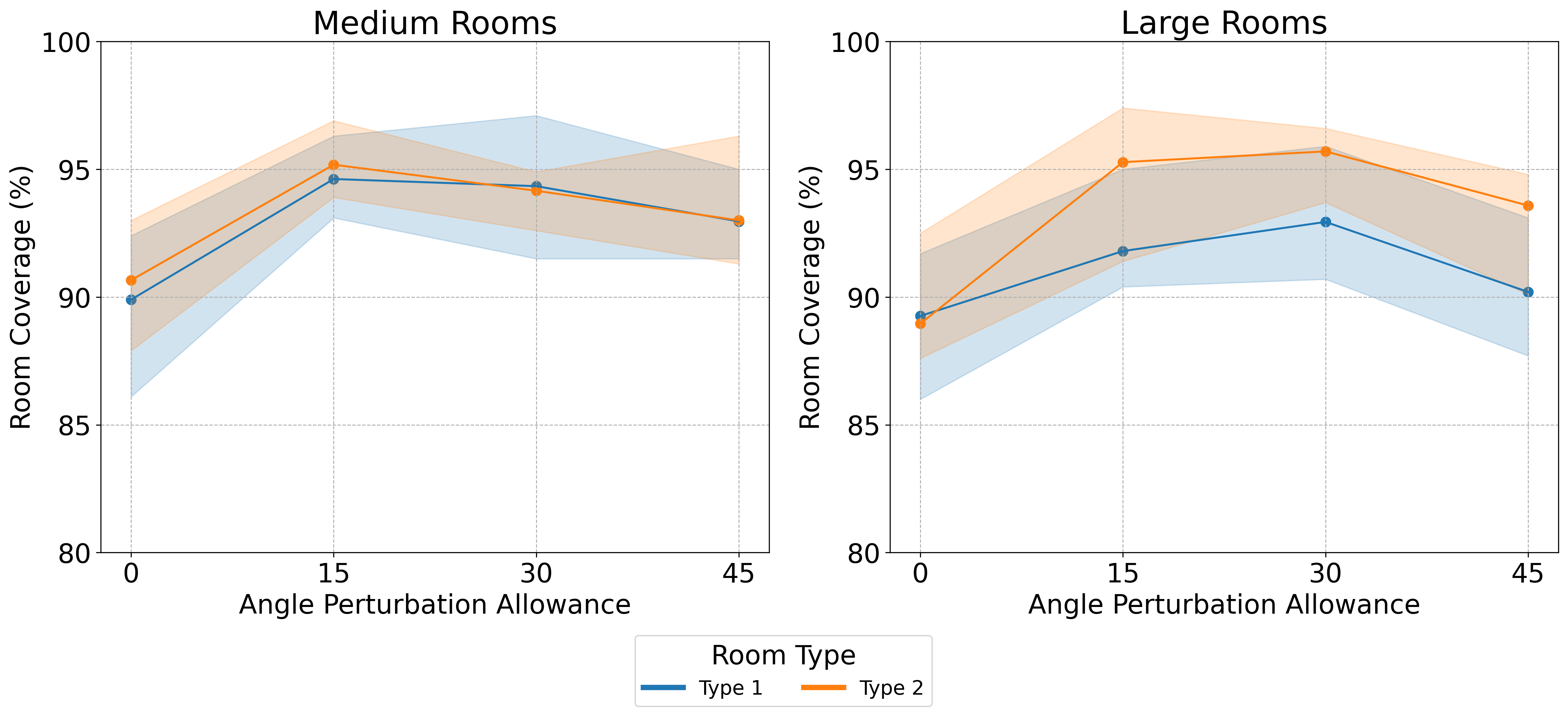}
    \caption{Sensitivity analysis of voxel perturbation allowance (VPA) on room coverage. The sensitivity analysis shows that an APA of $0^\circ$ performs poorly, while higher APAs, like $45^\circ$, result in random search behavior. Optimal values lie between $15^\circ$ and $30^\circ$. For large rooms, $30^\circ$ yields the highest mean coverage (94.32\%), and for medium rooms, $15^\circ$ is optimal (94.5\%). Based on these findings, $30^\circ$ is chosen as the preferred APA}
    \label{figG3:ee-apa-sensitivity}
\end{figure}

Firstly, it is clear that an APA of $0^\circ$ is not favorable, as it performs worse than all other angles. Also, at high APA of $45^\circ$, the model appears to lose focus on local solutions that could improve the overall outcome, resulting in a more random search behavior. So, a good hyperparameter choice is between $15^\circ$ and $30^\circ$. 

Specifically, the graph reveals that an APA of $30^\circ$ yields the highest mean coverage for the large room at 94.32\% mean coverage, compared to 93.54\% mean coverage with $15^\circ$ APA; while $15^\circ$ is optimal for the medium-sized rooms, achieving 94.5\% mean coverage compared to 94.25\% mean coverage with $30^\circ$ APA across the full algorithm run. The observed coverage ranges over different solutions are also comparable: an average range of 3.1\% at $15^\circ$ versus a range of 3.95\% at $30^\circ$ for the two medium-sized rooms; an average range of 5.3\% at $15^\circ$ versus a range of 4.05\% at $30^\circ$ for the two large-sized rooms. Given there are no significant differences, we choose a hyperparameter value of $30^\circ$, as it seems to be marginally better for the large-sized rooms.

\textbf{Iteration counts:} The iteration count is a critical hyperparameter in the E\&E strategy, determining the number of times the algorithm refines its camera placement configurations. In this sensitivity analysis, we explore the impact of varying iteration counts while maintaining a constant total number of sampled configurations. Specifically, we examine three different iteration counts: 5, 10, and 20, with each configuration maintaining a total of 8,000 sampled camera positions. Consequently, the number of samples per iteration is adjusted accordingly to 1600, 800, and 400, respectively.

\begin{figure}[h!]
    \centering
    \includegraphics[width=0.9\textwidth]{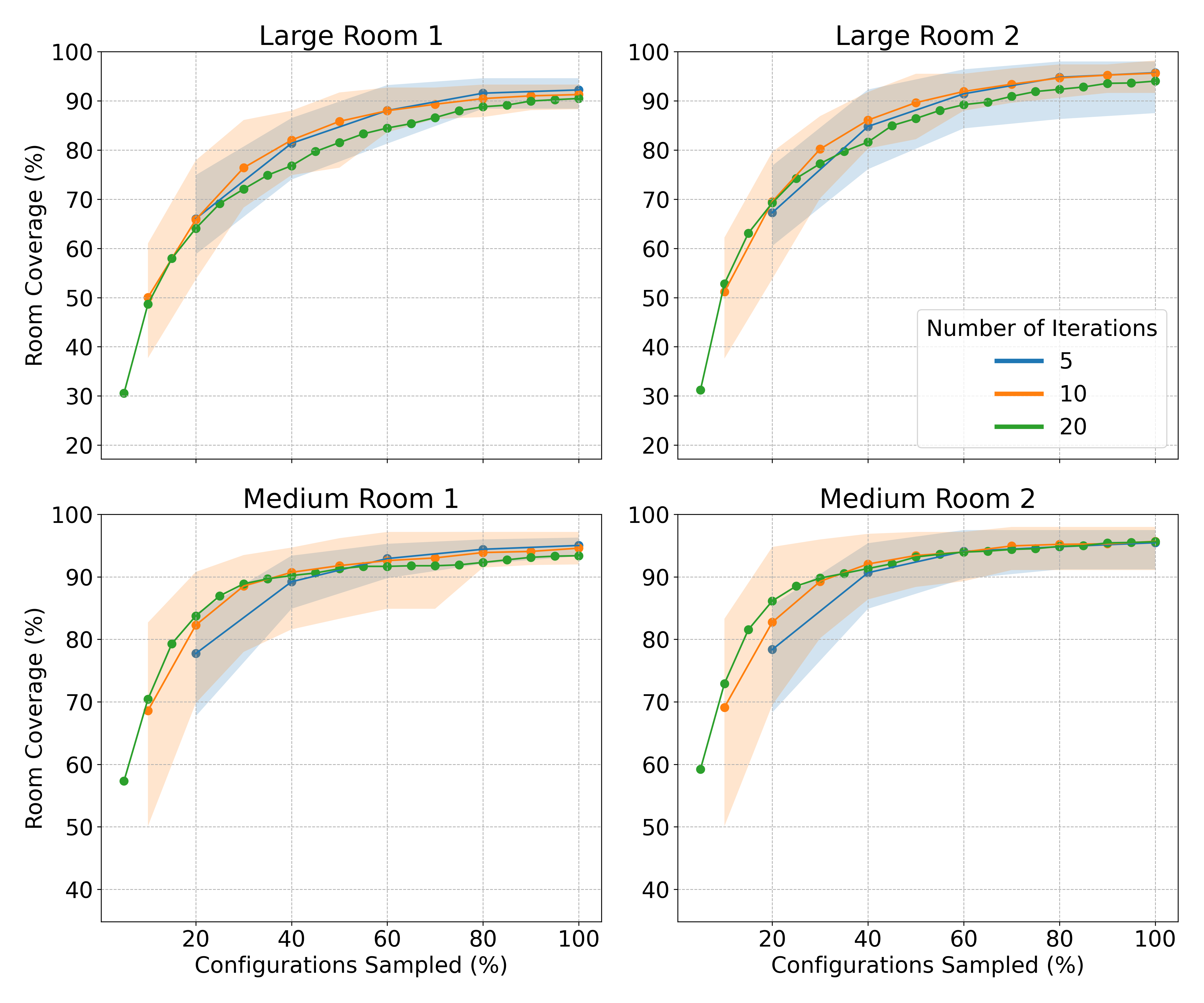}
    \caption{Room coverage as a function of configurations sampled, comparing different numbers of iterations. Shaded areas indicate coverage variability for 5 and 10 iterations. The 5-iteration run achieves the highest mean coverage for three out of four room types, but the 10-iteration run is preferred due to its tighter coverage range in large rooms, despite a near-tie in mean coverage}
    \label{figG4:ee-number-iterations-sensitivity}
\end{figure}

Figure \ref{figG4:ee-number-iterations-sensitivity} presents room coverage (y-axis) as a function of the total configurations sampled (x-axis). Different line plots correspond to varying numbers of total iterations, with the total number of camera configurations sampled held constant across runs. The shaded envelope around each line plot indicates the range of coverage outcomes generated by the algorithm when executed with different camera configurations. Notably, the envelope is constructed only for the 5-iteration and 10-iteration scenarios, as the 20-iteration run exhibits high variability and consistently lower mean coverage. While the mean coverage line for the 20-iteration run is included for completeness, it is not considered in selecting the optimal hyperparameter.

The mean coverage is highest with 5 iterations for three out of the four room types: 95.03\% for medium room type 1, 95.48\% for medium room type 2, and 95.71\% for large room type 2. For large room type 1, the mean coverage with 5 iterations is 92.21\%, placing it in second position. The mean coverage achieved with 10 iterations is very close to that of 5 iterations: 94.60\% for medium room type 1, 95.69\% for medium room type 2, 91.27\% for large room type 1, and 95.62\% for large room type 2. The difference between the two iteration counts is marginal, indicating a near tie in performance.

Despite the near-equal performance in terms of mean coverage, 10 iterations are chosen as the optimal setting. This decision is primarily based on the focus of the study on larger problems, where tighter coverage ranges are crucial. For large room type 1, the coverage range is slightly better for the 10-iteration runs at 5.0\% compared to the 5-iteration run at 6.2\%. However, for large room type 2, the coverage range is significantly tighter with 10 iterations (6.6\%) compared to 5 iterations (10.5\%). This reduction in variability makes 10 iterations the preferred choice for larger room scenarios.

\subsubsection{Sensitivity analysis: TUS strategy}

In this subsection, we evaluate the performance of the Target Uncovered Spaces strategy for optimal camera placement. This strategy aims to identify and target spaces that remain uncovered by the current camera configurations, thereby optimizing coverage and minimizing blind spots. We repeat the process used for the E\&E sampling strategy across various configurations and perform a sensitivity analysis on the hyperparameters to identify the optimal settings for the model.

\textbf{Uncovered search fraction:} The uncovered search fraction is a critical hyperparameter in the TUS strategy, as it determines the proportion of camera configurations specifically aimed at covering free spaces not yet covered by the optimally placed cameras from the previous iteration. The remaining camera configurations are generated randomly. A fraction value of 0 implies that all camera configurations are sampled randomly, with no targeted effort to cover previously uncovered spaces. We find that having a strict visibility requirement is detrimental as the coverage is worse than that of randomly sampled configurations. 

Our analysis indicates that this strategy is particularly effective when the resource budget is low and there are larger open spaces within the environment. To maximize its effectiveness, it is essential to deactivate the strict visibility requirement. This finding is supported by the analysis presented in Figure \ref{figG5:uncovered-search-fraction-sensitivity}, where the TUS strategy performs optimally under these conditions.

\begin{figure}[h]
    \centering
    \subfigure[]{\includegraphics[width=0.48\textwidth]{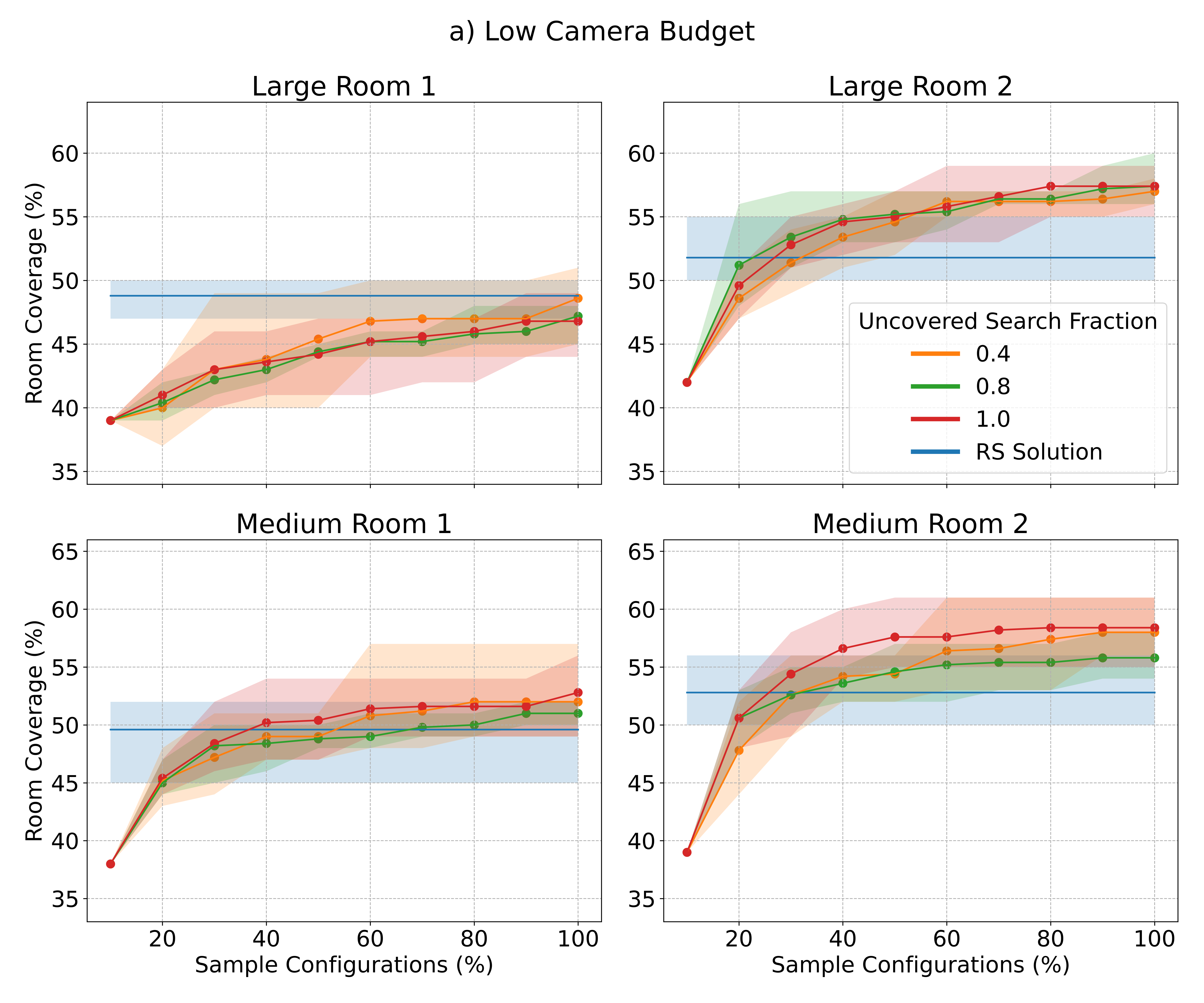} \label{figG5a}}
    \subfigure[]{\includegraphics[width=0.48\textwidth]{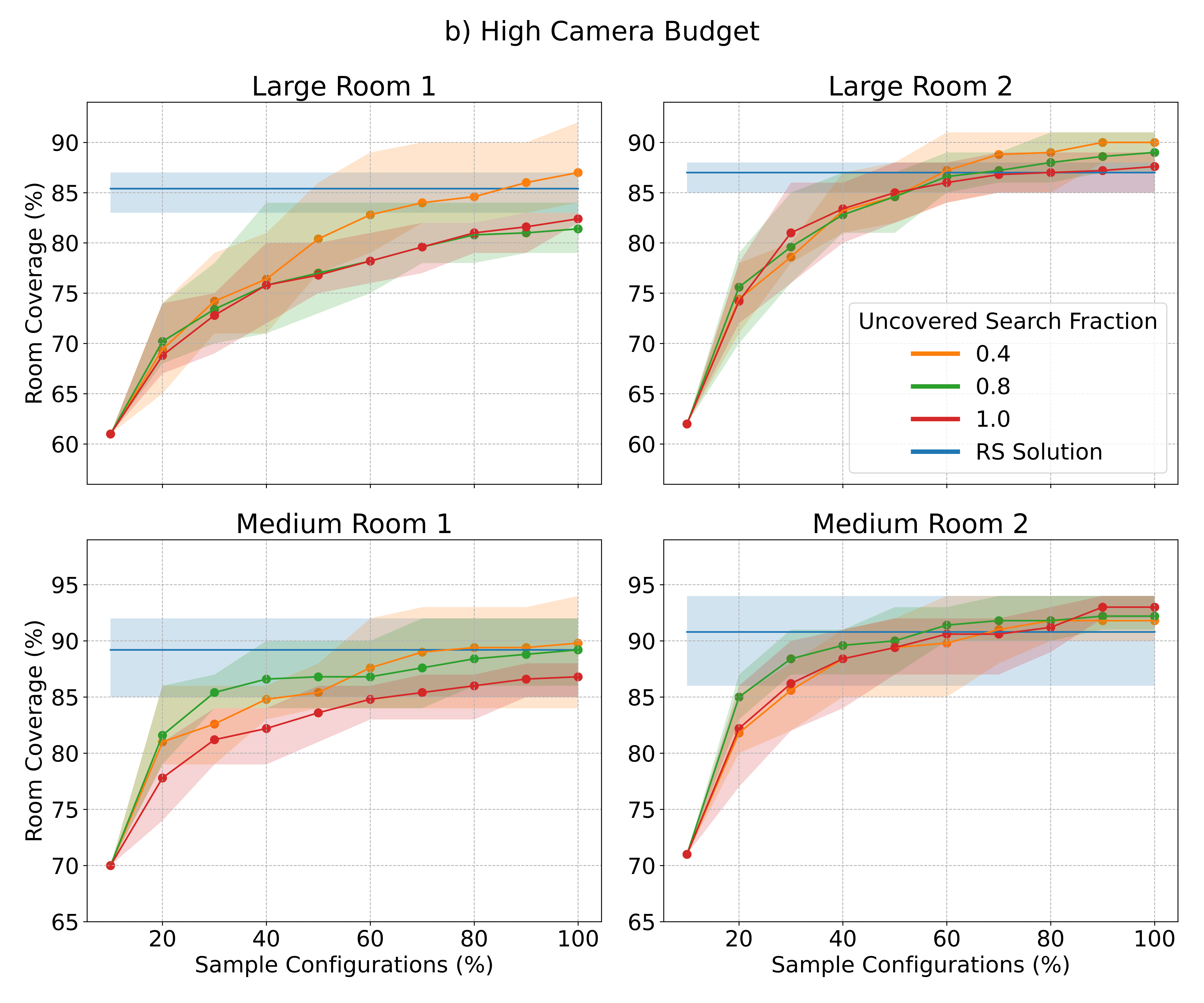} \label{figG5b}}
    \caption{This figure demonstrates that the TUS strategy is most effective with a low resource budget and larger open spaces, particularly when strict visibility requirements are deactivated. Part (a) shows the algorithm excelling with 50\% of the original resource budget in Large Room 2 and Medium Room 2, where unobstructed corridors are present. In contrast, in Large Room 1 and Medium Room 1, where views are more frequently obstructed, the algorithm performs similarly to the RS benchmark solution. Part (b) compares the algorithm's performance to the RS benchmark, showing no significant difference}
    \label{figG5:uncovered-search-fraction-sensitivity}
\end{figure}

Figure \yk{F}\ref{figG5a} illustrates that the algorithm excels when applied with a resource budget that is 50\% of what was previously allocated. This is especially evident in Large Room 2 and Medium Room 2, both characterized by walls that obstruct the same section of the room, leaving the other side unobstructed as a corridor. Conversely, in environments like Large Room 1 and Medium Room 1, where walls are positioned on alternate sides, obstructing views more frequently and resulting in no unobstructed common corridor, the algorithm performs comparably to the Mixed Integer Programming solution but does not significantly outperform it. This is when compared to Figure \yk{F}\ref{figG5b} where it is seen to be performing only as well as a random sampling algorithm. An uncovered search fraction of 0.4 performs better, marginally in some cases, than the RS benchmark solution in all cases, which is why we select it as the optimal hyperparameter value.

\textbf{Supervoxel size:} The algorithm relies on the concept of supervoxels to identify and target uncovered regions within an environment. Supervoxels are aggregations of smaller voxels, and their size plays a crucial role in determining the efficiency and effectiveness of the camera placement algorithm. In this sensitivity analysis, we evaluate the impact of different supervoxel sizes on the performance of the TUS strategy. Specifically, we examine four supervoxel sizes: 1, 3, 5, and 7, to determine the optimal granularity for balancing coverage accuracy and computational efficiency.

Based on the sensitivity analysis conducted, there does not seem to be a large difference between different supervoxel sizes (See Figure \ref{figG6:tus-supervoxel-sensitivity-fixed-low-budget}). The maximum difference in mean coverage is for medium-sized indoor space type 1, where the difference in the mean coverages between supervoxel sizes 3 (56\%)and 7 (59\%) is 3\%. The mean coverage is highest for sizes 7 (54.25\%) and sizes 5 (53.90\%), and the range between maximum and minimum coverage is lowest for sizes 5 (5.25\%) and size 7 (5.50\%). We normatively prescribe a supervoxel size of 5 (50\% of breadth).

\begin{figure}[h!]
    \centering
    \includegraphics[width=0.9\textwidth]{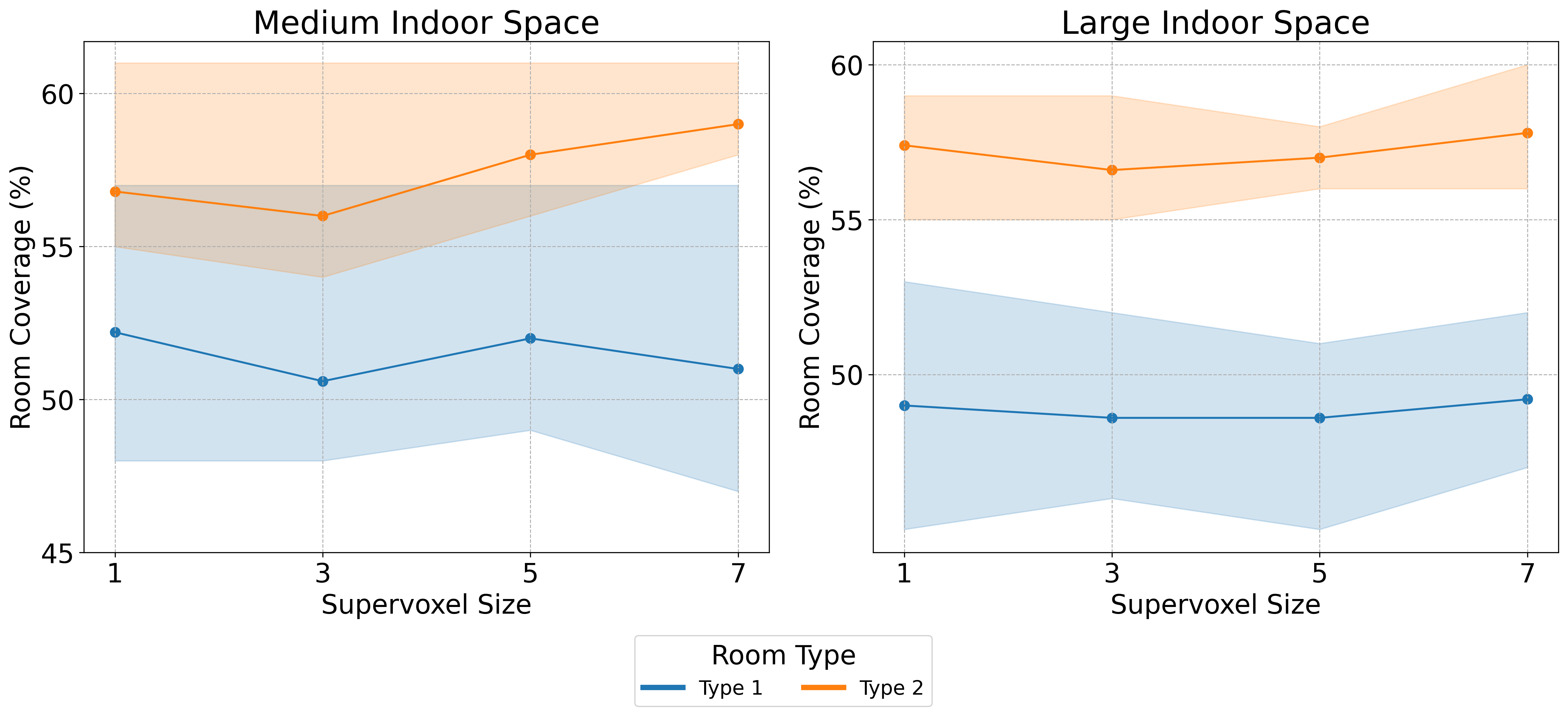}
    \caption{Sensitivity analysis of different supervoxel sizes (1, 3, 5, and 7) on the TUS strategy. The analysis shows minimal variation in mean coverage across sizes, with the largest difference being 3\% in medium-sized indoor space type 1. Supervoxel size 5 offers a good balance, with second to highest mean coverage (53.90\%) and the smallest range in coverage (5.25\%)}
    \label{figG6:tus-supervoxel-sensitivity-fixed-low-budget}
\end{figure}

\end{appendices}

\end{document}